\documentclass[a4paper]{amsart}

\RequirePackage{amsmath} \RequirePackage{amssymb}
\usepackage{amscd,latexsym,amsthm,amsfonts,amssymb,amsmath,amsxtra}
\usepackage[colorlinks=true,urlcolor=blue,citecolor=blue]{hyperref}
\usepackage{color}
\usepackage[all]{xy}
\usepackage[OT2,T1]{fontenc}
\usepackage{mathtools}
\usepackage{mathrsfs}

\DeclareSymbolFont{cyrletters}{OT2}{wncyr}{m}{n}
\DeclareMathSymbol{\Sha}{\mathalpha}{cyrletters}{"58}

\let\Re\undefined
\let\Im\undefined

\DeclareMathOperator{\Re}{Re}
\DeclareMathOperator{\Im}{Im}

\DeclareMathOperator{\Tr}{Tr}

\DeclareMathOperator{\supp}{supp}

\DeclareMathOperator{\GL}{GL}

\begin{document}
	
	\theoremstyle{plain}
\newtheorem{thm}{Theorem} \newtheorem{cor}[thm]{Corollary}
\newtheorem{thmy}{Theorem}
\renewcommand{\thethmy}{\Alph{thmy}}
\newenvironment{thmx}{\stepcounter{thm}\begin{thmy}}{\end{thmy}}
	\newtheorem{lemma}[thm]{Lemma}  \newtheorem{prop}[thm]{Proposition}
	\newtheorem{conj}[thm]{Conjecture}  \newtheorem{fact}[thm]{Fact}
	\newtheorem{claim}[thm]{Claim}
	\theoremstyle{definition}
	\newtheorem{defn}[thm]{Definition}
	\newtheorem{example}[thm]{Example}
	\newtheorem{exercise}[thm]{Exercise}
	\theoremstyle{remark}
	\newtheorem*{remark}{Remark}

	\newcommand{\BA}{{\mathbb {A}}} \newcommand{\BB}{{\mathbb {B}}}
	\newcommand{\BC}{{\mathbb {C}}} \newcommand{\BD}{{\mathbb {D}}}
	\newcommand{\BE}{{\mathbb {E}}} \newcommand{\BF}{{\mathbb {F}}}
	\newcommand{\BG}{{\mathbb {G}}} \newcommand{\BH}{{\mathbb {H}}}
	\newcommand{\BI}{{\mathbb {I}}} \newcommand{\BJ}{{\mathbb {J}}}
	\newcommand{\BK}{{\mathbb {K}}} \newcommand{\BL}{{\mathbb {L}}}
	\newcommand{\BM}{{\mathbb {M}}} \newcommand{\BN}{{\mathbb {N}}}
	\newcommand{\BO}{{\mathbb {O}}} \newcommand{\BP}{{\mathbb {P}}}
	\newcommand{\BQ}{{\mathbb {Q}}} \newcommand{\BR}{{\mathbb {R}}}
	\newcommand{\BS}{{\mathbb {S}}} \newcommand{\BT}{{\mathbb {T}}}
	\newcommand{\BU}{{\mathbb {U}}} \newcommand{\BV}{{\mathbb {V}}}
	\newcommand{\BW}{{\mathbb {W}}} \newcommand{\BX}{{\mathbb {X}}}
	\newcommand{\BY}{{\mathbb {Y}}} \newcommand{\BZ}{{\mathbb {Z}}}
	
	\newcommand{\CA}{{\mathcal {A}}} \newcommand{\CB}{{\mathcal {B}}}
	\newcommand{\CC}{{\mathcal {C}}} \renewcommand{\CD}{{\mathcal {D}}}
	\newcommand{\CE}{{\mathcal {E}}} \newcommand{\CF}{{\mathcal {F}}}
	\newcommand{\CG}{{\mathcal {G}}} \newcommand{\CH}{{\mathcal {H}}}
	\newcommand{\CI}{{\mathcal {I}}} \newcommand{\CJ}{{\mathcal {J}}}
	\newcommand{\CK}{{\mathcal {K}}} \newcommand{\CL}{{\mathcal {L}}}
	\newcommand{\CM}{{\mathcal {M}}} \newcommand{\CN}{{\mathcal {N}}}
	\newcommand{\CO}{{\mathcal {O}}} \newcommand{\CP}{{\mathcal {P}}}
	\newcommand{\CQ}{{\mathcal {Q}}} \newcommand{\CR}{{\mathcal {R}}}
	\newcommand{\CS}{{\mathcal {S}}} \newcommand{\CT}{{\mathcal {T}}}
	\newcommand{\CU}{{\mathcal {U}}} \newcommand{\CV}{{\mathcal {V}}}
	\newcommand{\CW}{{\mathcal {W}}} \newcommand{\CX}{{\mathcal {X}}}
	\newcommand{\CY}{{\mathcal {Y}}} \newcommand{\CZ}{{\mathcal {Z}}}
	
	\newcommand{\RA}{{\mathrm {A}}} \newcommand{\RB}{{\mathrm {B}}}
	\newcommand{\RC}{{\mathrm {C}}} \newcommand{\RD}{{\mathrm {D}}}
	\newcommand{\RE}{{\mathrm {E}}} \newcommand{\RF}{{\mathrm {F}}}
	\newcommand{\RG}{{\mathrm {G}}} \newcommand{\RH}{{\mathrm {H}}}
	\newcommand{\RI}{{\mathrm {I}}} \newcommand{\RJ}{{\mathrm {J}}}
	\newcommand{\RK}{{\mathrm {K}}} \newcommand{\RL}{{\mathrm {L}}}
	\newcommand{\RM}{{\mathrm {M}}} \newcommand{\RN}{{\mathrm {N}}}
	\newcommand{\RO}{{\mathrm {O}}} \newcommand{\RP}{{\mathrm {P}}}
	\newcommand{\RQ}{{\mathrm {Q}}} \newcommand{\RR}{{\mathrm {R}}}
	\newcommand{\RS}{{\mathrm {S}}} \newcommand{\RT}{{\mathrm {T}}}
	\newcommand{\RU}{{\mathrm {U}}} \newcommand{\RV}{{\mathrm {V}}}
	\newcommand{\RW}{{\mathrm {W}}} \newcommand{\RX}{{\mathrm {X}}}
	\newcommand{\RY}{{\mathrm {Y}}} \newcommand{\RZ}{{\mathrm {Z}}}
	
	\newcommand{\fa}{{\mathfrak{a}}} \newcommand{\fb}{{\mathfrak{b}}}
	\newcommand{\fc}{{\mathfrak{c}}} \newcommand{\fd}{{\mathfrak{d}}}
	\newcommand{\fe}{{\mathfrak{e}}} \newcommand{\ff}{{\mathfrak{f}}}
	\newcommand{\fg}{{\mathfrak{g}}} \newcommand{\fh}{{\mathfrak{h}}}
	\newcommand{\fii}{{\mathfrak{i}}} \newcommand{\fj}{{\mathfrak{j}}}
	\newcommand{\fk}{{\mathfrak{k}}} \newcommand{\fl}{{\mathfrak{l}}}
	\newcommand{\fm}{{\mathfrak{m}}} \newcommand{\fn}{{\mathfrak{n}}}
	\newcommand{\fo}{{\mathfrak{o}}} \newcommand{\fp}{{\mathfrak{p}}}
	\newcommand{\fq}{{\mathfrak{q}}} \newcommand{\fr}{{\mathfrak{r}}}
	\newcommand{\fs}{{\mathfrak{s}}} \newcommand{\ft}{{\mathfrak{t}}}
	\newcommand{\fu}{{\mathfrak{u}}} \newcommand{\fv}{{\mathfrak{v}}}
	\newcommand{\fw}{{\mathfrak{w}}} \newcommand{\fx}{{\mathfrak{x}}}
	\newcommand{\fy}{{\mathfrak{y}}} \newcommand{\fz}{{\mathfrak{z}}}
	\newcommand{\fA}{{\mathfrak{A}}} \newcommand{\fB}{{\mathfrak{B}}}
	\newcommand{\fC}{{\mathfrak{C}}} \newcommand{\fD}{{\mathfrak{D}}}
	\newcommand{\fE}{{\mathfrak{E}}} \newcommand{\fF}{{\mathfrak{F}}}
	\newcommand{\fG}{{\mathfrak{G}}} \newcommand{\fH}{{\mathfrak{H}}}
	\newcommand{\fI}{{\mathfrak{I}}} \newcommand{\fJ}{{\mathfrak{J}}}
	\newcommand{\fK}{{\mathfrak{K}}} \newcommand{\fL}{{\mathfrak{L}}}
	\newcommand{\fM}{{\mathfrak{M}}} \newcommand{\fN}{{\mathfrak{N}}}
	\newcommand{\fO}{{\mathfrak{O}}} \newcommand{\fP}{{\mathfrak{P}}}
	\newcommand{\fQ}{{\mathfrak{Q}}} \newcommand{\fR}{{\mathfrak{R}}}
	\newcommand{\fS}{{\mathfrak{S}}} \newcommand{\fT}{{\mathfrak{T}}}
	\newcommand{\fU}{{\mathfrak{U}}} \newcommand{\fV}{{\mathfrak{V}}}
	\newcommand{\fW}{{\mathfrak{W}}} \newcommand{\fX}{{\mathfrak{X}}}
	\newcommand{\fY}{{\mathfrak{Y}}} \newcommand{\fZ}{{\mathfrak{Z}}}
	\newcommand{\Int}{\operatorname{Int}}
	\newcommand{\Res}{\operatorname{Res}}
	\newcommand{\tr}{\operatorname{tr}}
	\newcommand{\Eis}{\operatorname{Eis}}
	\newcommand{\End}{\operatorname{End}}
	\newcommand{\K}{\operatorname{K}}
	\newcommand{\sgn}{\operatorname{sgn}}
	\newcommand{\Ker}{\operatorname{Ker}}
	\newcommand{\Ad}{\operatorname{Ad}}
	\newcommand{\Weyl}{\operatorname{Weyl}}
	\newcommand{\ram}{\operatorname{ram}}
	\newcommand{\Supp}{\operatorname{Supp}}
	\newcommand{\Cond}{\operatorname{Cond}}
	\newcommand{\diag}{\operatorname{diag}}
	\newcommand{\Ind}{\operatorname{Ind}}
	\newcommand{\E}{\operatorname{E}}
	\newcommand{\coker}{\operatorname{coker}}
	\newcommand{\Out}{\operatorname{Out}}
	\newcommand{\Geo}{\operatorname{Geo}}
	\newcommand{\Reg}{\operatorname{Reg}}
	\newcommand{\mix}{\operatorname{Mix}}
	\newcommand{\Sin}{\operatorname{Sing}}
	\newcommand{\V}{\operatorname{V}}
	\newcommand{\Span}{\operatorname{Span}}
	\title[Holomorphy of Adjoint L-functions for $\GL(n):$ $n\leq 4$]{Holomorphy of Adjoint L-functions for GL$(n)$: $n\leq 4$}%
	\author{Liyang Yang}

	\address{253-37 Caltech, Pasadena\\
		CA 91125, USA}
	\email{lyyang@caltech.edu}

	\begin{abstract}
We show entireness of complete adjoint L-functions associated to \textbf{any} unitary cuspidal representations of $\GL(3)$ or $\GL(4)$ over an arbitrary global field. Twisted cases are also investigated.
	\end{abstract}
	
	\date{\today}%
	\maketitle
	\tableofcontents
	\section{Introduction}
\subsection{A Folk Conjecture}
It is conjectured that for any $L$-series $L(s)$ in Selberg class, normalized to have functonal equation relating $s$ to $1-s,$ if $L(s)$ has a pole of order $r$ at $s=1,$ then $L(s)=\zeta(s)^r\cdot L_1(s),$ with $L_1(s)$ being holomorphic. This folk conjecture is wide open. For $L(s)$ attached to motivic $L$-functions, this is implied by Tate; while for $L(s)$ automorphic, Langlands Program implies it. In fact, every $L(s)$ of Selberg type is conjectured to be (isobaric) automorphic on $GL(n).$ In this paper, we consider one of the most fundamental cases, i.e., when $L(s)$ is a Rankin-Selberg convolution with a simple pole at $s=1.$
\medskip 
	
Let $F$ be a global field, and $\pi$ be any unitary cuspidal representation of $\GL(n,\mathbb{A}_F).$ Let $\tilde{\pi}$ be the contragredient of $\pi.$ Then one has the complete Rankin-Selberg $L$-function $\Lambda(s,\pi\times\tilde{\pi}),$ which has quite similar analytic properties as the complete Dedekind zeta function $\Lambda_{F}(s)$ associated to $F:$ they have simple poles at $s=0, 1;$ and are both holomorphic elsewhere. Hence the ratio
\begin{equation}\label{.}
\Lambda(s,\pi;\Ad):=\frac{\Lambda(s,\pi\times\tilde{\pi})}{\Lambda_{F}(s)}
\end{equation}
is meromorphic and is regular at $s=1.$ Conventionally $\Lambda(s,\pi;\Ad)$ is called complete adjoint $L$-function for $\pi$. One basic conjecture is 
\begin{conj}\label{P}
Let notation be as before. Then the complete adjoint $L$-function $\Lambda(s,\pi,\Ad)$ admits an analytic continuation to the whole complex plane.
\end{conj}

Note that the adjoint $L$-function defined by \eqref{.} is actually equal to the Langlands $L$-function associated to the adjoint action of the dual group $^{L}\GL(n;\mathbb{C})$ on the complex Lie algebra $\mathfrak{sl}(n,\mathbb{C})$ of $SL(n).$ Then according to Langlands Program, Conjecture \ref{P} should hold. On the other hand, to study Langlands functoriality conjecture, it is important to obtain analytic continuation of \textit{complete} $L$-functions, rather than their finite parts.

\medskip
	
The first breakthrough was made for classical holomorphic cusp forms by Shimura \cite{Shi75} and independently by Zagier \cite{Zag77}; Shimura's approach was generalized by Gelbart-Jacquet \cite{GJ78} to the adelic setting, while Zagier's method was further developed by Jacquet-Zagier \cite{JZ87} in terms of representation language. Furthermore, Jacquet and Zagier  proposed an auxiliary speculation that Conjecture \ref{P} might be a consequence of Dedekind Conjecture, which asserts that the ratio $\Lambda_E(s)/\Lambda_F(s)$ is entire for any number field extension $E/F.$ Note that $\Lambda_E(s)/\Lambda_F(s)$ can be written as a product of Artin $L$-functions, then Dedekind Conjecture is a consequence of Artin's holomorphy conjecture. Flicker \cite{Fli92} gave an argument suggesting that Dedekind Conjecture implies certain cases of Conjecture \ref{P} for general $n,$ under some local conditions on $\pi.$ In \cite{Yang19}, we proved the converse direction: Conjecture \ref{P} implies Dedekind Conjecture.
\medskip

Another approach to attack Conjecture \ref{P} for small rank $n$ (e.g., $n=3$) is based on an integral representation, which was pioneered by Ginzburg \cite{Gin91}, and a method of ruling out poles which was pioneered by Ginzburg-Jiang \cite{GJ00}. Typically this method helps continue \textit{partial} adjoint $L$-function to some right half plane. See \cite{HZ18} on $\GL(3)$ case for instance. 

\subsection{Statement of the Main Results}
In general, Conjecture \ref{P} remains wide open. It was not even known for general cuspidal representation of $\GL(3).$ In this paper, we show Conjecture \ref{P} holds for $n\leq 4$. In fact, we can handle the twist case as well: let
\begin{align*}
\Lambda(s,\pi,\Ad\otimes \tau):=\frac{\Lambda(s,\pi\otimes\tau\times \tilde{\pi})}{\Lambda(s,\tau)}
\end{align*}
be the twist adjoint $L$-function, where $\tau$ be a character on $F^{\times}\backslash\mathbb{A}_F^{\times}.$ We have
	\begin{thmx}\label{main1}
		Let notation be as before. Let $n\leq 4.$ Then the complete $L$-function $\Lambda(s,\pi,\Ad\otimes \tau)$ is entire, unless $\tau\neq1$ and $\pi\otimes\tau\simeq\pi,$ in which case $\Lambda(s,\pi,\Ad\otimes \tau)$ is meromorphic with only simple poles at $s=0, 1.$ In particular, Conjecture \ref{P} holds for any cuspidal representation $\pi$ when $n\leq 4.$
	\end{thmx}

Then a computation using local Langlands correspondence leads to
\begin{cor}\label{2cor}
	Let notation be as before. Let $n\leq 4.$ Then the finite $L$-function $L(s,\pi,\Ad\otimes \tau)=L(s,\pi\otimes\tau\times \tilde{\pi})/L(s,\tau)$ is entire, unless $\tau\neq1$ and $\pi\otimes\tau\simeq\pi,$ in which case $L(s,\pi,\Ad\otimes \tau)$ is meromorphic with only possible simple poles at $s=0, 1.$ In particular, the adjoint L-function $L(s,\pi,\Ad)=L(s,\pi\times \tilde{\pi})/\zeta_F(s)$ is entire.
\end{cor}

\begin{remark}
If $F$ is a function field, by using the cohomology of stacks of shtukas and the Arthur-Selberg trace formula, L. Lafforgue showed the Langlands correspondence of cuspidal representations $\pi$ of $\GL_n(\mathbb{A}_F)$ to Galois representations $\rho$ (see \cite{Laf02}). Then Theorem \ref{main1} follows from the identity $\Lambda(s,\pi,\Ad\otimes \tau)=\Lambda(s,\Ad\rho\otimes \tau)$ and analytic properties of  $\Lambda(s,\Ad\rho\otimes \tau),$ which is known well (see \cite{Wei74}). Hence we shall focus on the case that $F$ is a number field, where such a correspondence is not available yet. 
\end{remark}

\begin{remark}
If we admit Piatetski-Shapiro's strong conjecture on converse theorem (e.g. see Chap. 10 in \cite{Cog04}), Theorem \ref{main1} would imply that for any cuspidal representation $\pi$ of $\GL(n,\mathbb{A}_F),$ there exists an adjoint lifting $\Ad(\pi),$ which is a representation of $\GL(n^2-1,\mathbb{A}_F),$ in the sense of \cite{GJ78}. Hence, in principle, Theorem \ref{main1} will play a role in Langlands functoriality in this case.
\end{remark}

\subsection{Idea of Proofs and Plan of This Paper}
Our method is introduced in \cite{Yang19}, which is a generalization of \cite{JZ87} to higher rank case. Roughly speaking, we prove an identity of the form
\begin{align*}
\sum_{\pi}L(s,\pi,\Ad)\approx\sum_{[E:F]\leq n} \frac{\zeta_E(s)}{\zeta_F(s)}+\sum\text{L-S $L$-functions}+\sum\text{R-S $L$-functions},
\end{align*}
where L-S means Langlands-Shahidi and R-S refers to Rankin-Selberg for non-discrete representations, and sums above are typically infinite. We then show the convergence of sums and meromorphic continuation of the above mentioned $L$-functions and cancellation of their poles. In conjunction with certain spectral analysis and computing global root number, we eventually prove Theorem \ref{main1}.

\medskip

Let $G=GL(n),$ $n\leq 4.$ We consider a smooth function $\varphi:$ $G(\mathbb{A}_F) \rightarrow \mathbb{C}$ which is left and right $K$-finite, transforms by a unitary character $\omega$ of $Z_G\left(\mathbb{A}_F\right),$ and has compact support modulo $Z_G\left(\mathbb{A}_F\right).$ Then $\varphi$ defines an integral operator 
$$
R(\varphi)f(y)=\int_{Z_G(\mathbb{A}_F)\backslash G(\mathbb{A}_F)}\varphi(x)f(yx)dx,
$$ 
on the space $L^2\left(G(F)\backslash G(\mathbb{A}_F),\omega^{-1}\right)$ of functions on $G(F)\backslash G(\mathbb{A}_F)$ which transform under $Z_{G}(\mathbb{A}_F)$ by $w^{-1}$ and are square integrable on $G(F)Z_{G}(\mathbb{A}_F)\backslash G(\mathbb{A}_F).$ This operator can clearly be represented by the kernel function
$$
\K(x,y)=\sum_{\gamma\in Z_G(F)\backslash G(F)}\varphi(x^{-1}\gamma y).
$$

It is well known that $L^2\left(G(F)\backslash G(\mathbb{A}_F),\omega^{-1}\right)$ decomposes into the direct sums of the space $L_0^2\left(G(F)\backslash G(\mathbb{A}_F),\omega^{-1}\right)$ of cusp forms and spaces $L_{\Eis}^2\left(G(F)\backslash G(\mathbb{A}_F),\omega^{-1}\right)$ and $L_{\Res}^2\left(G(F)\backslash G(\mathbb{A}_F),\omega^{-1}\right)$ defined using Eisenstein series and residues of Eisenstein series respectively. Then $\K$ splits up as: $\K=\K_0+\K_{\Eis}+\K_{\Res}.$ Selberg trace formula gives an expression for the trace of the operator $R(\varphi)$ restricted to the discrete spectrum, and this is given by 
$$
\int_{G(F)Z(\mathbb{A}_F)\backslash G(\mathbb{A}_F)}\K_0(x,x)dx.
$$

We denote by $\mathcal{S}(\mathbb{A}_F^n)$ the space of Schwartz-Bruhat functions on the vector space $\mathbb{A}_F^n$ and by $\mathcal{S}_0(\mathbb{A}_F^n)$ the subspace spanned by products $\Phi=\prod_v\Phi_v$ whose components at real and complex places have the form
\begin{align*}
\Phi_v(x_v)=e^{-\pi \sum_{j=1}^nx_{v,j}^2}\cdot Q(x_{v,1},x_{v,2},\cdots,x_{v,n}),\ x_v=(x_{v,1},x_{v,2},\cdots,x_{v,n})\in F_v^n,
\end{align*}
where $F_v\simeq \mathbb{R},$ and $Q(x_{v,1},x_{v,2},\cdots,x_{v,n})\in \mathbb{C}[x_{v,1},x_{v,2},\cdots,x_{v,n}];$ and
\begin{align*}
\Phi_v(x_v)=e^{-2\pi \sum_{j=1}^nx_{v,j}\bar{x}_{v,j}}\cdot Q(x_{v,1},\bar{x}_{v,1},x_{v,2},\bar{x}_{v,2},\cdots,x_{v,n},\bar{x}_{v,n}),
\end{align*}
where $F_v\simeq \mathbb{C}$ and $Q(x_{v,1},\bar{x}_{v,1},x_{v,2},\bar{x}_{v,2},\cdots,x_{v,n},\bar{x}_{v,n})$ is a polynomial in the ring $\mathbb{C}[x_{v,1},\bar{x}_{v,1},x_{v,2},\bar{x}_{v,2},\cdots,x_{v,n},\bar{x}_{v,n}].$ 

Denote by $\Xi_F$ the set of characters on $F^{\times}\backslash\mathbb{A}_F^{\times}$ which are trivial on $\mathbb{R}_+^{\times}.$ Let $\Phi\in\mathcal{S}_0(\mathbb{A}_F^n)$ and $\tau\in\Xi_F.$ Let $\eta=(0,\cdots,0,1)\in F^n.$ Set
$$
f(x,\Phi,\tau;s)=\tau(\det x)|\det x|^s\int_{\mathbb{A}_F^{\times}}\Phi(\eta tx)\tau(t)^n|t|^{ns}d^{\times}t,
$$
which is a Tate integral (up to holomorphic factors) for $\Lambda(ns,x.\Phi,\tau^{n}).$ It converges absolutely uniformly in compact subsets of $\Re(s)>1/n.$ Since the mirabolic subgroup $P_0$ is the stabilizer of $\eta.$ Let $P=P_0Z_G$ be the full $(n-1,1)$ parabolic subgroup of $G,$ then $f(x,s)\in \Ind_{P(\mathbb{A}_F)}^{G(\mathbb{A}_F)}(\delta_P^{s-1/2}\tau^{-n}),$ where $\delta_P$ is the modulus character for the parabolic $P.$ Then we can define the Eisenstein series 
\begin{equation}\label{eis}
E_P(x,\Phi,\tau;s)=\sum_{\gamma\in P(F)\backslash G(F)}f(x,\Phi,\tau;s),
\end{equation}
which converges absolutely for $\Re(s)>1.$ Also, we define the integral:
\begin{equation}\label{X}
I_0^{\varphi}(s,\tau)=\int_{G(F)Z(\mathbb{A}_F)\backslash G(\mathbb{A}_F)}\K_0(x,x)E_P(x,\Phi;s)dx.
\end{equation}
If there is no confusion in the context, we will alway write $I(s,\tau)$ (resp. $f(x,s)$) instead of $I^{\varphi}(s,\tau)$ (resp. $f(x,\Phi,\tau;s)$) for simplicity.
\medskip 

In \cite{Yang19} (see Theorem A), we proved the expansion of $I_0(s,\tau)=I_0^{\varphi}(s,\tau):$
\begin{equation}\label{a}
I_0(s,\tau)=I_{\Geo,\Reg}(s,\tau)+I_{\infty,\Reg}(s,\tau)+I_{\Sin}(s,\tau)+ \sum_{\chi}\int_{(i\mathbb{R})^{n-1}}I_{\chi}(s,\tau,\lambda)d\lambda;
\end{equation}
and investigated analytic behaviors of $I_{\Geo,\Reg}(s,\tau),$ $I_{\infty,\Reg}(s,\tau)$ and $I_{\chi}(s,\tau,\lambda).$ Nevertheless, we still need to study the delicate geometric term $I_{\Sin}(s,\tau)$ and prove the sum over $\chi$ in the spectral side admits a meromorphic continuation to some domain containing $\Re(s)\geq 1/2.$ This is the goal of this paper. As a consequence, we will deduce Theorem \ref{main1}.

\medskip 

By Proposition in Section 3.3 of \cite{JZ87}, Theorem \ref{main1} will follow if $I_0(s,\tau)\cdot \Lambda(s,\tau)^{-1},$ $\Re(s)>1,$ admits a holomorphic continuation outside $s=1$. On the other hand, by Theorem D and Theorem E in \cite{Yang19}, we see $\Lambda(s,\tau)^{-1}\cdot I_{\Geo,\Reg}(s,\tau)$ and $\Lambda(s,\tau)^{-1}\cdot I_{\infty,\Reg}(s,\tau)$ admits a meromorphic continuation to the half plane $\Re(s)>1/3,$ holomorphic when $s\notin\{ 1, 1/2\},$ and has a possible simple pole at $s=1/2$ if $\tau^2=1,$ namely, $\tau$ is either trivial or has order 2. Therefore, according to decomposition \eqref{a}, it suffices to show $\Lambda(s,\tau)^{-1}\cdot I_{\Sin}(s,\tau)$ and $\Lambda(s,\tau)^{-1}\cdot\sum_{\chi}\int_{(i\mathbb{R})^{n-1}}I_{\chi}(s,\tau,\lambda)d\lambda$ admit meromorphic continuation to the whole $s$-plane, and the poles of all these mentioned functions cancel. 
\medskip

In Section \ref{geo}, we study $I_{\Sin}(s,\tau),$ proving Theorem \ref{Y} for $G=\GL(3)$ and $\GL(4)$ separately. In fact, if we further decompose the distributions by Bruhat decomposition, it is easy to see that many cells give no contribution. However, there are some cells such that the corresponding distributions diverge. Such problematic cells will be gathered together and the distribution $I_{\infty}^{\mix}(s,\tau)$ from (finite) linear combination of these cells will be shown vanishing via Poisson summation and Fourier expansion of certain orbital integrals (see Proposition \ref{mix}). Moreover, we obtain analytic behaviors of surviving (convergent) parts, they either contribute products of degree 1 $L$-functions, or may be reduced to Jacquet-Zagier's work \cite{JZ87} on $\GL(2)$ (e.g., see Proposition \ref{12}, Proposition \ref{000} and Proposition \ref{0004}).

\medskip 

In Section \ref{7.3}, we study $I^{(1)}_{\infty}(s,\tau)=\sum_{\chi}\int_{(i\mathbb{R})^{n-1}}I_{\chi}(s,\tau,\lambda)d\lambda,$ obtaining meromorphic continuation of it. When $\tau=1,$ the residue of $I^{(1)}_{\infty}(s,\tau)$ at $s=1$ should give the weighted character distribution in Arthur-Selberg trace formula. We call $I^{(1)}_{\infty}(s,\tau)$ the generic character distribution for $G.$ In \cite{Yang19}, we obtained meromorphic continuation of $\int_{(i\mathbb{R})^{n-1}}I_{\chi}(s,\tau,\lambda)d\lambda,$ which is related to Rankin-Selberg convolution for \textit{non-cuspidal} representations. Thus, we can write $I^{(1)}_{\infty}(s,\tau)$ as an infinite sum of meromorphic functions, yet each individual may have poles. Then the next step is to analyze these possible poles and show that they do cancel with each other (see Theorem \ref{78}). However, by this approach we can only rule out all potential poles of $I_0(s,\tau) \cdot \Lambda(s,\tau)^{-1}$ except for a possible simple pole at $s=1/2$ when $\tau$ is quadratic. 

\medskip

In Section \ref{8.}, we will prove Theorem \ref{main1}. In fact, Theorem \ref{78} in Section \ref{7.3} will eventually imply the that $\Lambda(s,\pi,\Ad\otimes\tau)$ admits a meromorphic continuation with at most a simple pole at $s=1/2.$ To remedy it, we prove the root number of $\Lambda(s,\pi,\Ad\otimes\tau)$ is always 1 in this case (see Proposition \ref{76lem}). This would exclude the possibility of existence of a simple pole at $s=1/2.$ Now Theorem \ref{main1} follows.
\medskip

\textbf{Acknowledgements}
I am very grateful to my advisor Dinakar Ramakrishnan for instructive discussions and helpful comments. I would like to thank Ashay Burungale, Li Cai, Herv\'e Jacquet, Dihua Jiang, Simon Marshall, Kimball Martin, Yiannis Sakellaridis, Chen Wan and Xinwen Zhu for their precise comments and useful suggestions. Part of this paper was revised during my visit to \'{E}cole polytechnique f\'ed\'erale de Lausanne in Switzerland and I would like to thank their hospitality.

\section{Contributions from Geometric Side}\label{geo}
\subsection{Basic Notation and Singular Orbital Distributions}\label{2.1}
Fix an integer $n\geq2.$ The maximal unipotent subgroup of $G(\mathbb{A}_F),$ denoted by $N(\mathbb{A}_F),$ is defined to be the set of all $n\times n$ upper triangular matrices in $G(\mathbb{A}_F)$ with ones on the diagonal and arbitrary entries above the diagonal. Let $\psi_{F/\mathbb{Q}}(\cdot)=e^{2\pi i\Tr_{F/\mathbb{Q}}(\cdot)}$ be the standard additive character, then we can define a character $\theta:$ $N(\mathbb{A}_F)\rightarrow \mathbb{C}^{\times}$ by
\begin{align*}
\theta(u)=\prod_{i=1}^{n-1}\psi_{F/\mathbb{Q}}\left(u_{i,i+1}\right),\quad \forall\ u=(u_{i,j})_{n\times n}\in N(\mathbb{A}_F).
\end{align*}  

Let $R_k$ be the standard parabolic subgroup of $G$ of type $(k,n-k)$ consisting of matrices whose $\GL(n-k)$ part is upper triangular unipotent.  Let $V_k$ be the unipotent subgroup of the standard parabolic subgroup of type $(k-1,1,n-k).$ Denote by $V_k'=\diag(I_k, N_{n-k}).$ For an algebraic group $H$ over $F,$ we will use the notation $[H]$ to refer $H(F)\backslash H(\mathbb{A}_F)$ for simplicity.
\medskip

Let $\widetilde{V}_k$ be the unipotent subgroup of the standard parabolic subgroup of $\GL(n)$ of type $(k,n-k).$ Let $\widetilde{V}_k'=\widetilde{V}_k\backslash \widetilde{V}_{k-1}.$ Let $N_k=\diag(I_{k-1},N_{2},I_{n-k-1}),$ the unipotent subgroup corresponding to the root $w_k,$ $1\leq k\leq n-1.$ For an algebraic group $H,$ sometimes we will write $H$ for its $F$-points $H(F)$ for simplicity. Also, for sets $A$ and $B,$ denote by $A^B$ the set $\{b^{-1}ab:\ a\in A,\ b\in B\}.$
\medskip 
\subsubsection{Fourier Expansion of Mirabolic Orbital Functions}
Let $h$ be a Schwartz function on $G(\mathbb{A}_F).$ Let $\mathcal{S}$ be a subset of $G(F).$ Let 
\begin{align*}
O_h(x,y)=\sum_{\gamma\in \mathcal{S}^{P_0(F)}}h(x^{-1}\gamma y),
\end{align*}
where $\mathcal{S}^{P_0(F)}$ is the set consisting of $p^{-1}\gamma p,$ for all $\gamma\in\mathcal{S}$ and $p\in P_0(F).$ We call $O_h(x,y)$ a mirabolic orbital function on $G(\mathbb{A}_F)\times G(\mathbb{A}_F)$ associated to $h$ and $\mathcal{S}.$ 
\begin{prop}\label{Fourier}
	Let notation be as before. Then
	\begin{equation}\label{fourier}
	O_h(x,y)=\sum_{k=1}^n\sum_{\delta_k\in R_{k-1}(F)\backslash R_{n-1}(F)}\int_{[V_k']}\int_{[V_k]}O_h(uu'\delta_kx,\delta_ky)\theta(u')du'du,
	\end{equation}
	if the right hand side converges absolutely and locally uniformly.
\end{prop}

Proposition \ref{Fourier} will play a role in the crucial Proposition \ref{mix} (see Sec. \ref{2.3}). Since the proof of \eqref{fourier} is essentially the same as Prop. 17 in \cite{Yang19}, we omit the proof here.

\subsubsection{The Singular Orbital Distribution }
Denote by $\mathfrak{S}=\bigcup _{k=1}^{n-1}\left(Z_G(F)\backslash Q_k(F)\right)^{P_0(F)}.$ Following the approach in \cite{Yang19}, we will treat $I(s)$ via the decomposition
\begin{equation}\label{1}
\K_0(x,y)=\K_{\Geo,\Reg}(x,y)+\K_{\Geo,\Sin}(x,y)-\K_{\infty}(x,y),
\end{equation}
where $\mathcal{C}$ runs through all nontrivial conjugacy classes in $G(F)/Z_G(F)$ and 
\begin{align*}
& \K_{\Geo, \Reg}(x,y)=\sum_{\mathcal{C}} \sum_{\substack{\gamma\in \mathcal{C}-\mathfrak{G}}}\varphi(x^{-1}\gamma y),\quad \K_{\Geo,\Sin}(x,y)=\sum_{\gamma\in \mathfrak{G}}\varphi(x^{-1}\gamma y), 
\end{align*}
and for simplicity we denote by $\K_{\infty}(x,y)=\K_{\Eis}(x,y)+\K_{\Res}(x,y).$ Then substitute Fourier expansion of $\K_{\infty}(x,y)$ (e.g. Prop. 17 in loc. cit.) into \eqref{1} to obtain 
\begin{equation}\label{2}
\K_0(x,y)=\K_{\Geo,\Reg}(x,y)+\K_{\Geo,\Sin}(x,y)-\sum_{k=1}^n\K_{\infty}^{(k)}(x,y),
\end{equation}
where the sum over $k$ indicates the Fourier expansion of $\K_{\infty}(x,y):$
\begin{align*}
\K_{\infty}^{(k)}(x,y)=\sum_{\delta_k\in R_{k-1}(F)\backslash R_{n-1}(F)}\int_{[V_k']}\int_{[V_k]}\K_{\infty}(uu'\delta_kx,\delta_ky)du\theta(u')du'
\end{align*}

 We can further decompose $\K_{\infty}^{(n)}(x,y)=\K_{\infty,\Reg}(x,y)+\K_{\infty,\Sin}(x,y),$ where
\begin{align*}
& \K_{\infty, \Reg}(x,y)=\int_{[N_P]}\K_{\Geo,\Reg}(ux,y)du,\ \K_{\infty,\Sin}(x,y)=\int_{[N_P]}\K_{\Geo,\Sin}(ux,y)du.
\end{align*}
\medskip

Let $X_G=Z_G(\mathbb{A}_F)P_0(F)\backslash G(\mathbb{A}_F).$ By the above expansion \eqref{2}, we then obtain 
\begin{align*}
I_0(s,\tau)=I_{\Geo,\Reg}(s,\tau)-I_{\infty,\Reg}(s,\tau)+I_{\Sin}(s,\tau)-I_{\infty}^{(1)}(s,\tau),
\end{align*}
where $I_{\Geo,\Reg}(s,\tau),$ $I_{\infty,\Reg}(s,\tau)$ and $I_{\infty}^{(1)}(s,\tau)$ are defined by integrating the kernel functions $\K_{\Geo,\Reg}(x,x),$ $\K_{\infty,\Reg}(x,x)$ and $\K_{\infty}^{(1)}(x,x)$ against $f(x,s)$ over $X_G,$ respectively; and the distribution $I_{\Sin}(s,\tau)$ is defined by 
\begin{align*}
I_{\Sin}(s,\tau)=\int_{X_G}\bigg[\K_{\Geo,\Sin}(x,x)-\K_{\infty,\Sin}(x,x)-\sum_{k=2}^{n-1}\K_{\infty}^{(k)}(x,x)\bigg]\cdot f(x,s)dx.
\end{align*}

In fact, the integral with respect to each term in the bracket will diverge, while the linear combination $\K_{\Geo,\Sin}(x,x)-\K_{\infty,\Sin}(x,x)-\sum_{k=2}^{n-1}\K_{\infty}^{(k)}(x,x)$ will make the divergent parts cancel. Hence we will call $I_{\Sin}(s,\tau)$ singular orbital distribution for $G.$  

In loc. cit. we investigate analytic behaviors of $I_{\Geo,\Reg}(s,\tau),$ $I_{\infty,\Reg}(s,\tau)$ and (partially) $I_{\infty}^{(1)}(s,\tau),$ circumventing $I_{\Sin}(s,\tau)$ by a choice of test functions. In this section, we shall use general test functions to prove some basic properties of $I_{\Sin}(s,\tau)$, and conclude the following result:

\begin{thmx}\label{Y}
Let notation be as before. Let $n\leq 4.$ Then $I_{\Sin}(s,\tau)$ admits a meromorphic continuation to the whole $s$-plane. Moreover, the function $I_{\Sin}(s,\tau)/\Lambda(s,\tau)$ is holomorphic in the right half plane $\Re(s)>0$ if $s\notin\{1, 1/2, 1/3, \cdots, 1/n\},$ and $I_{\Sin}(s,\tau)\cdot \Lambda(s,\tau)^{-1}$ may have at most simple poles when $s\in \{1/2, 1/3, \cdots, 1/n\}.$
\end{thmx}
\begin{remark}
To deal with general $\GL(n),$ one of the initial steps is to classify the relevant orbital integrals of Fourier type for all $2\leq k\leq n.$ The classification of $k=1$ case, i.e., Kloosterman integrals, can also be found in \cite{BFG84} or \cite{Jac03}. For lower rank, e.g., $n\leq 4,$ we can do this by brute force.
\end{remark}

The proof of Theorem \ref{Y} follows readily by gathering results in Sec. \ref{2.2} and Sec. \ref{2.3} below.

 \subsection{Singular Expansion for $\GL(3)$}\label{2.2}
Let notation be as before. Recall that 
\begin{equation}\label{0023}
I_{\infty}^{\Sin}(s,\tau)=\int_{Z_G(\mathbb{A}_F)P_0(F)\backslash G(\mathbb{A}_F)}\bigg[\K_{\Sin}(x,x)-\K_{\infty}^{(2)}(x,x)\bigg]\cdot f(x,s)dx,
\end{equation}

To prove Theorem \ref{Y}, we need to investigate $ \K_{\Sin}(x,x)=\K_{\Geo,\Sin}(x,x)-\K_{\infty,\Sin}^{(3)}(x,x)$ and $\K_{\infty}^{(2)}(x,x).$ From the definition of $\K_{\Geo,\Sin}(x,x)$ and $\K_{\infty,\Sin}^{(3)}(x,x),$ we need a description of $\mathfrak{G}:$
\begin{lemma}\label{21}
	Let notation be as before. Then we have 
	\begin{equation}\label{001}
	\mathfrak{G}=P_0(F)\bigsqcup (B_0w_2N_2)^{B_0\backslash P_0}.
	\end{equation}
	Moreover, any $\gamma \in \mathfrak{G}- P_0(F)$ can be written uniquely as 
	\begin{equation}\label{004}
	\gamma=p^{-1}bw_{2}up,
	\end{equation}
	where $p\in B_0(F)\backslash P_0(F),$ $b\in B_0(F),$ and $u\in N_2(F).$
\end{lemma} 
Since Lemma \ref{21} is a straightforward computation using Bruhat decomposition, we omit the proof. However, we will proved a detailed proof to Lemma \ref{22} (see Sec. \ref{2.3}), which is a higher rank version of Lemma \ref{21}.

\medskip 
Let $\mathcal{A}_1(F)=(B_0w_2N_2)^{B_0\backslash P_0},$ and $\mathcal{A}_2(F)=P_0.$ For $1\leq i\leq 2,$ we denote by 
\begin{align*}
&\K_{\Geo,\Sin,i}(x,y)=\sum_{\gamma\in \mathcal{A}_i(F)}\varphi(x^{-1}\gamma y);\\
&\K_{\infty,\Sin,i}^{(3)}(x,y)=\int_{N_P(F)\backslash N_P(\mathbb{A}_F)}\sum_{\gamma\in \mathcal{A}_i(F)}\varphi(x^{-1}u^{-1}\gamma y)du.
\end{align*}

Then, $\K_{\Sin}(x,y)=\K_{\Sin,1}(x,y)+\K_{\Sin,2}(x,y),$ where $\K_{\Sin,i}(x,y)=\K_{\Geo,\Sin,i}(x,y)-\K_{\infty,\Sin,i}^{(3)}(x,y),$ $1\leq i\leq 2.$ On the other hand, by Bruhat decomposition, $\K_{\infty}^{(2)}(x,x)=\sum_{i=1}^5\K_{\infty,i}^{(2)}(x,x),$ where
\begin{align*}
\K_{\infty,i}^{(2)}(x,x)=\sum_{\delta\in R_1(F)\backslash P_0(F)}\int_{[N_1]}\int_{[N_P]}\sum_{\gamma\in\mathcal{B}_i(F)}\varphi(x^{-1}\delta^{-1}u^{-1}v^{-1}\gamma x)\theta(u)dudv,
\end{align*}
with $\mathcal{B}_1(F)=B_0(F)w_2N_2(F),$ $\mathcal{B}_2(F)=P_0(F),$ $\mathcal{B}_3(F)=B_0(F)w_1w_2w_1N(F),$ $\mathcal{B}_4(F)=B_0(F)w_1w_2N_{12}(F),$ and $\mathcal{B}_5(F)=B_0(F)w_2w_1N_{21}(F).$

Denote by $\K_{\Sin,2}(x,x)=\K_{\Geo,\Sin,2}(x,y)-\K_{\infty,\Sin,2}^{(3)}(x,y)-\K_{\infty,2}^{(2)}(x,x).$ Then one can apply Proposition \ref{Fourier} to $\K_{\Geo,\Sin,2}(x,y),$ to deduce
\begin{equation}\label{0024}
\K_{\Sin,2}(x,x)=\sum_{\delta\in N(F)\backslash P_0(F)}\int_{N(F)\backslash N(\mathbb{A}_F)}\sum_{\gamma\in P_0(F)}\varphi(x^{-1}u^{-1}\gamma x)\theta(u)du.
\end{equation}

Let $I_{\Sin,2}(s)=\int_{Z_G(\mathbb{A}_F)P_0(F)\backslash G(\mathbb{A}_F)}\K_{\Sin,2}(x,x)\cdot f(x,s)dx.$ Then by \eqref{0024},
\begin{align*}
I_{\Sin,2}(s)=&\int_{Z_G(\mathbb{A}_F)N(F)\backslash G(\mathbb{A}_F)}\int_{N(F)\backslash N(\mathbb{A}_F)}\sum_{\gamma\in P_0(F)}\varphi(x^{-1}u^{-1}\gamma x)\theta(u)duf(x,s)dx.
\end{align*}

Using Bruhat decomposition to write $P_0(F)=B_0(F)\sqcup B_0(F)w_1N_1(F),$ then 
\begin{equation}\label{0025}
I_{\Sin,2}(s)=\int_{Z_G(\mathbb{A}_F)N(F)\backslash G(\mathbb{A}_F)}\int_{[N]}\sum_{\gamma\in B_0(F)}\varphi(x^{-1}u^{-1}\gamma x)\theta(u)duf(x,s)dx,
\end{equation}
since the contribution from $\gamma\in B_0(F)w_1N_1(F)$ vanishes. Now we can apply Iwasawa decomposition $G(\mathbb{A}_F)=N(\mathbb{A}_F)T(\mathbb{A}_F)K$ into \eqref{0025} to obtain
\begin{align*}
I_{\Sin,2}(s)=&\int_{\mathbb{A}_F^{\times}}\int_{\mathbb{A}_F^{\times}}\int_{K}\sum_{t_1,t_2\in F^{\times}}\int_{\mathbb{A}_F}\int_{\mathbb{A}_F}\int_{\mathbb{A}_F}\varphi\left(k^{-1}\begin{pmatrix}
t_1& &\\
&t_2&\\
&&1
\end{pmatrix}\begin{pmatrix}
1&c&f\\
&1&e\\
&&1
\end{pmatrix}k\right)\\
&\quad\theta(ac)\theta(be)\tau^2(a)\tau(b)|a|^{2s}|b|^s\cdot f(k,s)dcdfdedkd^{\times}ad^{\times}b.
\end{align*}

Then by Tate's thesis, we conclude that $I_{\Sin,2}(s)$ is an integral representation for  $\Lambda(s,\tau)\Lambda(2s,\tau^2)\Lambda(3s,\tau^3).$ Hence $I_{\Sin,2}(s)$ converges absolutely when $\Re(s)>1,$ and it has the analytic property 
\begin{equation}\label{0026}
I_{\Sin,2}(s)\sim \Lambda(s,\tau)\Lambda(2s,\tau^2)\Lambda(3s,\tau^3).
\end{equation}

As a consequence, $I_{\Sin,2}(s)$ admits a meromorphic continuation to $s$-plane, with possible poles (which are simple if exist) at $s\in\{1, 1/2, 1/3\}.$ 

\bigskip 

By a simple changing of variables, we see 
\begin{equation}\label{0027}
\int_{N(F)\backslash N(\mathbb{A}_F)}\K_{\infty,4}^{(2)}(nx,nx)dn=\int_{N(F)\backslash N(\mathbb{A}_F)}\K_{\infty,5}^{(2)}(nx,nx)dn=0.
\end{equation}

So we have to deal with the rest contribution from $\K_{\infty}^{(2)}(x,x),$ namely,
\begin{equation}\label{0028}
I_{\infty,i}^{(2)}(s)=\int_{Z_G(\mathbb{A}_F)P_0(F)\backslash G(\mathbb{A}_F)}\K_{\infty,i}^{(2)}(x,x)\cdot f(x,s)dx,
\end{equation}
where $i\in \{1,3\}.$ We compute $I_{\infty,3}^{(2)}(s)$ first:
\begin{align*}
I_{\infty,3}^{(2)}(s)=\int_{Z_G(\mathbb{A}_F)R_1(F)\backslash G(\mathbb{A}_F)}\int_{[N_1]}\int_{[N_P]}\sum_{\gamma\in \mathcal{B}_3(F)}\varphi(x^{-1}v^{-1}u^{-1}\gamma x)\theta(u)dudvf(x,s)dx.
\end{align*}

Let $\widetilde{w}=w_1w_2w_1.$ Again, apply Iwasawa decomposition to see 
\begin{align*}
I_{\infty,3}^{(2)}(s)=&\int_{(\mathbb{A}_F^{\times})^2}\int_{K}\int_{\mathbb{A}_F^6}\sum_{t}\varphi\left(k^{-1}\begin{pmatrix}
1&a&b\\
&1&c\\
&&1
\end{pmatrix}\widetilde{w}\begin{pmatrix}
tt_1^2t_2^2& &\\
&t_1t_2&\\
&&1
\end{pmatrix}\begin{pmatrix}
1&e&f\\
&1&g\\
&&1
\end{pmatrix}k\right)\\
&\quad \theta(t_1g)\theta(t_1c)|t_1t_2|^{s+2}|t_1|^s\tau(t_1)\tau(t_1t_2)w(t_1t_2)d^{\times}t_1d^{\times}t_2f(k,s)dk\\
=&\int_{(\mathbb{A}_F^{\times})^2}\int_{K}\int_{\mathbb{A}_F^6}\sum_{t\in F^{\times}}\varphi\left(k^{-1}\begin{pmatrix}
1&a&b\\
&1&c\\
&&1
\end{pmatrix}\widetilde{w}\begin{pmatrix}
tt_2^2& &\\
&t_2&\\
&&1
\end{pmatrix}\begin{pmatrix}
1&e&f\\
&1&g\\
&&1
\end{pmatrix}k\right)\\
&\quad \theta(t_1g)\theta(t_1c)|t_2|^{s+2}|t_1|^{s}\tau(t_1)\tau(t_2)w(t_2)d^{\times}t_1d^{\times}t_2f(k,s)dk.
\end{align*}
Then by Tate's thesis and intertwining operator theory, we conclude that $I_{\infty,3}^{(2)}(s)$ is an integral representation for  $\Lambda(s,\tau)\Lambda(s+1,\tau)\Lambda(3s,\tau^3)/\Lambda(s+2,\tau).$ Hence $I_{\infty,3}^{(2)}(s)$ converges absolutely when $\Re(s)>1,$ and it has the analytic property 
\begin{equation}\label{0029}
I_{\infty,3}^{(2)}(s)\sim \frac{\Lambda(s,\tau)\Lambda(s+1,\tau)\Lambda(3s,\tau^3)}{\Lambda(s+2,\tau)}.
\end{equation}

\bigskip

We claim the term $I_{\infty,1}^{(2)}(s)$ will be canceled by contribution from some part of $\K_{\Sin,1}(x,x).$ This will be presented in the following computation. Denote by 
\begin{align*}
\K_{\Sin,1}^{(2)}(x;y)=\sum_{\gamma\in B_0w_2N_2}\varphi(x^{-1}y^{-1}\gamma x)-\int_{[N_P]}\sum_{\gamma\in B_0w_2N_2}\varphi(x^{-1}u^{-1}y^{-1}\gamma x)du.
\end{align*}

Then $\K_{\Sin,1}^{(2)}(x;y)$ is well defined function with respect to $y$ on $B_0(F)\backslash B_0(\mathbb{A}_F)\subset P^2_0(F)\backslash G^2(\mathbb{A}_F),$ where $G^2=\diag(\GL(2),1),$ and $P^2$ is the (only) standard parabolic subgroup of $G^2;$ and $P_0^2$ is the mirabolic subgroup of $P^2.$ Hence, we can apply Fourier expansion to $\K_{\Sin,1}^{(2)}(x;y)$ and set $y=I_3$ to obtain:
\begin{equation}\label{0019}
\K_{\Sin,1}^{(2)}(x;I_3)= \K_{\Sin,1}^{(2,1)}(x,x)-\K_{\Sin,1}^{(2,2)}(x,x)+ \K_{\Sin,1}^{(1,1)}(x,x),
\end{equation}
where 
\begin{align*}
\K_{\Sin,1}^{(2,1)}(x,x)=&\int_{N_{1}(F)\backslash N_{1}(\mathbb{A}_F)}\sum_{\gamma\in B_0w_2N_2}\varphi(x^{-1}u^{-1}\gamma x)du;\\
\K_{\Sin,1}^{(1,1)}(x,x)=&\sum_{\delta\in N(F)\backslash R_1(F)}\int_{N_{1}(F)\backslash N_{1}(\mathbb{A}_F)}\sum_{\gamma\in B_0w_2N_2}\varphi(x^{-1}\delta^{-1}v^{-1}\gamma x)\theta(v)dv;\\
\K_{\Sin,1}^{(2,2)}(x,x)=&\int_{N_P(F)\backslash N_P(\mathbb{A}_F)}\int_{N_{1}(F)\backslash N_{1}(\mathbb{A}_F)}\sum_{\gamma\in B_0w_2N_2}\varphi(x^{-1}u^{-1}v^{-1}\gamma x)dvdu.
\end{align*}

To deal with $ \K_{\Sin,1}^{(2,1)}(x),$ we will apply Poisson summation: write $\gamma=b_0w_2n_2\in B_0(F)w_2N_2(F),$ where $b_0\in B_0(F)$ and $n_2\in N_2(F).$ Noting $N_2(F)\simeq F,$ we can apply Poisson summation to see $\K_{\Sin,1}^{(2,1)}(x,x)=\K_{\Sin,1,0}^{(2,1)}(x,x)+\K_{\Sin,1,\neq 0}^{(2,1)}(x,x),$ where the constant term $\K_{\Sin,1,0}^{(2,1)}(x,x)$ is equal to 
\begin{align*}
\int_{N_{1}(F)\backslash N_{1}(\mathbb{A}_F)}\int_{N_2(\mathbb{A}_F)}\sum_{b_0\in B_0(F)}\varphi(x^{-1}u^{-1}b_0w_2vx)dudv;
\end{align*}
and $\K_{\Sin,1,\neq 0}^{(2,1)}(x,x),$ the contribution from non-constant terms, is equal to 
\begin{align*}
\sum_{\beta\in F^{\times}}\int_{N_{1}(F)\backslash N_{1}(\mathbb{A}_F)}\int_{N_2(\mathbb{A}_F)}\sum_{b_0\in Z_G(F)\backslash B(F)}\varphi(x^{-1}u^{-1}b_0w_2vx)\theta(\beta v)dudv.
\end{align*}

By a change of variable, we see $\K_{\Sin,1,\neq 0}^{(2,1)}(x,x)$ can be rewritten as 
\begin{align*}
\sum_{\lambda\in R_1(F)\backslash B_0(F)}\int_{N_{1}(F)\backslash N_{1}(\mathbb{A}_F)}\int_{N_2(\mathbb{A}_F)}\sum_{b_0\in B_0(F)}\varphi(x^{-1}\lambda^{-1}u^{-1}b_0w_2v\lambda x)\theta(v)dudv.
\end{align*}

Hence the decomposition \eqref{0019} can be refined as 
\begin{equation}\label{0020}
\K_{\Sin,1}^{(2)}(x;I_3)= \K_{\Sin,1,0}^{(2,1)}(x,x)-\K_{\Sin,1}^{(2,2)}(x,x)+\K_{\Sin,1,\neq 0}^{(2,1)}(x,x)+ \K_{\Sin,1}^{(1,1)}(x,x),
\end{equation}

Integrating \eqref{0020} over $[N_P]=N_P(F)\backslash N_P(\mathbb{A}_F)$ to see 
\begin{equation}\label{0021}
\int\K_{\Sin,1}^{(2)}(nx;I_3)dn=\int \K_{\Sin,1,\neq 0}^{(2,1)}(nx,nx)dn+\int\K_{\Sin,1}^{(1,1)}(nx,nx)dn.
\end{equation}

Also, substituting the expression of $\K_{\Sin,1,\neq 0}^{(2,1)}(x,x)$ we then obtain:
\begin{equation}\label{0022}
\sum_{p\in B_0(F)\backslash P_0(F)}\int_{[N_P]} \K_{\Sin,1,\neq 0}^{(2,1)}(nx,nx)dn=\int_{[N_P]}\K_{\infty,1}^{(2)}(nx,nx)dn.
\end{equation}

Hence, by \eqref{0026}, \eqref{0029}, \eqref{0021}, \eqref{0022} and \eqref{0023}, we only need to consider the contribution from $\K_{\infty,1}^{(2)}(x,x),$ $\K_{\Sin,1,\neq 0}^{(2,1)}(x,x)$ and $\K_{\Sin,1}^{(1,1)}(x,x).$ In fact, a straightforward computation shows that the contribution from $\K_{\Sin,1,\neq 0}^{(2,1)}(x,x)$ cancels that from $\K_{\infty,1}^{(2)}(x,x).$ Therefore, we only need to compute the contribution from $\K_{\Sin,1}^{(1,1)}(x,x).$
\medskip 

To deal with $\K_{\Sin,1}^{(1,1)}(x,x),$ we still need to apply Poisson summation, which implies that  $\K_{\Sin,1}^{(1,1)}(x,x)=\K_{\Sin,1,0}^{(1,1)}(x,x)+\K_{\Sin,1,\neq 0}^{(1,1)}(x,x),$ where the constant term $\K_{\Sin,1,0}^{(1,1)}(x,x)$ is equal to 
\begin{align*}
\sum_{\delta\in N(F)\backslash R_1(F)}\int_{N_{1}(F)\backslash N_{1}(\mathbb{A}_F)}\int_{N_2(\mathbb{A}_F)}\sum_{b_0\in B_0(F)}\varphi(x^{-1}\delta^{-1}u^{-1}b_0w_2v\delta x)\theta(u)dudv;
\end{align*}
and $\K_{\Sin,1,\neq0}^{(1,1)}(x,x)$ the contribution from non-constant terms, is equal to 
\begin{align*}
\sum_{\delta\in N(F)\backslash R_1(F)}\sum_{\beta\in F^{\times}}\int_{[N_{1}]}\int_{N_2(\mathbb{A}_F)}\sum_{b_0\in Z_G(F)\backslash B(F)}\varphi(x^{-1}\delta^{-1}u^{-1}b_0w_2v\delta x)\theta(u)\theta(\beta v)dudv.
\end{align*}

By a change of variable, we see $\K_{\Sin,1,\neq 0}^{(1,1)}(x,x)$ can be rewritten as 
\begin{align*}
\sum_{\lambda\in N(F)\backslash B_0(F)}\int_{N_{1}(F)\backslash N_{1}(\mathbb{A}_F)}\int_{N_2(\mathbb{A}_F)}\sum_{b_0\in B_0(F)}\varphi(x^{-1}\lambda^{-1}u^{-1}b_0w_2v\lambda x)\theta(u)\theta(v)dudv.
\end{align*}

As before, we can form the distributions respectively: 
\begin{align*}
&I_{\Sin,1,0}^{(1,1)}(s)=\int_{Z_G(\mathbb{A}_F)B_0(F)\backslash G(\mathbb{A}_F))}\K_{\Sin,1,0}^{(1,1)}(s)f(x,s)dx;\\
&I_{\Sin,1,\neq 0}^{(1,1)}(s)=\int_{Z_G(\mathbb{A}_F)B_0(F)\backslash G(\mathbb{A}_F))}\K_{\Sin,1,\neq 0}^{(1,1)}(s)f(x,s)dx.
\end{align*}

Let $t_0=\diag(t_1,t_2,1)\in B_0(F).$ Let
\begin{align*}
h(t_0):=\int_{[N_2\backslash N]}\int_{[N_1]}\sum_{n\in N(F)}\varphi(x^{-1}v^{-1}u^{-1}nt_0w_2vy)\theta(u)dudv.
\end{align*}

Then $h(t_0)$ is well defined. Let $v_0=\begin{pmatrix}
1&a&b\\
&1&\\
&&1
\end{pmatrix},$ where $a=-t_1b.$ Note that 
\begin{align*}
h(t_0)=&\int_{[N_2\backslash N]}\int_{[N_1]}\sum_{n\in N(F)}\varphi(x^{-1}v^{-1}v_0^{-1}u^{-1}nt_0w_2v_0vy)\theta(u)dudv\\
=&\theta((t_1^{-1}-t_1t_2^{-1})a)\int_{[N_2\backslash N]}\int_{[N_1]}\sum_{n\in N(F)}\varphi(x^{-1}v^{-1}u^{-1}nt_0w_2vy)\theta(u)dudv,
\end{align*}
namely, $h(t_0)=\theta((t_1^{-1}-t_1t_2^{-1})a)h(t_0)$ for any $a\in \mathbb{A}_F.$ Hence $h(t_0)$ is nonvanishing unless $t_2=t_1^2.$ Therefore, we can replace $\K_{\Sin,1,0}^{(1,1)}(s)$ with
\begin{align*}
\sum_{\delta\in N(F)\backslash R_1(F)}\int_{N_{1}(F)\backslash N_{1}(\mathbb{A}_F)}\int_{N_2(\mathbb{A}_F)}\sum_{b_0\in B_0^*(F)}\varphi(x^{-1}\delta^{-1}u^{-1}b_0w_2v\delta x)\theta(u)dudv,
\end{align*}
where $B_0^*(F)$ consists of elements $\diag(t,t^2,1)\mod Z_G(F),$ $t\in F^{\times}.$ Then
\begin{align*}
I_{\Sin,1,0}^{(1,1)}(s)=&\int_{Z_G(\mathbb{A}_F)N(F)\backslash G(\mathbb{A}_F)}\sum_{b\in F}\sum_{c\in F}\int_{\mathbb{A}_F}\int_{\mathbb{A}_F}f(x,s)\cdot\\
&\quad \varphi\left(x^{-1}\begin{pmatrix}
1&a&\\
&1&c\\
&&1
\end{pmatrix}w_2\begin{pmatrix}
1&b&\\
&1&e\\
&&1
\end{pmatrix}x\right)\theta(a)dadedx
\end{align*}

Now we use Iawasawa decomposition to obtain 
\begin{align*}
I_{\Sin,1,0}^{(1,1)}(s)=&\int_{\mathbb{A}_F^{\times}}\int_{\mathbb{A}_F^{\times}}\int_{K}\int_{\mathbb{A}_F^4}\varphi\left(k^{-1}\begin{pmatrix}
1&a&\\
&1&c\\
&&1
\end{pmatrix}\begin{pmatrix}
t_1&&\\
&t_1^2&\\
&&1
\end{pmatrix}w_2\begin{pmatrix}
1&b&\\
&1&e\\
&&1
\end{pmatrix}k\right)\\
&\quad \theta(t_1t_2a)\theta(t_2b)|t_1|^{2s}|t_2|^s\omega(t_1)\tau(t_1)^2\tau(t_2)f(k,s)dadbdcdedkd^{\times}t_2d^{\times}t_1.
\end{align*}

Then by Tate's thesis, we conclude that $I_{\Sin,1,0}^{(1,1)}(s)$ can be written as 
\begin{align*}
I_{\Sin,1,0}^{(1,1)}(s)=&\Lambda(s,\tau)\Lambda(3s,\tau^3)\int_{\mathbb{A}_F^{\times}}Q_{\varphi}(t_1,s)|t_1|^{2s}|w(t_1)\tau(t_1)^2d^{\times}t_1,
\end{align*}
where $Q_{\varphi}(t_1,s)$ is entire with respect to $s$ and has compact support as function of $t_1.$ Hence $I_{\Sin,1,0}^{(1,1)}(s)$ converges absolutely when $\Re(s)>1,$ and it has the analytic property 
\begin{equation}\label{0030}
I_{\Sin,1,0}^{(1,1)}(s)\sim \Lambda(s,\tau)\Lambda(3s,\tau^3).
\end{equation}
\medskip 

Let $X=\mathbb{A}_F^{\times}\times \mathbb{A}_F^{\times}.$ Likewise, we have
\begin{align*}
I_{\Sin,1,\neq 0}^{(1,1)}&(s)=\int_{X}\sum_{t\in F^{\times}}\int_{K}\int_{\mathbb{A}_F^4}\varphi\left(k^{-1}\begin{pmatrix}
1&a&\\
&1&c\\
&&1
\end{pmatrix}\begin{pmatrix}
t_1&&\\
&t_1^2&\\
&&1
\end{pmatrix}w_2\begin{pmatrix}
1&b&\\
&1&e\\
&&1
\end{pmatrix}k\right)\\
& \theta(t_1t_2a)\theta(t_2b)\theta(tt_1c)\theta(tt_1e)|t_1|^{2s}|t_2|^s\omega(t_1)\tau(t_1)^2\tau(t_2)f(k,s)dadbdcdedkdx.
\end{align*}
Likewise, $t_1$ actually runs over a compact set, by Tate's thesis, we conclude that $I_{\Sin,1,\neq 0}^{(1,1)}(s)$ is an integral representation for  $\Lambda(s,\tau)\Lambda(3s,\tau^3).$ Hence $I_{\Sin,1,\neq 0}^{(1,1)}(s)$ converges absolutely when $\Re(s)>1,$ and it has the analytic property 
\begin{equation}\label{0031}
I_{\Sin,1,\neq 0}^{(1,1)}(s)\sim \Lambda(s,\tau)\Lambda(3s,\tau^3).
\end{equation}
\medskip

Now we put the above formulas together to see  
\begin{equation}\label{0032}
I_{\Sin}(s)=I_{\Sin,2}(s)+I_{\infty,3}^{(2)}(s)+I_{\Sin,1,0}^{(1,1)}(s)+I_{\Sin,1,\neq 0}^{(1,1)}(s).
\end{equation}
By \eqref{0026}, \eqref{0029}, \eqref{0030} and \eqref{0031}, we then conclude that $I_{\Sin}(s)$ converges absolutely when $\Re(s)>1;$ admits a meromorphic continuation to the whole $s$-plane; moreover, $I_{\Sin}(s)\cdot \Lambda(s,\tau)^{-1}$ admits a meromorphic continuation to $\Re(s)>1/3,$ with possible simple poles at $s\in\{1, 1/2\},$ proving Theorem \ref{Y}.

\subsection{Singular Expansion for $\GL(4)$}\label{2.3}

To study $ I_{\Sin}(s),$ we need to investigate $ \K_{\Sin}(x,y):=\K_{\Geo,\Sin}(x,y)-\K_{\infty,\Sin}^{(4)}(x,y)$ and each $\K_{\infty}^{(k)}(x,y),$ $2\leq k\leq 3.$ Hence, we first need a similar result as Lemma \ref{21} to describe the $P(F)$-conjugacy classes of 
\begin{align*}
(Q_1\cup Q_2)- P=Bw_3N\sqcup  Bw_1w_3N\sqcup Bw_2w_3N\sqcup Bw_3w_2N\sqcup Bw_2w_3w_2N
\end{align*} 
in terms of $B(F)\backslash P(F).$ Let $S$ be the standard parabolic subgroup of type $(2,1,1).$ Denote by $S_0(F)=Z_G(F)\backslash S(F).$  First, we consider the conjugation by $S(F)\backslash P(F).$

\begin{lemma}\label{22}
	Let notation be as before. Denote by $\mathfrak{G}_1=(Bw_3N\sqcup  Bw_1w_3N\sqcup Bw_2w_3N)^{B(F)\backslash S(F)}.$ Then 
	\begin{equation}\label{24}
	\mathfrak{G}-P(F)=\mathfrak{G}_1^{S(F)\backslash P(F)}.
	\end{equation}
	Moreover, $Bw_3N\sqcup  Bw_1w_3N\sqcup Bw_2w_3N$ forms a set of representatives of $B(F)\backslash P(F)$-conjugacy classes of $Q_1(F)\cup Q_2(F)-P(F).$
\end{lemma}
\begin{proof}
	By Bruhat decomposition, we see 
	\begin{equation}\label{0039}
	S(F)\backslash P(F)=\{1\}\sqcup w_2N_2\sqcup w_2w_1N_{21}.
	\end{equation}
	Since $(Bw_3w_2N)^{w_2}\subseteq Bw_2w_3w_2N\sqcup Bw_2w_3N,$ the $B(F)\backslash P(F)$ conjugacy class of $Bw_3w_2N$ is contained in that of $Bw_2w_3w_2N\sqcup Bw_2w_3N.$ Let $\gamma\in Bw_2w_3w_2N.$ Write $\gamma$ into its Bruhat normal form:
	\begin{align*}
	\gamma=\begin{pmatrix}
	1&* &*&*\\
	& 1&a&*\\
	&&1&*\\
	&&&1
	\end{pmatrix}\begin{pmatrix}
	t_1& &&\\
	& t_2&&\\
	&&t_3&\\
	&&&t_4
	\end{pmatrix}w_2w_3w_2\begin{pmatrix}
	1 &&&\\
	& 1&b&*\\
	&&1&*\\
	&&&1
	\end{pmatrix}.
	\end{align*}
	If $a+b\neq 0,$ then $\gamma\in (Bw_2w_3N)^{w_2N_2};$ if $a+b= 0,$ then $\gamma\in (Bw_3N)^{w_2N_2}.$ Hence $(Bw_2w_3w_2N)^{B(F)\backslash P(F)}\subseteq (Bw_3N)^{B(F)\backslash P(F)}\cup (Bw_2w_3N)^{B(F)\backslash P(F)}.$ Thus, $\mathfrak{G}-P(F)=(Bw_3N\sqcup  Bw_1w_3N\sqcup Bw_2w_3N)^{B(F)\backslash P(F)}.$ Hence, \eqref{24} follows. Moreover, we have:
	\begin{claim}\label{31}
		The set $Bw_3N_3\sqcup  Bw_1w_3N_{13}\sqcup Bw_2w_3N_{23}$ forms representatives of $(Bw_3N\sqcup  Bw_1w_3N\sqcup Bw_2w_3N)^{S(F)\backslash P(F)}.$
	\end{claim}
	This proves the rest of Lemma \ref{22}.
\end{proof}

\begin{proof}[Proof of Claim \ref{31}] Let $w, s_{\alpha}$ be Weyl elements and the length $l(s_{\alpha})=1$. Let $C(w)$ and $C(s_{\alpha})$ be the Bruhat cells, respectively. Recall we have proved in \cite{Yang19} that 
	\begin{equation}\label{20}
	C(w)^{s_{\alpha}}=
	\begin{cases}
	C(s_{\alpha}ws_{\alpha}),\ \text{if $l(s_{\alpha}ws_{\alpha})=l(w)+2$;}\\
	C(s_{\alpha}w)\sqcup C(s_{\alpha}ws_{\alpha}),\ \text{if $l(s_{\alpha}w)<l(w),$ $l(s_{\alpha}ws_{\alpha})>l(s_{\alpha}w);$}\\
	C(ws_{\alpha})\sqcup C(s_{\alpha}ws_{\alpha}),\ \text{if $l(ws_{\alpha})<l(w),$ $l(s_{\alpha}ws_{\alpha})>l(ws_{\alpha});$}\\
	C(w)\sqcup C(s_{\alpha}w)\sqcup C(ws_{\alpha})\sqcup C(s_{\alpha}ws_{\alpha}),\ \text{otherwise,}
	\end{cases}
	\end{equation}
	where $C(w)^{s_{\alpha}}:=C(s_{\alpha})C(w)C(s_{\alpha}).$ Then by \eqref{20} and \eqref{0039} we see that 
	\begin{align*}
	&(Bw_3N)^{S(F)\backslash P(F)}\subseteq Bw_3N\sqcup Bw_2w_3w_2N\sqcup Bw_1w_2w_3w_2w_1N;\\
	&(Bw_1w_3N)^{S(F)\backslash P(F)}\subseteq Bw_1w_3N\sqcup Bw_2w_1w_3w_2N\sqcup Bw_1w_2w_1w_3w_2w_1N;\\
	&(Bw_2w_3N)^{S(F)\backslash P(F)}\subseteq Bw_2w_3N\sqcup Bw_3w_2N\sqcup Bw_1w_3w_2w_1N\sqcup Bw_1w_2w_3w_2w_1N.
	\end{align*}
	
	Thus, by the disjointness of different Bruhat cells, the only possible intersection of orbits $(Bw_3N)^{S(F)\backslash P(F)},$ $(Bw_1w_3N)^{S(F)\backslash P(F)}$ and $(Bw_2w_3N)^{S(F)\backslash P(F)}$ must lie in $Bw_1w_2w_3w_2w_1N.$ Suppose $(Bw_3N)^{S(F)\backslash P(F)}\cap (Bw_2w_3N)^{S(F)\backslash P(F)}$ is nonempty. Then there exists some $b\in B(F,)$ $v_3\in N_3(F)$ and $u_{21}\in N_{21}(F)$ such that 
	\begin{equation}\label{0040}
	w_2w_{1}u_{21}^{-1}w_1w_2bw_3v_{3}w_2w_1u_{21}w_1w_2\in B(F)w_2w_3N_{23}(F).
	\end{equation}
	However, $w_2w_1u_{21}w_1w_2\in B(F)\sqcup B(F)w_2w_1w_2N(F)\sqcup B(F)w_2N(F).$ Denote by $\gamma=w_2w_{1}u_{21}^{-1}w_1w_2bw_3v_{3}w_2w_1u_{21}w_1w_2.$ Then applying \eqref{20} again we obtain that 
	\begin{equation}\label{0041}
	\gamma\in Bw_3N \sqcup Bw_2w_3w_2N  \sqcup Bw_1w_2w_3w_2w_1N\sqcup Bw_1w_2w_1w_3w_2w_1N.
	\end{equation}
	
	Nevertheless, the Bruhat cells on the right hand side of \eqref{0041} are different from $B(F)w_2w_3N_{23}(F),$ hence there is no intersection with $B(F)w_2w_3N_{23}(F),$ namely, \eqref{0040} cannot hold. A contradiction!
	\medskip 
	
	Thus the orbits $(Bw_3N)^{S(F)\backslash P(F)},$ $(Bw_1w_3N)^{S(F)\backslash P(F)}$ and $(Bw_2w_3N)^{S(F)\backslash P(F)}$ do not have any intersection. Next we need to show these orbits are transversal. We verify them separately as follows:
	
	\begin{enumerate}
		\item[(i).] Assume there are $b_1w_3u_1, b_2w_3v_1\in B(F)w_3N_3(F),$ and $\lambda_1, \lambda_2\in B(F)\backslash P(F),$ such that $\lambda_1^{-1}b_1w_3u_1\lambda_1=\lambda_2^{-1}b_2w_3v_1\lambda_2.$ Then by disjointness of different Bruhat cells, $\lambda_1$ and $\lambda_2$ must lie in the same connected component given on the right hand side of \eqref{0039}. Assume further $\lambda_1\neq \lambda_2,$ then $\lambda_1\lambda_2^{-1}\in B(F)w_2N(F)\sqcup B(F)w_2w_1w_2N(F).$ Then $\lambda_1^{-1}b_1w_3u_1\lambda_1$ can not equal $\lambda_2^{-1}b_2w_3v_1\lambda_2.$ A contradiction! Thus the conjugation of $B(F)\backslash P(F)$ on $B(F)w_3N_3(F)$ is transversal.
		\item[(ii).] Assume there are $b_1w_1w_3u_1, b_2w_1w_3v_1\in B(F)w_1w_3N_3(F),$ and $\lambda_1, \lambda_2\in B(F)\backslash P(F),$ such that $\lambda_1^{-1}b_1w_1w_3u_1\lambda_1=\lambda_2^{-1}b_2w_1w_3v_1\lambda_2.$ Then by disjointness of different Bruhat cells, $\lambda_1$ and $\lambda_2$ must lie in the same connected component given on the right hand side of \eqref{0039}. Assume further $\lambda_1\neq \lambda_2,$ then $\lambda_1\lambda_2^{-1}\in B(F)w_2N(F)\sqcup B(F)w_2w_1w_2N(F).$ Then by \eqref{20}, $\lambda_1^{-1}b_1w_1w_3u_1\lambda_1$ can not equal $\lambda_2^{-1}b_2w_1w_3v_1\lambda_2.$ A contradiction! Thus the conjugation of $B(F)\backslash P(F)$ on $B(F)w_1w_3N_3(F)$ is transversal.
		\item[(iii).] Assume there are $b_1w_2w_3u_1, b_2w_2w_3v_1\in B(F)w_1w_3N_3(F),$ and $\lambda_1, \lambda_2\in B(F)\backslash P(F),$ such that $\lambda_1^{-1}b_1w_2w_3u_1\lambda_1=\lambda_2^{-1}b_2w_2w_3v_1\lambda_2.$ Likewise, $\lambda_1$ and $\lambda_2$ must lie in the same connected component given on the right hand side of \eqref{0039}. Assume further $\lambda_1\neq \lambda_2,$ then $\lambda_1\lambda_2^{-1}\in B(F)w_2N(F)\sqcup B(F)w_2w_1w_2N(F).$ Then by \eqref{20}, $\lambda_1^{-1}b_1w_1w_3u_1\lambda_1\neq \lambda_2^{-1}b_2w_1w_3v_1\lambda_2.$ A contradiction! Thus the conjugation of $B(F)\backslash P(F)$ on $B(F)w_1w_3N_3(F)$ is transversal.
	\end{enumerate}
	
	Therefore, the set $Bw_3N_3\sqcup  Bw_1w_3N_{13}\sqcup Bw_2w_3N_{23}$ forms representatives of $(Bw_3N\sqcup  Bw_1w_3N\sqcup Bw_2w_3N)^{S(F)\backslash P(F)}.$
\end{proof}

\begin{lemma}\label{030}
	Let notation be as in Lemma \ref{22}. Then
	\begin{equation}\label{031}
	\mathfrak{G}_1=(Bw_2w_3N_{23})^{B(F)\backslash S(F)}\sqcup Sw_3N_3.
	\end{equation}
	Moreover, the set $Bw_2w_3N_{23}$ consists of representatives of $(Bw_2w_3N_{23})^{B(F)\backslash S(F)}.$
\end{lemma}
\begin{proof}
	Since $l(w_1w_2w_3w_1)=l(w_2w_3)+2,$ the set $Bw_2w_3N_{23}$ consists of representatives of $(Bw_2w_3N_{23})^{B(F)\backslash S(F)}.$ Hence \eqref{031} follows from Lemma \ref{22}.
\end{proof}
\medskip 

Let $\mathcal{A}_1(F)=(S_0w_3N_3)^{S_0\backslash P_0},$ $\mathcal{A}_2(F)=(B_0w_2w_3N_{23})^{B_0\backslash P_0},$ and $\mathcal{A}_0(F)=P_0.$ For $0\leq i\leq 2,$ we denote by 
\begin{align*}
&\K_{\Geo,\Sin,i}(x,y)=\sum_{\gamma\in \mathcal{A}_i(F)}\varphi(x^{-1}\gamma y);\\
&\K_{\infty,\Sin,i}^{(4)}(x,y)=\int_{N_P(F)\backslash N_P(\mathbb{A}_F)}\sum_{\gamma\in \mathcal{A}_i(F)}\varphi(x^{-1}u^{-1}\gamma y)du.
\end{align*}

Then by Lemma \ref{030}, we have 
\begin{align*}
\K_{\Sin}(x,y)=\K_{\Sin,0}(x,y)+\K_{\Sin,1}(x,y)+\K_{\Sin,2}(x,y),
\end{align*}
where $\K_{\Sin,0}(x,y)=\K_{\Geo,\Sin,0}(x,y)-\K_{\infty,\Sin,0}^{(4)}(x,y),$ and
\begin{align*}
\K_{\Sin,1}(x,y)&=\sum_{\lambda\in S_0(F)\backslash P_0(F)}\K_{\Geo,\Sin,1}(\lambda x,\lambda y)-\sum_{\lambda\in S_0(F)\backslash P_0(F)}\K_{\infty,\Sin,1}^{(4)}(\lambda x,\lambda y),\\
\K_{\Sin,2}(x,y)&=\sum_{\delta\in B_0(F)\backslash P_0(F)}\K_{\Geo,\Sin,2}(\delta x,\delta y)-\sum_{\delta\in B_0(F)\backslash P_0(F)}\K_{\infty,\Sin,2}^{(4)}(\delta x,\delta y).
\end{align*}

\medskip

One then has to handle terms on the right hand side of the above identity separately. We deal with $\K_{\Sin,0}(x,y)$ first. Denote by 
\begin{align*}
\K_{\infty,0}^{(k)}(x,x)=\sum_{\delta_k\in R_{k-1}(F)\backslash R_{3}(F)}\int_{[V_k']}\int_{[V_k]}\sum_{\gamma\in P_0(F)}\varphi((uu'\delta_kx)^{-1}\gamma yx)du\theta(u')du',
\end{align*}
where $2\leq k\leq 3.$ Denote by $\K_{\infty,0}^{(1)}(x,x)=\K_{\Sin,0}(x,x)-\K_{\infty,0}^{(3)}(x,x)-\K_{\infty,0}^{(2)}(x,x).$ Hence we can apply Proposition \ref{Fourier} to get 
\begin{align*}
\K_{\infty,0}^{(1)}(x,x)=\sum_{\delta\in N(F)\backslash P_0(F)}\int_{N(F)\backslash N(\mathbb{A}_F)}\sum_{\gamma\in P_0(F)}\varphi(x^{-1}\delta^{-1}n^{-1}\gamma x)\theta(u)du.
\end{align*}
\subsubsection{Contribution from $\K_{\infty,0}^{(1)}(x,x)$}
Now we defined the distribution $I_{\infty,0}^{(1)}(s)$ correspondingly, namely,
\begin{equation}\label{0033}
I_{\infty,0}^{(1)}(s):=\int_{Z_G(\mathbb{A}_F)P_0(F)\backslash G(\mathbb{A}_F)}\K_{\infty,0}^{(1)}(x,x)f(x,s)dx.
\end{equation}

Using Bruhat decomposition $P_0(F)=B_0(F)\sqcup B_0(F)w_1N(F)\sqcup B_0(F)w_2N(F)\sqcup B_0(F)w_2w_1N(F)\sqcup B_0(F)w_1w_2N(F)\sqcup B_0(F)w_1w_2w_1N(F)$ to further expand the function $\K_{\infty,0}^{(1)}(x,x),$ then substituting them into \eqref{0033}, we then obtain 
\begin{equation}\label{0034}
I_{\infty,0}^{(1)}(s)=\int_{Z_G(\mathbb{A}_F)N(F)\backslash G(\mathbb{A}_F)}\int_{[N]}\sum_{\gamma\in B_0(F)}\varphi(x^{-1}\delta^{-1}n^{-1}\gamma x)\theta(u)duf(x,s)dx.
\end{equation}

Now we can apply Iwasawa decomposition $G(\mathbb{A}_F)=N(\mathbb{A}_F)T(\mathbb{A}_F)K$ into \eqref{0034} to obtain
\begin{align*}
I_{\infty,0}^{(1)}(s)=&\int_{(\mathbb{A}_F^{\times})^3}\int_{K}\sum_{t_1,t_2,t_3\in F^{\times}}\int_{\mathbb{A}_F^6}\varphi\left(k^{-1}\begin{pmatrix}
t_1& &&\\
&t_2&&\\
&&t_3&\\
&&&1
\end{pmatrix}\begin{pmatrix}
1&a&b&c\\
&1&e&f\\
&&1&g\\
&&&1
\end{pmatrix}k\right)\\
&\quad\theta(\alpha a)\theta(\beta e)\theta(\gamma g)\tau(\alpha)\tau(\beta)^2\tau(\gamma)^3|\alpha|^s|\beta|^{2s}|\gamma|^{3s}f(k,s)dndkd^{\times}\alpha d^{\times}\beta d^{\times}\gamma,
\end{align*}
where $dn=dadbdcdedfdg.$ Then by Tate's thesis, we conclude that $I_{\infty,0}^{(1)}(s)$ is an integral representation for  $\Lambda(s,\tau)\Lambda(2s,\tau^2)\Lambda(3s,\tau^3)\Lambda(4s,\tau^4).$ Hence $I_{\infty,0}^{(1)}(s)$ converges absolutely when $\Re(s)>1,$ and it has the analytic property 
\begin{equation}\label{0035}
I_{\infty,0}^{(1)}(s)\sim \Lambda(s,\tau)\Lambda(2s,\tau^2)\Lambda(3s,\tau^3)\Lambda(4s,\tau^4).
\end{equation}

As a consequence, $I_{\infty,0}^{(1)}(s)$ admits a meromorphic continuation to $s$-plane, with possible poles (which are simple if exist) at $s\in\{1, 1/2, 1/3, 1/4\}.$ 
\medskip 

\subsubsection{Contributions from $\K_{\infty}^{(2)}(x,x)$}
For a Weyl element $w,$ denote by $C(w)$ the Bruhat cell $B(F)w N(F).$ Then 
\begin{align*}
\GL(4,F)=&P(F)\sqcup C(w_3)\sqcup C(w_1w_3)\sqcup C(w_2w_3)\sqcup C(w_1w_3w_2)\sqcup C(w_1w_2w_3w_1w_2w_1) \\
&\ \sqcup C(w_3w_2)\sqcup C(w_2w_1w_3)\sqcup C(w_2w_3w_2)\sqcup C(w_3w_2w_1)\sqcup C(w_3w_1w_2w_1)\\
&\  \sqcup C(w_2w_3w_1w_2)\sqcup C(w_2w_3w_2w_1)\sqcup C(w_1w_2w_3w_1)\sqcup C(w_1w_2w_3w_2)\\
&\  \sqcup C(w_2w_3w_1w_2w_1)\sqcup C(w_1w_2w_3w_1w_2)\sqcup C(w_1w_2w_3w_2w_1)\sqcup C(w_1w_2w_3).
\end{align*}

Based on this decomposition, we can write $\K_{\infty}^{(2)}(x,y)=\sum_{i=0}^{18}\K_{\infty,i}^{(2)}(x,y),$ where
\begin{equation}\label{0001}
\K_{\infty,i}^{(2)}(x,y)=\sum_{\delta\in R_{1}(F)\backslash P_{0}(F)}\int_{[V_2']}\int_{[V_2]}\sum_{\gamma\in \mathcal{B}_i^{(2)}(F)}\varphi((uu'\delta x)^{-1}\gamma y)du\theta(u')du',
\end{equation}
where $\mathcal{B}_{0}^{(2)}(F)=P_0(F),$ and $\sqcup_{i=1}^{18}\mathcal{B}_{i}^{(2)}(F)$ consists of the above 18 Bruhat cells modulo $Z_G(F)$. Explicitly, let $\mathcal{B}_1^{(2)}(F)=C(w_2w_3)\sqcup C(w_3w_2)\sqcup C(w_2w_3w_2),$ $\mathcal{B}_2^{(2)}(F)=C(w_1w_2w_3w_2w_1),$ and    
$\mathcal{B}_3^{(2)}(F)=C(w_1w_2w_3w_1w_2w_1).$ Denote also by 
\begin{align*}
\K_{\infty,i}^{*}(x)=\int_{[N]}\int_{[N_1]}\int_{[V_2]}\sum_{\gamma\in \mathcal{B}_i^{(2)}(F)}\varphi(x^{-1}n^{-1}u'^{-1}u^{-1}\gamma nx)du\theta(u')du'dn.
\end{align*}

Then a straightforward computation shows that 
\begin{align*}
\int_{N(F)\backslash N(\mathbb{A}_F)}\int_{N_1(F)\backslash N_1(\mathbb{A}_F)}\int_{V_2(F)\backslash V_2(\mathbb{A}_F)}\K(uu'nx, nx)du\theta(u')du'dn=\sum_{i=0}^2\K_{\infty,i}^{*}(x).
\end{align*}

Thus, formally one has $I_{\infty}^{(2)}(s)=I_{\infty,0}^{(2)}(s)+I_{\infty,1}^{(2)}(s)+I_{\infty,2}^{(2)}(s)+I_{\infty,3}^{(2)}(s),$ where 
\begin{equation}\label{0002}
I_{\infty,i}^{(2)}(s)=\int_{Z_G(\mathbb{A}_F)P_0(F)\backslash G(\mathbb{A}_F)}\K_{\infty,i}^{(2)}(x,x)f(x,s)dx,\quad 1\leq i\leq 3.
\end{equation}

\begin{prop}\label{11}
	Let notation be as before. Then $I_{\infty,2}^{(2)}(s)$ admits a meromorphic continuation to the whole $s$-plane, and 
	\begin{equation}\label{0036}
I_{\infty,2}^{(2)}(s)\sim \Lambda(s,\tau)\Lambda(2s,\tau^2)\Lambda(4s,\tau^4).
	\end{equation}
	
\end{prop}
\begin{proof}
For any $\gamma\in\mathcal{B}_2^{(2)}(F),$ we can write $\gamma$ uniquely as $\gamma=u_1tu_2,$ where $u_1\in N(F),$ $t=\diag(t_1,t_2,t_3,1)$ and $u_2\in N_{w_1w_2w_3w_2w_1}(F).$ Let $I_4\neq v\in N_2(\mathbb{A}_F).$ Substituting \eqref{0001} into \eqref{0002} we then obtain, by writing $X_1=Z_G(\mathbb{A}_F)R_1(F)\backslash G(\mathbb{A}_F),$ that $I_{\infty,2}^{(2)}(s)=\sum_{t_1,t_2,t_3}I_{\infty,2}^{(2)}(s;t_1,t_2,t_3),$ where 
\begin{align*}
I_{\infty,2}^{(2)}(s;t_1,t_2,t_3)=&\int_{X_1}
\int_{[V_2']}\int_{[V_2]}\sum_{u_1}\sum_{u_2}\varphi((uu'x)^{-1}\gamma x)du\theta(u')du'f(x,s)dx,
\end{align*}
since $I_{\infty,2}^{(2)}(s)$ converges absolutely when $\Re(s)>1.$ Now a changing of variable $x\mapsto vx$ implies $I_{\infty,2}^{(2)}(s;t_1,t_2,t_3)=\theta((1-t_3t_2^{-1})v)I_{\infty,2}^{(2)}(s;t_1,t_2,t_3).$ Hence, we have  $I_{\infty,2}^{(2)}(s;t_1,t_2,t_3)=0$ unless $t_2=t_3.$ Therefore, using Iwasawa decomposition, 
\begin{align*}
I_{\infty,2}^{(2)}(s)=&\int_{(\mathbb{A}_F^{\times})^3}\int_{K}\sum_{t_1\in F^{\times}}\int_{\mathbb{A}_F^{11}}\varphi\left(k^{-1}\lambda_{t_1,\gamma,e,g}^{a,b,c,f}\widetilde{w}\begin{pmatrix}
1&a'&b'&c'\\
&1&&f'\\
&&1&g'\\
&&&1
\end{pmatrix}k\right)\theta(\alpha e)\theta(\beta g)\\
&\quad\theta(\beta g')\tau^2(\alpha)\tau(\beta)|\alpha|^{2s}|\beta|^{s}\tau \omega(\gamma)|\gamma|^{s}f(k,s)dndkd^{\times}\alpha d^{\times}\beta d^{\times}\gamma,
\end{align*}
where $\widetilde{w}=w_1w_2w_3w_2w_1,$ $dn=dadb\cdots dg da'\cdots dg';$ and
\begin{align*}
\lambda_{t_1,\gamma,e,g}^{a,b,c,f}=\begin{pmatrix}
1&a&b&c\\
&1&e&f\\
&&1&g\\
&&&1
\end{pmatrix}\begin{pmatrix}
\gamma^2t_1& &&\\
&\gamma&&\\
&&\gamma&\\
&&&1
\end{pmatrix}.
\end{align*}
Since $\gamma$ runs over a compact subset of $\mathbb{A}_{F}^{\times},$ we then conclude \eqref{11} from Tate's thesis.
\end{proof}

\begin{prop}\label{12}
	Let notation be as before. Then $I_{\infty,3}^{(2)}(s)$ admits a meromorphic continuation to the whole $s$-plane, and 
	\begin{align*}
	\frac{I_{\infty,3}^{(2)}(s)}{\Lambda(2s,\tau^2)\Lambda(4s,\tau^4)}=\sum_{\substack{[E:F]=2}}Q_E(s,\tau)\Lambda(s,\tau\circ N_{E/F})+\frac{Q(s,\tau)\Lambda(s,\tau)\Lambda(2s,\tau^2)}{\Lambda(s+1,\tau)} ,
	\end{align*}
	where the sum over number fields $E$ is finite, each $Q_E(s,\tau)$ is entire; and $Q(s,\tau)$ is entire.
\end{prop}
\begin{proof}
	Let $\widetilde{w}=w_1w_2w_3w_1w_2w_1=\widetilde{w}_1w_2,$ where $\widetilde{w}_1=w_1w_2w_3w_2w_1$  Then 
	\begin{equation}\label{0003}
	\begin{pmatrix}
	1&a&b&c\\
	&1&e&f\\
	&&1&g\\
	&&&1
	\end{pmatrix}\widetilde{w}=\begin{pmatrix}
	1&a&b&c\\
	&1&&f\\
	&&1&g\\
	&&&1
	\end{pmatrix}\widetilde{w}_1\begin{pmatrix}
	1&&&\\
	&1&e&\\
	&&1&\\
	&&&1
	\end{pmatrix}w_2.
	\end{equation}
	
	Then we can apply Iwasawa decomposition to see
		\begin{align*}
		I_{\infty,3}^{(2)}(s)
		=&\int_{(\mathbb{A}_F^{\times})^3}\int_{K}\int_{\mathbb{A}_F^{12}}\sum_{t_1,t_2}\varphi\left(k^{-1}\begin{pmatrix}
		1&a&b&c\\
		&1&e&f\\
		&&1&g\\
		&&&1
		\end{pmatrix}\widetilde{w}\lambda_{t_1,t_2,y_2,y_3}^{h,l,m,p,q,r}k\right)\theta(y_1g)\theta(y_1r)\\
		&\  \theta(y_2e)\theta(y_2p)|y_3|^{s+3}|y_2|^{s+1}|y_1|^{2s}\tau(y_1)^2\tau(y_2)\tau\omega(y_3)dnd^{\times}yf(k,s)dk,
		\end{align*}
where $dn=dadb\cdots dg\cdot dhdl\cdots dr;$ $d^{\times}y=d^{\times}y_1d^{\times}y_2d^{\times}y_3;$ and 
\begin{align*}
\lambda_{t_1,t_2,y_2,y_3}^{h,l,m,p,q,r}=\begin{pmatrix}
t_1y_3^2& &&\\
&t_2y_3y_2&\\
&&y_3y_2^{-1}\\
&&&1
\end{pmatrix}\begin{pmatrix}
1&h&l&m\\
&1&p&q\\
&&1&r\\
&&&1
\end{pmatrix}.
\end{align*}	

Then Proposition \ref{12} follows from induction: the integral over $y_1$ and $k$ contributes the $L$-factor $\Lambda(2s,\tau^2)\Lambda(4s,\tau^4);$ and $y_3$ runs over a compact set, thus the integral over $y_3$ contributes an entire function; the only thing left is the contribution from integration over $y_2,$ which can be reduced (by \eqref{0003} and Fourier expansion of $\K(x,y)$) to the geometric side of Jacquet-Zagier's work \cite{JZ87} in $\GL(2)$ case. 
\end{proof}

However, neither $I_{\infty,0}^{(2)}(s)$ nor $I_{\infty,1}^{(2)}(s)$ converges for any $s\in\mathbb{C}.$ Since the contribution from $\K_{\infty,0}^{(2)}(x,y)$ has been handled in \eqref{0033}, we only need to deal with the contribution from $\K_{\infty,1}^{(2)}(x,y).$ In fact, we will see in the below, $\K_{\infty,1}^{(2)}(x,x),$ in conjunction with some singular parts of $\K_{\infty}^{(3)}(x,x),$ will be canceled by the singular part of $\K_{\Sin,2}(x,x)=\K_{\Geo,\Sin,2}(x,x)-\K_{\infty,\Sin,2}^{(4)}(x,x).$ 

\medskip

\subsubsection{Contributions Related to $\K_{\Sin,1}(x,x),$ $\K_{\Sin,2}(x,x)$ and $\K_{\infty}^{(3)}(x,x)$}
By Bruhat decomposition, 
\begin{equation}\label{0037}
Pw_3P=Sw_3S\sqcup Sw_3Sw_2S\sqcup Sw_2Sw_3S\sqcup Sw_2Sw_3Sw_2S.
\end{equation}

On the other hand, one can verify that 
\begin{align*}
\int_{[N_{S}]}\sum_{\delta\in R_{2}(F)\backslash R_{3}(F)}\int_{[V_3']}\int_{[V_3]}\sum_{\gamma\in X}\varphi((uu'n\delta x)^{-1}\gamma nx)du\theta(u')du'dn=0,
\end{align*}
where $N_S$ is the unipotent subgroup of $S,$ and $X=Sw_3Sw_2S\sqcup Sw_2Sw_3S.$ Hence, we only need to consider the contribution from $\gamma\in Sw_3S\sqcup Sw_2Sw_3Sw_2S.$ Let 
\begin{align*}
&\K_{\infty,1}^{(3)}(x,x)=\sum_{\delta}\int_{[\widetilde{V}_3']}\int_{[\widetilde{V}_3]}\sum_{\gamma\in Sw_3S}\varphi(x^{-1}\delta^{-1}u^{-1}v^{-1}\gamma \delta x)du\theta(v)dv,\\
&\K_{\infty,2}^{(3)}(x,x)=\sum_{\delta}\int_{[\widetilde{V}_3']}\int_{[\widetilde{V}_3]}\sum_{\gamma\in Sw_2Sw_3Sw_2S}\varphi(x^{-1}\delta^{-1}u^{-1}v^{-1}\gamma \delta x)du\theta(v)dv,
\end{align*}
where $\delta$ runs through $R_{2}(F)\backslash R_{3}(F).$ Denote also by 
\begin{align*}
\K_{\Sin,1}(x;y)=\sum_{\lambda\in S_0\backslash P_0}\K_{\Geo,\Sin,1}(\lambda x,y\lambda x)-\sum_{\lambda\in S_0\backslash P_0}\K_{\infty,\Sin,1}^{(4)}(\lambda x,y\lambda x).
\end{align*}
Then $\K_{\Sin,1}(x;y)$ is a Schwartz function on $S_0(F)\backslash R_3(\mathbb{A}_F).$ Hence, we can apply Fourier expansion to $\K_{\Sin,1}(x;y)$ and evaluate at $y=I_4$ to obtain
\begin{align*}
\K_{\Sin,1}(x,x)=\K_{\Sin,1}(x;I_4)=\K_{\Sin,1}^{(1)}(x,x)+\K_{\Sin,1}^{(2)}(x,x),
\end{align*}
where we denote by $T_3=\diag(I_2,\GL_1,1),$ and
\begin{align*}
\K_{\Sin,1}^{(1)}(x,x)&=\sum_{\lambda\in T_3(F)B_0(F)\backslash P_0(F)}\int_{[N_P\backslash N]}\K_{\Geo,\Sin,1}(u\lambda x,\lambda x)\theta(u)du;\\
\K_{\Sin,1}^{(2)}(x,x)&=\sum_{\lambda\in T_3(F)N(F)\backslash P_0(F)}\int_{[N_1]}\int_{[N_{21}]}\K_{\Geo,\Sin,1}(uv\lambda x,\lambda x)\theta(u)dudv.
\end{align*}

\begin{lemma}\label{30}
Let notation be as before. Then the distribution 
\begin{align*}
I_{\Sin,1}^{(2)}(s)=\int_{Z_G(\mathbb{A}_F)P_0(F)\backslash G(\mathbb{A}_F)}\K_{\Sin,1}^{(2)}(x,x)f(x,s)dx
\end{align*}
converges absolutely when $\Re(s)>1.$ Moreover, $I_{\Sin,1}^{(2)}(s)$ admits a meromorphic continuation to $s\in \mathbb{C}$ such that 
\begin{equation}\label{32}
I_{\Sin,1}^{(2)}(s)\sim \Lambda(s,\tau)^2\Lambda(4s,\tau^4).
\end{equation}
\end{lemma}
\begin{proof}
This can be reduced to the treatment of $I_{\Sin,1,0}^{(1,1)}(s)$ and $I_{\Sin,1,\neq 0}^{(1,1)}(s)$ in $\GL(3)$ case. In fact, a straightforward computation shows \eqref{32}.
\end{proof}

Recall that we have the decomposition \eqref{0037}. In this subsection, we further decompose the set $Sw_2Sw_3Sw_2S:$ 

\begin{lemma}\label{28.}
	Let notation be as before. Then $Sw_2Sw_3Sw_2S$ is equal to
	\begin{equation}\label{0018}
	(Bw_2w_3w_2N\sqcup Bw_2w_3w_2w_1N\sqcup Bw_2w_1w_3w_2N\sqcup Bw_2w_1w_3w_2w_1N)^{B\backslash S}.
	\end{equation}
	Moreover, the set $Bw_2w_3w_2N\sqcup Bw_2w_3w_2w_1\sqcup Bw_2w_1w_3w_2N\sqcup Bw_2w_1w_3w_2w_1$ consists of representatives under the conjugation of $B(F)\backslash S(F).$
\end{lemma}

\begin{proof}
	Since $S=B\sqcup Bw_1N,$ we have $Sw_2Sw_3Sw_2S=Bw_2w_3w_2N\sqcup Bw_1w_2w_3w_2N\sqcup Bw_2w_3w_2w_1\sqcup Bw_1w_2w_3w_2w_1N\sqcup Bw_2w_1w_3w_2N\sqcup Bw_2w_1w_3w_2w_1\sqcup Bw_1w_2w_1w_3w_2N\sqcup Bw_1w_2w_1w_3w_2w_1N.$ Noting that $B\backslash S=\{1\}\sqcup w_1N_1,$ by \eqref{20} we deduce that 
	\begin{align*}
	(Bw_2w_3w_2N\sqcup Bw_2w_3w_2w_1N\sqcup Bw_2w_1w_3w_2N\sqcup Bw_2w_1w_3w_2w_1N)^{B\backslash S}.
	\end{align*} 
	is contained in $Sw_2Sw_3Sw_2S.$ Hence, it is sufficient to show that 
	\begin{align*}
	Bw_1w_2w_3w_2N\sqcup Bw_1w_2w_3w_2w_1N\sqcup Bw_1w_2w_1w_3w_2N\sqcup Bw_1w_2w_1w_3w_2w_1N
	\end{align*}
	is contained in $(Bw_2w_3w_2N\sqcup Bw_2w_3w_2w_1\sqcup Bw_2w_1w_3w_2N\sqcup Bw_2w_1w_3w_2w_1)^{B\backslash S}.$
	\begin{enumerate}
		\item[(i).] Let $\gamma\in Bw_1w_2w_3w_2N.$ Then one can write 
		\begin{align*}
		\gamma=\begin{pmatrix}
		1&a &*&*\\
		& 1&*&*\\
		&&1&*\\
		&&&1
		\end{pmatrix}\begin{pmatrix}
		t_1& &&\\
		& t_2&&\\
		&&t_3&\\
		&&&t_4
		\end{pmatrix}w_1w_2w_3w_2\begin{pmatrix}
		1 &&&*\\
		& 1&*&*\\
		&&1&*\\
		&&&1
		\end{pmatrix}.
		\end{align*}
		
		Let $\delta=w_1\begin{pmatrix}
		1 &-a\\
		& 1&\\
		&&I_2
		\end{pmatrix}\in w_1N_1.$Then $\delta\gamma\delta^{-1}\in Bw_2w_3w_2w_1N.$ Hence 
		\begin{equation}\label{0014}
		Bw_1w_2w_3w_2N\subseteq (Bw_2w_3w_2w_1N)^{B\backslash S}.
		\end{equation}
		
		\item[(ii).] Let $\gamma\in Bw_1w_2w_3w_2w_1N.$ Then one can write 
		\begin{align*}
		\gamma=\begin{pmatrix}
		1&a &*&*\\
		& 1&*&*\\
		&&1&*\\
		&&&1
		\end{pmatrix}\begin{pmatrix}
		t_1& &&\\
		& t_2&&\\
		&&t_3&\\
		&&&t_4
		\end{pmatrix}w_1w_2w_3w_2w_1\begin{pmatrix}
		1 &b&*&*\\
		& 1&&*\\
		&&1&*\\
		&&&1
		\end{pmatrix}.
		\end{align*}
		Let $\delta=w_1\begin{pmatrix}
		1 &-a\\
		& 1&\\
		&&I_2
		\end{pmatrix}\in w_1N_1.$ If $a+b=0,$ then $\delta\gamma\delta^{-1}\in Bw_2w_3w_2N;$ if $a+b\neq 0,$ then $\delta\gamma\delta^{-1}\in Bw_2w_3w_2w_1N.$ In all, we have 
		\begin{equation}\label{0015}
		Bw_1w_2w_3w_2w_1N\subseteq (Bw_2w_3w_2N\sqcup Bw_2w_3w_2w_1N)^{B\backslash S}.
		\end{equation}

		\item[(iii).] Let $\gamma\in Bw_1w_2w_1w_3w_2N.$ Then one can write 
		\begin{align*}
		\gamma=\begin{pmatrix}
		1&a &*&*\\
		& 1&*&*\\
		&&1&*\\
		&&&1
		\end{pmatrix}\begin{pmatrix}
		t_1& &&\\
		& t_2&&\\
		&&t_3&\\
		&&&t_4
		\end{pmatrix}w_1w_2w_1w_3w_2\begin{pmatrix}
		1 &&*&*\\
		& 1&*&*\\
		&&1&*\\
		&&&1
		\end{pmatrix}.
		\end{align*}
		
		Let $\delta=w_1\begin{pmatrix}
		1 &-a\\
		& 1&\\
		&&I_2
		\end{pmatrix}.$ Then $\delta\gamma\delta^{-1}\in Bw_2w_1w_3w_2w_1N;$ namely,
		\begin{equation}\label{0016}
		Bw_1w_2w_3w_2w_1N\subseteq (Bw_2w_1w_3w_2w_1N)^{B\backslash S}.
		\end{equation}
		
		\item[(iv).] Let $\gamma\in Bw_1w_2w_1w_3w_2w_1N.$ Then one can write 
		\begin{align*}
		\gamma=\begin{pmatrix}
		1&a &*&*\\
		& 1&*&*\\
		&&1&*\\
		&&&1
		\end{pmatrix}\begin{pmatrix}
		t_1& &&\\
		& t_2&&\\
		&&t_3&\\
		&&&t_4
		\end{pmatrix}w_1w_2w_1w_3w_2w_1\begin{pmatrix}
		1 &b&*&*\\
		& 1&*&*\\
		&&1&*\\
		&&&1
		\end{pmatrix}.
		\end{align*}
		
		Let $\delta=w_1\begin{pmatrix}
		1 &-a\\
		& 1&\\
		&&I_2
		\end{pmatrix}.$ If $a+b=0,$ then $\delta\gamma\delta^{-1}\in Bw_2w_1w_3w_2N;$ if $a+b\neq 0,$ then $\delta\gamma\delta^{-1}\in Bw_2w_1w_3w_2w_1N.$ In all, we have 
		\begin{equation}\label{0017}
		Bw_1w_2w_3w_2w_1N\subseteq (Bw_2w_1w_3w_2N\sqcup Bw_2w_1w_3w_2w_1N)^{B\backslash S}.
		\end{equation}
	\end{enumerate}
	
	One then deduces from \eqref{0014}, \eqref{0015}, \eqref{0016} and \eqref{0017} that $Sw_2Sw_3Sw_2S$ is equal to \eqref{0018}. Also, by \eqref{20}, supposing 
	\begin{align*}
	\gamma_1, \gamma_2\in Bw_2w_3w_2N\sqcup Bw_2w_3w_2w_1N\sqcup Bw_2w_1w_3w_2N\sqcup Bw_2w_1w_3w_2w_1N,
	\end{align*}
	and $\gamma_1\in \gamma_2^{B\backslash S},$ then $\gamma_1$ and $\gamma_2$ must lie in the same Bruhat cell. However, by uniqueness of Bruhat normal form and \eqref{20}, this cannot happen unless $\gamma_1=\gamma_2.$ Hence, Lemma \ref{28.} follows.
\end{proof}
\medskip 

According to  Lemma \ref{28.}, we can set $\mathcal{B}_1(F)=(Bw_2w_3w_2N)^{B(F)\backslash S(F)},$ $\mathcal{B}_2(F)=(Bw_2w_3w_2w_1N)^{B(F)\backslash S(F)},$ $\mathcal{B}_3(F)=(Bw_2w_1w_3w_2N)^{B(F)\backslash S(F)},$ and let $\mathcal{B}_4(F)=(Bw_2w_1w_3w_2w_1N)^{B(F)\backslash S(F)}.$ Then we obtain a refined decomposition  $\K_{\infty,2}^{(3)}(x,x)=\K_{\infty,2;1}^{(3)}(x,x)+\K_{\infty,2;2}^{(3)}(x,x)+\K_{\infty,2;3}^{(3)}(x,x)+\K_{\infty,2;4}^{(3)}(x,x),$ where 
\begin{align*}
&\K_{\infty,2;k}^{(3)}(x,x)=\sum_{\delta\in R_2(F)\backslash R_3(F)}\int_{[\widetilde{V}_3']}\int_{[\widetilde{V}_3]}\sum_{\gamma\in  \mathcal{B}_{k,0}(F)}\varphi(x^{-1}\delta^{-1}u^{-1}v^{-1}\gamma \delta x)du\theta(v)dv,
\end{align*}
where $\mathcal{B}_{k,0}(F)=Z_G(F)\backslash \mathcal{B}_{k,0}(F),$ and $1\leq k\leq 4.$ Let $B_2$ be the group consisting of nonsingular $4\times 4$ matrix of the form $\begin{pmatrix}
* &*&*&*\\
& *&*&*\\
&&1&*\\
&&&1
\end{pmatrix}.$ Then 
\begin{align*}
\K_{\infty,2;k}^{(3)}(x,x)=\sum_{\delta\in B_2(F)\backslash R_3(F)}\int_{[\widetilde{V}_3']}\int_{[\widetilde{V}_3]}\sum_{\gamma\in  \mathcal{B}_{k,0}^*(F)}\varphi(x^{-1}\delta^{-1}u^{-1}v^{-1}\gamma \delta x)du\theta(v)dv,
\end{align*}
where $\mathcal{B}_{1,0}^*(F)=B_0w_2w_3w_2N,$ $\mathcal{B}_{2,0}^*(F)=B_0w_2w_3w_2w_1N,$ $\mathcal{B}_{3,0}^*(F)=B_0w_2w_1w_3w_2N,$ and $\mathcal{B}_{4,0}^*(F)=B_0w_2w_1w_3w_2w_1N.$

\bigskip 

In conjunction with the contribution from $\K_{\Geo,\Sin,2}(x,y),$ we (formally) define 
\begin{align*}
I_{\Sin}^{\mix}(s)=\int_{X}\bigg\{\K_{\Sin,1}^{(1)}(x)+\K_{\Sin,2}(x)-\K_{\infty,1}^{(2)}(x)-\K_{\infty,2;1}^{(3)}(x)\bigg\}\cdot f(x,s)dx,
\end{align*}
where $X=Z_G(\mathbb{A}_F)P_0(F)\backslash G(\mathbb{A}_F).$ Then we have 
\begin{prop}\label{mix}
Let notation be as before. Then $I_{\Sin}^{\mix}(s)=0.$
\end{prop}

Let $\Phi=B_0w_3N\sqcup B_0w_2w_3N\sqcup B_0w_3w_2N\sqcup B_0w_2w_3w_2N.$ Let $Q$ be the standard parabolic subgroup of $\GL(4)$ of type $(1,3).$ Denote by $N_Q$ the unipotent of $Q.$ Let $H$ be the standard parabolic subgroup of $\GL(4)$ of type $(1,2,1).$ Set $H_0=Z_G\backslash H.$ Let
\begin{align*}
\Upsilon_{\infty}^{(2)}(x)&=\sum_{\delta\in R_{1}(F)\backslash R_{3}(F)}\int_{[V_2']}\int_{[V_2]}\sum_{\gamma\in \Phi}\varphi((uu'\delta x)^{-1}\gamma \delta x)du\theta(u')du';\\
\Upsilon_{\infty}^{(3)}(x)&=\sum_{\delta\in B_2(F)\backslash R_3(F)}\int_{[\widetilde{V}_3']}\int_{[\widetilde{V}_3]}\sum_{\gamma\in  \Phi-B_0w_3N}\varphi(x^{-1}\delta^{-1}u^{-1}v^{-1}\gamma \delta x)du\theta(v)dv;\\
\Upsilon_{\infty}^{(4)}(x)&=\sum_{\delta\in B_2(F)\backslash R_3(F)}\int_{[\widetilde{V}_3']}\int_{[\widetilde{V}_3]}\sum_{\gamma\in  B_0w_3N}\varphi(x^{-1}\delta^{-1}u^{-1}v^{-1}\gamma \delta x)du\theta(v)dv.
\end{align*}

Set $\Delta_{\Phi}(x;y)=\Delta_{\Phi}^{(1)}(x;y)-\Delta_{\Phi}^{(2)}(x;y),$ where for any set $S,$
\begin{align*}
\Delta_{S}^{(1)}(x;y)&=\sum_{\delta\in H_0(F)\backslash R_3(F)}\int_{[N_Q]}\sum_{\gamma\in S}\varphi(x^{-1}\delta^{-1}v^{-1}\gamma y\delta x)dv,\\
\Delta_{S}^{(2)}(x;y)&=\sum_{\delta\in H_0(F)\backslash R_3(F)}\int_{[N_Q]}\int_{[N_R]}\sum_{\gamma\in S}\varphi(x^{-1}\delta^{-1}u^{-1}v^{-1}\gamma y\delta x)dudv.
\end{align*}
\begin{lemma}\label{37}
Let notation be as before. Then 
\begin{equation}\label{36}
\Delta_{\Phi}(x;I_4)=\K_{\Sin,1}^{(1)}(x)+\K_{\Sin,2}(x)+\Upsilon_{\infty}^{(4)}(x).
\end{equation}
\end{lemma}
\begin{proof}
For fixed $x,$ the function $\Delta_{\Phi}(x;y)$ is a Schwartz function with respect to $y\in H_0(F)\backslash Q(\mathbb{A}_F).$ Thus we can apply Proposition \ref{Fourier} to $\Delta_{\Phi}(x;y)$ and evaluate at $y=I_4$ to obtain $\Delta_{\Phi}(x;I_4)=\Upsilon_{\infty}^{(2)}(x)+\Upsilon_{\infty}^{(3)}(x)+\Upsilon_{\infty}^{(4)}(x).$ 

\begin{claim}\label{40}
	Let notation be as before. Then 
	\begin{equation}\label{33}
	\Phi=(B_0(F)w_2w_3N(F)\sqcup B_0(F)w_3N(F))^{B_0(F)\backslash H_0(F)}.
	\end{equation}
\end{claim}

Since Claim \ref{40} follows from the proof of Lemma \ref{22}, we thus omit the proof. Then by \eqref{33} we conclude that 
\begin{equation}\label{35}
\Delta_{\Phi}(x;I_4)-\K_{\Sin,2}(x)=\Delta_{*}^{(1)}(x;I_4)-\Delta_{*}^{(2)}(x;I_4),
\end{equation}
where $\Delta_{*}^{(k)}(x;I_4)=\Delta_{(B_0(F)w_3N(F))^{B_0(F)\backslash H_0(F)}}^{(k)}(x;I_4),$ $1\leq k\leq 2.$ Explicitly, 
\begin{align*}
\Delta_{*}^{(1)}(x;I_4)&=\sum_{\delta\in B_0(F)\backslash R_3(F)}\int_{[N_Q]}\sum_{\gamma\in B_0(F)w_3N(F)}\varphi(x^{-1}\delta^{-1}v^{-1}\gamma \delta x)dv,\\
\Delta_{*}^{(2)}(x;I_4)&=\sum_{\delta\in B_0(F)\backslash R_3(F)}\int_{[N_Q]}\int_{[N_R]}\sum_{\gamma\in B_0(F)w_3N(F)}\varphi(x^{-1}\delta^{-1}u^{-1}v^{-1}\gamma \delta x)dudv.
\end{align*}

It then follows form Poisson summation that 
\begin{equation}\label{34}
\Delta_{*}^{(1)}(x;I_4)=\K_{\Sin,1}^{(1)}(x)+\Upsilon_{\infty}^{(4)}(x)+\Delta_{*}^{(2)}(x;I_4).
\end{equation}
Hence, \eqref{36} follows from \eqref{35} and \eqref{34}. 
\end{proof}

\begin{proof}[Proof of Proposition \ref{mix}]
Considering the compatibility of Bruhat normal forms and the generic character, we have
\begin{align*}
I_{\Sin}^{\mix}(s)=\int_{X}\bigg\{\K_{\Sin,1}^{(1)}(x)+\K_{\Sin,2}(x)+\Upsilon_{\infty}^{(4)}(x)-\Delta_{\Phi}(x;I_4)\bigg\}\cdot f(x,s)dx.
\end{align*}
Then $I_{\Sin}^{\mix}(s)=0,$ as a consequence of Lemma \ref{37}.
\end{proof}

\subsubsection{Contributions from $\K_{\infty,2;k}^{(3)}$} 
Let notation be as before. Let $2\leq k\leq 4.$ Define the distribution by 
\begin{align*}
I_{\infty,2;k}^{(3)}(s)=\int_{Z_G(\mathbb{A}_F)B_2(F)\backslash G(\mathbb{A}_F)}\K_{\infty,2;k}^{(3)}(x,x)f(x,s)dx.
\end{align*}
Write $X=Z_G(\mathbb{A}_F)B_2(F)\backslash G(\mathbb{A}_F).$ Then explicitly we have 
\begin{align*}
I_{\infty,2;k}^{(3)}(s)=\int_{X}\int_{[\widetilde{V}_3']}\int_{[\widetilde{V}_3]}\sum_{\gamma\in  \mathcal{B}_{k,0}^*(F)}\varphi(x^{-1}u^{-1}v^{-1}\gamma x)du\theta(v)dvf(x,s)dx.
\end{align*}

\begin{prop}\label{002}
	Let notation be as before. Then $I_{\infty,2;2}^{(3)}(s)$ admits a meromorphic continuation to the whole $s$-plane, and 
	\begin{equation}\label{003}
I_{\infty,2;2}^{(3)}(s)\sim \Lambda(s,\tau)\Lambda(4s,\tau^4).
	\end{equation}
	
\end{prop}
\begin{proof}
Let $\widetilde{w}=w_2w_3w_2w_1.$ For any $\gamma\in\mathcal{B}_{2,0}^*(F),$ we can write $\gamma$ uniquely as $\gamma=u_1tu_2,$ where $u_1\in N(F),$ $t=\diag(t_1,t_2,t_3,1)$ and $u_2\in N_{\widetilde{w}}(F).$ Note that 
\begin{align*}
\begin{pmatrix}
y_1^{-1}&&&\\
&y_2^{-1}&&\\
&&y_3^{-1}&\\
&&&1
\end{pmatrix}\widetilde{w}\begin{pmatrix}
y_1&&&\\
&y_2&&\\
&&y_3&\\
&&&1
\end{pmatrix}=y_2^{-1}\widetilde{w}\begin{pmatrix}
y_1y_2&&&\\
&y_2^2y_1^{-1}&&\\
&&y_2&\\
&&&1
\end{pmatrix}
\end{align*}
Then 
\begin{align*}
I_{\infty,2;2}^{(3)}(s)=&\int_{(\mathbb{A}_F^{\times})^3}\int_{K}\sum_{t_2\in F^{\times}}\int_{\mathbb{A}_F^{11}}\varphi\left(k^{-1}\lambda_{t_2,y_1,y_2,e,g}^{a,b,c,e,f}\cdot \begin{pmatrix}
1&a'&b'&c'\\
&1&&\\
&&1&g'\\
&&&1
\end{pmatrix}k\right)\theta(y_3 g)\\
&\quad\theta(y_3 g')\tau(y_1)\tau\omega(y_2)\tau(y_3)|y_1|^{s+3}|y_2|^{s-1}|y_3|^{s}f(k,s)dndkd^{\times}y_1 d^{\times}y_2 d^{\times}y_3,
\end{align*}
where $\widetilde{w}=w_1w_2w_3w_2w_1,$ $dn=dadb\cdots dg da'\cdots dg';$ and
\begin{align*}
\lambda_{t_2,y_1,y_2,e,g}^{a,b,c,e,f}=\begin{pmatrix}
1&a&b&c\\
&1&e&f\\
&&1&g\\
&&&1
\end{pmatrix}\widetilde{w}\begin{pmatrix}
y_1y_2 &&\\
&y_2^2y_1^{-1}t_2&&\\
&&y_2&\\
&&&1
\end{pmatrix}.
\end{align*}
Since $y_2$ and $y_1y_2$ runs over compact subsets of $\mathbb{A}_{F}^{\times},$ $y_1$ runs over some compact subset as well. We then conclude \eqref{003} from Tate's thesis.
\end{proof}

\begin{prop}\label{000}
	Let notation be as before. Then $I_{\infty,2;3}^{(3)}(s)\cdot \Lambda(s,\tau)^{-1}$ admits a holomorphic continuation when $\Re(s)>0$ and $s\neq 1.$ 
	
\end{prop}
\begin{proof}
Let $X=Z_G(\mathbb{A}_F)B_2(F)\backslash G(\mathbb{A}_F).$ By definition, we have 
\begin{align*}
I_{\infty,2;3}^{(3)}(s)=\int_{X}\int_{[\widetilde{V}_3']}\int_{[\widetilde{V}_3]}\sum_{\gamma\in  \mathcal{B}_{3,0}^*(F)}\varphi(x^{-1}u^{-1}v^{-1}\gamma x)du\theta(v)dvf(x,s)dx.
\end{align*} 
Let $\widetilde{w}=w_2w_3w_1w_2.$ Then by changing of variables we then have 
\begin{align*}
I_{\infty,2;3}^{(3)}(s)=\int_{Z_G(\mathbb{A}_F)N(F)\backslash G(\mathbb{A}_F)}\int_{[\widetilde{V}_3']}\int_{[\widetilde{V}_3]}\sum_{\gamma}\varphi(x^{-1}u^{-1}v^{-1}\gamma x)du\theta(v)dvf(x,s)dx,
\end{align*}
where $\gamma=u_1\widetilde{w}tu_2,$ with $t=\diag(t_1,t_2,1,1)\in \diag(F^{\times}\backslash (F^{\times})^2, F^{\times},1,1);$ and $u_1, u_2\in N_2(F)\backslash N(F).$ Applying Iwasawa decomposition we then obtain 
\begin{align*}
I_{\infty,2;3}^{(3)}(s)=\int_{(\mathbb{A}_F^{\times})^3}\int_{\mathbb{A}_F^{10}}\sum_{t_1\in F^{\times}\backslash (F^{\times})^2}\sum_{t_2\in F^{\times}}\varphi(\cdots)|a_1|^{3s+3}|a_2|^{2s+4}|a_3|^{s+1}\theta(a_1g)\theta(a_1g'),
\end{align*}
where the first ellipsis represents the expression 
\begin{equation}\label{53}
\begin{pmatrix}
1&&b&c\\
&1&d&e\\
&&1&g\\
&&&1
\end{pmatrix}\widetilde{w}\begin{pmatrix}
a_2a_3t_1&&&\\
&a_1a_2t_2&&\\
&&a_2^{-1}a_3^{-1}&\\
&&&a_1^{-1}a_2^{-1}
\end{pmatrix}\begin{pmatrix}
1&&b'&c'\\
&1&d'&e'\\
&&1&g'\\
&&&1
\end{pmatrix}.
\end{equation}

Then we rewrite \eqref{53} and apply a change of variables to see $I_{\infty,2;3}^{(3)}(s)$ becomes
\begin{align*}
\int_{(\mathbb{A}_F^{\times})^3}\int_{\mathbb{A}_F^{10}}\sum_{t_1}\sum_{t_2}\varphi(u_1w_2w_3u_2w_1 au_3w_2u_4)|a_1|^{2s+2}|a_2|^{2s+2}|a_3|^{s+1}\theta(a_1g)\theta(a_1g'),
\end{align*}
where 
\begin{align*}
&u_1=\begin{pmatrix}
1&&&c\\
&1&d&e\\
&&1&g\\
&&&1
\end{pmatrix},\ u_2=\begin{pmatrix}
1&b&&\\
&1&&\\
&&1&\\
&&&1
\end{pmatrix},\ u_3=\begin{pmatrix}
1&b'&&\\
&1&&\\
&&1&\\
&&&1
\end{pmatrix}\\
&a=\begin{pmatrix}
a_2a_3t_1&&&\\
&a_2^{-1}a_3^{-1}&&\\
&&a_1a_2t_2&\\
&&&a_1^{-1}a_2^{-1}
\end{pmatrix},\ u_4=\begin{pmatrix}
1&&&c'\\
&1&d'&e'\\
&&1&g'\\
&&&1
\end{pmatrix}.
\end{align*}

From this expression, the analytic behavior of $I_{\infty,2;3}^{(3)}(s)$ can be detected via Jacquet-Zagier trace formula on $\GL(2).$ The contribution from $a_3$ and $t_1$ can be computed by the lemma in Sec. 2.4 of \cite{JZ87}, and can be further realized as a finite sum of intertwining operators; the contribution from $a_2$ and $t_2$ can be handles via Fourier expansion, the same as Proposition \ref{12}. As a consequence, $I_{\infty,2;3}^{(3)}(s)$ converges absolutely when $\Re(s)>1;$ and $I_{\infty,2;3}^{(3)}(s)\cdot \Lambda(s,\tau)^{-1}$ admits a meromorphic continuation when $\Re(s)>0,$ with the only possible pole at $s=1.$ 
\end{proof}

\begin{prop}\label{0004}
Let notation be as before. Then $I_{\infty,2;4}^{(3)}(s)\cdot \Lambda(s,\tau)^{-1}$ admits a meromorphic continuation to $\Re(s)>0,$ with no pole outside $s=1.$ 
\end{prop}
\begin{proof}
Let $X=Z_G(\mathbb{A}_F)B_2(F)\backslash G(\mathbb{A}_F).$ By definition, we have 
\begin{align*}
I_{\infty,2;4}^{(3)}(s)=\int_{X}\int_{[\widetilde{V}_3']}\int_{[\widetilde{V}_3]}\sum_{\gamma\in  \mathcal{B}_{4,0}^*(F)}\varphi(x^{-1}u^{-1}v^{-1}\gamma x)du\theta(v)dvf(x,s)dx.
\end{align*} 
Let $\widetilde{w}=w_2w_3w_1w_2w_1.$ Then by changing of variables we then have 
\begin{align*}
I_{\infty,2;4}^{(3)}(s)=\int_{Z_G(\mathbb{A}_F)N(F)\backslash G(\mathbb{A}_F)}\int_{[\widetilde{V}_3']}\int_{[\widetilde{V}_3]}\sum_{\gamma}\varphi(x^{-1}u^{-1}v^{-1}\gamma x)du\theta(v)dvf(x,s)dx,
\end{align*}
where $\gamma=u_1\widetilde{w}tu_2,$ with $u_1\in N_2(F)\backslash N(F),$ $u_2\in N(F),$ and $t=\diag(t_1,t_2,1,1)\in \diag(F^{\times}\backslash (F^{\times})^2, F^{\times},1,1).$ Applying Iwasawa decomposition we then obtain 
\begin{align*}
I_{\infty,2;4}^{(3)}(s)=\int_{(\mathbb{A}_F^{\times})^3}\int_{\mathbb{A}_F^{11}}\sum_{t_1\in F^{\times}\backslash (F^{\times})^2}\sum_{t_2\in F^{\times}}\varphi(\cdots)|a_1|^{3s+3}|a_2|^{2s+4}|a_3|^{s+2}\theta(a_1g)\theta(a_1g'),
\end{align*}
where the first ellipsis represents the expression 
\begin{equation}\label{52}
\begin{pmatrix}
1&&b&c\\
&1&d&e\\
&&1&g\\
&&&1
\end{pmatrix}\widetilde{w}\begin{pmatrix}
a_1a_2a_3t_1&&&\\
&a_2t_2&&\\
&&a_2^{-1}a_3^{-1}&\\
&&&a_1^{-1}a_2^{-1}
\end{pmatrix}\begin{pmatrix}
1&a'&b'&c'\\
&1&d'&e'\\
&&1&g'\\
&&&1
\end{pmatrix}.
\end{equation}

Then we rewrite \eqref{52} and apply a change of variables to see $I_{\infty,2;4}^{(3)}(s)$ becomes
\begin{align*}
\int_{(\mathbb{A}_F^{\times})^3}\int_{\mathbb{A}_F^{11}}\sum_{t_1}\sum_{t_2}\varphi(u_1w_2w_3w_1u_2w_2 au_3w_1u_4)|a_1|^{2s+2}|a_2|^{2s+2}|a_3|^{s+1}\theta(a_1g)\theta(a_1g'),
\end{align*}
where 
\begin{align*}
&u_1=\begin{pmatrix}
1&&b&c\\
&1&d&\\
&&1&g\\
&&&1
\end{pmatrix},\ u_2=\begin{pmatrix}
1&d'&&\\
&1&e&\\
&&1&\\
&&&1
\end{pmatrix},\ u_3=\begin{pmatrix}
1&&&\\
&1&b'&\\
&&1&\\
&&&1
\end{pmatrix}\\
&a=\begin{pmatrix}
a_2&&&\\
&a_2a_3t_1&&\\
&&a_2^{-1}a_3^{-1}a_1&\\
&&&a_1^{-1}a_2^{-1}t_2
\end{pmatrix},\ u_4=\begin{pmatrix}
1&a'&&c'\\
&1&&e'\\
&&1&g'\\
&&&1
\end{pmatrix}.
\end{align*}

From this expression, the analytic behavior of $I_{\infty,2;4}^{(3)}(s)$ can be deduced from Jacquet-Zagier trace formula on $\GL(2).$ Precisely, the contribution from $a_3$ and $t_1$ can be computed by the lemma in Sec. 2.4 of \cite{JZ87}, and eventually be realized as a finite sum of intertwining operators; the contribution from $a_2$ and $t_2$ can be handles via Fourier expansion, the same as Proposition \ref{12}. As a consequence, $I_{\infty,2;4}^{(3)}(s)$ converges absolutely when $\Re(s)>1;$ and $I_{\infty,2;4}^{(3)}(s)\cdot \Lambda(s,\tau)^{-1}$ admits a meromorphic continuation when $\Re(s)>0,$ with the only possible pole at $s=1.$ 
\end{proof}

\section{Contributions from Spectral Side}\label{7.3}
In this section, we deal with the generic character distribution $I_{\infty}^{(1)}(s,\tau).$ By Theorem G in \cite{Yang19}, when $\Re(s)>1,$ $I_{\infty}^{(1)}(s,\tau)$ is equal to 
\begin{equation}\label{3}
\sum_{\chi}\sum_{P\in \mathcal{P}}\frac{1}{c_P}\sum_{\phi_1, \phi_2}\int_{\Lambda^*}\langle\mathcal{I}_P(\lambda,\varphi)\phi_2,\phi_1\rangle\int_{Y_G} W_{1}(x;\lambda)\overline{W_{2}(x;\lambda)}f(x,s)dxd\lambda,
\end{equation}
where $Y_G=Z_G(\mathbb{A}_F)N(\mathbb{A}_F)\backslash G(\mathbb{A}_F),$ $\chi$ runs over proper cuspidal data, i.e., $\chi$ is not of the form $\{(G,\pi)\};$ and $\phi_1,$ $\phi_2$ runs over an orthogonal basis $\mathfrak{B}_{P,\chi}$ of the representation space determined by $\chi.$ The sum converges absolutely. Particularly, as a function of $s,$ $I_{\infty}^{(1)}(s)$ is analytic when $\Re(s)>1.$ Moreover, when $\tau^k\neq 1,$ $1\leq k\leq n,$ then Theorem G in loc. cit. and functional equation give meromorphic continuation of $I_{\infty}^{(1)}(s)$ to the whole $s$-plane. 

However, for general $\tau,$ e.g., $\tau=1,$ the continuation of $I_{\infty}^{(1)}(s)$ is rather involved, since the function \eqref{3} is singular at every point on the boundary $\Re(s)=1.$ To continue $I_{\infty}^{(1)}(s)$ meromorphically to the whole plane, we will follow Section 8 in loc. cit., taking advantage of zero-free regions of Rankin-Selberg convolutions and estimates from analytic number theory. 

\subsection{Notation and Zero-free Region}
In this subsection, we introduce some notation used in Section 8 of \cite{Yang19}. Let $\Sigma_F$ be the set of places on $F.$ Recall that we fix the unitary character $\tau.$ Let $\mathcal{D}_{\tau}$ be a standard (open) zero-free region of $L_F(s,\tau)$ (e.g. ref. \cite{Bru06}). We fix such a $\mathcal{D}_{\tau}$ once for all. Let
\begin{equation}\label{R}
\mathcal{R}(1/2;\tau)^-:=\{s\in\mathbb{C}:\ 2s\in \mathcal{D}_{\tau}\}\supsetneq \{s\in\mathbb{C}:\ \Re(s)\geq 1/2\}.
\end{equation}

In Section \ref{3.2}, we will continue $I_{\infty}^{(1)}(s,\tau)$ to the open set $\mathcal{R}(1/2;\tau)^-.$ Invoking \eqref{R} with functional equation we then obtain a meromorphic continuation of $I_{\infty}^{(1)}(s)$ to the whole complex plane.

Let $G=\GL(3)$ or $\GL(4).$ Let $P$ be a standard parabolic subgroup of $G$ of type $(n_1,n_2,\cdots,n_r).$ Let $\mathfrak{X}_P$ be the subset of cuspidal data $\chi=\{(M,\sigma)\}$ such that $M=M_P=\diag(M_1,M_2,\cdots,M_r),$ where $M_i$ is $n_i$ by $n_i$ matrix, $1\leq i\leq r.$ We may write $\sigma=(\sigma_1,\sigma_2,\cdots,\sigma_r),$ where $\sigma_i\in\mathcal{A}_0(M_i(F)\backslash M_i(\mathbb{A}_F)).$ Let $\pi$ be a representation induced from $\chi=\{(M,\sigma)\}.$ 

For any $\boldsymbol{\lambda}=(\lambda_1,\lambda_2,\cdots,\lambda_r)\in i\mathfrak{a}_P^*/i\mathfrak{a}_G^*\simeq (i\mathbb{R})^{r-1},$ satisfying that $\lambda_1+\lambda_2+\cdots+\lambda_r=0,$ we let $\boldsymbol{\kappa}=(\kappa_1,\kappa_2,\cdots,\kappa_{r})\in \mathbb{C}^{r-1}$ be such that 
\begin{equation}\label{145}
\begin{cases}
\kappa_j=\lambda_j-\lambda_{j+1},\ 1\leq j\leq r-1,\\
\kappa_r=\lambda_1-\lambda_r=\kappa_1+\kappa_2+\cdots+\kappa_{r-1}.
\end{cases}
\end{equation}

Then we have a bijection $i\mathfrak{a}_P^*/i\mathfrak{a}_G^*\xleftrightarrow[]{1:1}i\mathfrak{a}_P^*/i\mathfrak{a}_G^*,$ $\boldsymbol{\lambda}\mapsto\boldsymbol{\kappa}$ given by \eqref{145}, which induces a change of coordinates with $d\boldsymbol{\lambda}=m_Pd\boldsymbol{\kappa},$ where $m_P$ is an absolute constant (the determinant of the transform \eqref{145}). So that we can write $\boldsymbol{\lambda}=\boldsymbol{\lambda}(\boldsymbol{\kappa}).$ Let \begin{align*}
R_{\varphi}(s,\lambda;\phi_2)=\sum_{\phi_1\in \mathfrak{B}_{P,\chi}}\langle\mathcal{I}_P(\lambda,\varphi)\phi_1,\phi_2\rangle\cdot\frac{\Psi(s,W_1,W_2;\lambda)}{\Lambda(s,\pi_{\lambda}\otimes\tau\times\widetilde{\pi}_{-\lambda})},\ \Re(s)>1,
\end{align*}
where $\Lambda(s,\pi_{\lambda}\otimes\tau\times\widetilde{\pi}_{-\lambda})$ is the complete $L$-function, defined by $\prod_{v\in\Sigma_F}L_v(s,\pi_{\lambda,v}\otimes\tau_v\times\widetilde{\pi}_{-\lambda,v});$ and $\Psi(s,W_1,W_2;\lambda)=\int_{Y_G} W_{1}(x;\lambda)\overline{W_{2}(x;\lambda)}f(x,s)dxd\lambda$ is the Rankin-Selberg period (see Section 6 of \cite{Yang19} for basic analytic properties).
\medskip

Then we can write $R_{\varphi}(s,\boldsymbol{\lambda};\phi)=R_{\varphi}(s,\boldsymbol{\kappa};\phi)$ and $\Lambda(s,\pi_{\lambda}\otimes\tau\times\widetilde{\pi}_{-\lambda})=\Lambda(s,\pi_{\boldsymbol{\kappa}}\otimes\tau\times\widetilde{\pi}_{-\boldsymbol{\kappa}}).$ Recall that if $v\in \Sigma_{F,fin}$ is a finite place such that $\pi_v$ is unramified and $\Phi_v=\Phi_v^{\circ}$ is the characteristic function of $G(\mathcal{O}_{F,v}).$ Assume further that $\phi_{1,v}=\phi_{2,v}=\phi_v^{\circ}$ be the unique $G(\mathcal{O}_{F,v})$-fixed vector in the space of $\pi_v$ such that $\phi_v^0(e)=1.$ Then $R_v(s,W_{1,v},W_{2,v};\boldsymbol{\lambda})=R_v(s,W_{1,v},W_{2,v};\boldsymbol{\kappa})$ is equal to 
\begin{equation}\label{146}
\prod_{1\leq i<r}\prod_{i< j\leq r}L_v(1+\boldsymbol{\kappa}_{i,j},\sigma_{i,v}\times\widetilde{\sigma}_{j,v})^{-1}\cdot L_v(1-\boldsymbol{\kappa}_{i,j},\widetilde{\sigma}_{i,v}\times\sigma_{j,v})^{-1},
\end{equation}
where $\boldsymbol{\kappa}_{i,j}=\boldsymbol{\kappa}_{i}+\cdots+\boldsymbol{\kappa}_{j-1}.$ By the $K$-finiteness of $\varphi,$ there exists a finite set $S_{\varphi,\tau,\Phi}$ of nonarchimedean places such that for any $\pi$ from some cuspidal datum $\chi\in\mathfrak{X}_P,$ $R_v(s,W_{1,v},W_{2,v};\boldsymbol{\kappa})$ is equal to the formula in \eqref{146}. Then according to Proposition 43 and Proposition 50 in loc. cit., we see that, when $\Re(s)>0,$ $R_v(s,W_{1,v},W_{2,v};\boldsymbol{\kappa})$ is independent of $s$ for all but finitely many places $v.$ Therefore, as a function of $s,$ $R_{\varphi}(s,\boldsymbol{\kappa};\phi)$ is a finite product of holomorphic function in $\Re(s)>0;$ for any given $s$ such that $\Re(s)>0,$ as a complex function of multiple variables with respect to $\boldsymbol{\kappa},$ $R_{\varphi}(s,\boldsymbol{\kappa};\phi)$ has the property that $R_{\varphi}(s,\boldsymbol{\kappa};\phi)L_{S}(\boldsymbol{\kappa},\pi,\widetilde{\pi})$ is holomorphic, where $L_S(\boldsymbol{\kappa},\pi,\widetilde{\pi})$ is denoted by the meromorphic function 
\begin{align*}
\prod_{1\leq i<r}\prod_{i< j\leq r}\prod_{v\in S_{\varphi,\tau,\Phi}}L_v(1+\boldsymbol{\kappa}_{i,j},\sigma_{i,v}\times\widetilde{\sigma}_{j,v})\cdot L_v(1-\boldsymbol{\kappa}_{i,j},\widetilde{\sigma}_{i,v}\times\sigma_{j,v}).
\end{align*}
Hence $R_{\varphi}(s,\boldsymbol{\kappa};\phi)$ is holomorphic in some domain $\mathcal{D}$ if $L_{S}(\boldsymbol{\kappa},\pi,\widetilde{\pi})$ is nonvanishing in $\mathcal{D}.$ Now we are picking up such a zero-free region $\mathcal{D}$ explicitly. 

Let $1\leq m,m'\leq n$ be two integers. Let $\sigma\in\mathcal{A}_0(\GL_m(F)\backslash \GL_{m}(\mathbb{A}_F))$ and $\sigma'\in\mathcal{A}_0(\GL_{m'}(F)\backslash \GL_{m'}(\mathbb{A}_F)).$ Fix $\epsilon_0>0.$ For any $c'>0,$ let $\mathcal{D}_{c'}(\sigma,\sigma')$ be 
\begin{equation}\label{A}
\Bigg\{\kappa=\beta+i\gamma:\ \beta\geq 1-c'\cdot\Big[\frac{(C(\sigma)C(\sigma'))^{-2(m+m')}}{(|\gamma|+3)^{2mm'[F:\mathbb{Q}]}}\Big]^{\frac12+\frac{1}{2(m+m')}-\epsilon_0}\Bigg\},
\end{equation}
if $\sigma'\ncong\widetilde{\sigma};$ and let $\mathcal{D}_{c'}(\sigma,\sigma')$ denote by the region 
\begin{equation}\label{B}
\Bigg\{\kappa=\beta+i\gamma:\ \beta\geq 1-c'\cdot\Big[\frac{(C(\sigma))^{-8m}}{(|\gamma|+3)^{2mm^2[F:\mathbb{Q}]}}\Big]^{-\frac78+\frac{5}{8m}-\epsilon_0}\Bigg\},
\end{equation}
if $\sigma'\simeq \widetilde{\sigma}.$ According to \cite{Bru06} and the Appendix of \cite{Lap13}, there exists a constant $c_{m,m'}>0$ depending only on $m$ and $m',$ such that $L(\boldsymbol{\kappa},\sigma\times\sigma')$ does not vanish in $\boldsymbol{\kappa}=(\kappa_1,\cdots,\kappa_{r})\in \mathcal{D}_{c_{m,m'}}(\sigma,\sigma')\times\cdots\times \mathcal{D}_{c_{m,m'}}(\sigma,\sigma').$ Let $c=\min_{1\leq m,m'\leq n}c_{m,m'}$ and  $\mathcal{C}(\sigma,\sigma')$ be the boundary of $\mathcal{D}_{c}(\sigma,\sigma').$ We may assume that $c$ is small such that the curve $\mathcal{C}(\sigma,\sigma')$ lies in the strip $1-1/(n+4)<\Re(\kappa_j)<1,$ $1\leq j\leq r.$ Fix such a $c$ henceforth. Note that by our choice of $c,$ $L(\boldsymbol{\kappa},\sigma\times\sigma')$ is nonvanishing in $\mathcal{D}_c(\sigma,\sigma')\times \cdots\times \mathcal{D}_c(\sigma,\sigma')$ for any $1\leq m,m'\leq n.$ For $v\in S_{\varphi,\tau,\Phi},$ we have that 
\begin{align*}
\big|L_v(\boldsymbol{\kappa},\sigma_{v}\times{\sigma'}_{v})^{-1}\big|\leq\prod_{i=1}^r\prod_{j=1}^r\left(1+q_v^{1-\frac1{m^2+1}-\frac1{m'^2+1}}\right)^{n_i+n_j}<\infty,
\end{align*}
for any $\boldsymbol{\kappa}$ such that each $\Re(\kappa_j)\geq 0,$ $1\leq j\leq r.$ Let $L_{S}(\boldsymbol{\kappa},\sigma\times\sigma')=L(\boldsymbol{\kappa},\sigma\times\sigma')\prod_{v\in S_{\varphi,\tau,\Phi}}L_v(\boldsymbol{\kappa},\sigma_{v}\times{\sigma'}_{v})^{-1}.$ Then $L_{S}(\boldsymbol{\kappa},\sigma\times\sigma')$ is nonvanishing in $\mathcal{D}_c(\sigma,\sigma')\times \cdots\times \mathcal{D}_c(\sigma,\sigma')$ for any $1\leq m,m'\leq n.$

Let $\chi\in \mathfrak{X}_P$ and $\pi=\Ind_{P(\mathbb{A}_F)}^{G(\mathbb{A}_F)}(\sigma_1,\sigma_2,\cdots,\sigma_r)\in\chi.$ For any $\epsilon\in (0,1]$ we set
\begin{align*}
\mathcal{D}_{\chi}(\epsilon)=\bigcap_{1\leq i\leq r}\bigcap_{i<j\leq r}\Big\{\kappa\in\mathbb{C}:\ \Re(\kappa)\geq0,\ 1-\kappa\in \mathcal{D}_{c\epsilon}(\sigma_i,\sigma_j)\Big\}.
\end{align*}
Also, for $\epsilon=0,$ we set $\mathcal{D}_{\chi}(\epsilon)=\big\{\kappa\in\mathbb{C}:\ \Re(\kappa)\geq0\big\}.$ Then by the above discussion, as a function of $\boldsymbol{\kappa},$ $L_S(\boldsymbol{\kappa},\pi,\widetilde{\pi})$ is nonzero in the region $\mathcal{D}_{\chi}(\boldsymbol{\epsilon})=\big\{\boldsymbol{\kappa}=(\kappa_1,\cdots,\kappa_r)\in\mathbb{C}^r:\ \kappa_l\in \mathcal{D}_{\chi}(\epsilon_l)\big\},$ where $\boldsymbol{\epsilon}=(\epsilon_1,\cdots,\epsilon_r)\in [0,1]^r.$ We can write $\mathcal{D}_{\chi}(\boldsymbol{\epsilon})$ as a product space $\mathcal{D}_{\chi}(\boldsymbol{\epsilon})=\prod_{l=1}^r\mathcal{D}_{\chi}(\epsilon_l),$ and let $\partial\mathcal{D}_{\chi}(\epsilon_l)$ be the boundary of $\mathcal{D}_{\chi}(\epsilon_l).$ Then when $\epsilon_l>0,$ $\partial\mathcal{D}_{\chi}(\epsilon_l)$ has two connected components and one of which is exactly the imaginary axis. Let $\mathcal{C}_{\chi}(\epsilon_l)$ be the other component, which is a continuous curve, where $0\leq\epsilon_l\leq 1.$ When $\epsilon_l=0,$ let $\mathcal{C}_{\chi}(\epsilon_l)$ be the maginary axis. Set $\mathcal{C}_{\chi}(\boldsymbol{\epsilon})=\mathcal{C}_{\chi}(\epsilon_1)\times \cdots\times \mathcal{C}_{\chi}(\epsilon_{r-1}),$ $0\leq \epsilon_l\leq 1,$ $1\leq l\leq r-1.$

Let $\boldsymbol{\epsilon}=(\epsilon_1,\cdots,\epsilon_{r-1})\in [0,1]^{r-1}.$ Then by the above construction, $R_{\varphi}(s,\boldsymbol{\kappa};\phi)$ is holomorphic in $\mathcal{D}_{\chi}(\boldsymbol{\epsilon}).$ Hence $R_{\varphi}(s,\boldsymbol{\kappa};\phi)\Lambda(s,\pi_{\boldsymbol{\kappa}}\otimes\tau\times\widetilde{\pi}_{-\boldsymbol{\kappa}})$ is holomorphic in $\mathcal{D}_{\chi}(\boldsymbol{\epsilon}).$ Moreover,  $L_S(\boldsymbol{\kappa},\pi,\widetilde{\pi})\neq0$ on $\mathcal{C}_{\chi}(\boldsymbol{\epsilon}),$ for any $\boldsymbol{\epsilon}=(\epsilon_1,\cdots,\epsilon_{r-1})\in [0,1]^{r-1}$ and any cuspidal datum $\chi\in \mathfrak{X}_P.$ Let $\Re(s)>1.$ For any $\phi\in \mathfrak{B}_{P,\chi}$ and $\boldsymbol{\epsilon}=(\epsilon_1,\cdots,\epsilon_{r-1})\in [0,1]^{r-1},$ let
\begin{align*}
J_{P,\chi}(s;\phi,\mathcal{C}_{\chi}(\boldsymbol{\epsilon}))=\int_{\mathcal{C}_{\chi}(\boldsymbol{\epsilon})} R_{\varphi}(s,\boldsymbol{\kappa};\phi)\Lambda(s,\pi_{\boldsymbol{\kappa}}\otimes\tau\times\widetilde{\pi}_{-\boldsymbol{\kappa}})d\boldsymbol{\kappa}.
\end{align*}
which is well defined because $J_{P,\chi}(s;\phi,\mathcal{C}_{\chi}(\boldsymbol{\epsilon}))=J_{P,\chi}(s;\phi,\mathcal{C}_{\chi}(\boldsymbol{0}))$ by Cauchy integral formula. Therefore, according to Theorem F in loc. cit.,
\begin{align*}
\sum_{P}\frac1{c_P}\sum_{\chi\in\mathfrak{X}_P}\sum_{\phi\in \mathfrak{B}_{P,\chi}}\int_{\mathcal{C}_{\chi}(\boldsymbol{\epsilon})} \big|R_{\varphi}(s,\boldsymbol{\kappa};\phi)\Lambda(s,\pi_{\boldsymbol{\kappa}}\otimes\tau\times\widetilde{\pi}_{-\boldsymbol{\kappa}})\big|d\boldsymbol{\kappa}<\infty
\end{align*}
for any $\Re(s)>1,$ $\boldsymbol{\epsilon}=(\epsilon_1,\cdots,\epsilon_{r-1})\in [0,1]^{r-1}.$

Let $\boldsymbol{\epsilon}=(\epsilon_1,\cdots,\epsilon_{r-1})\in [0,1]^{r-1}.$ For any $\beta\geq 1/2,$ we denote by 
\begin{equation}\label{149}
\mathcal{R}(\beta;\chi,\boldsymbol{\epsilon})=\Big\{s\in 1+\mathcal{D}_{\chi}(\boldsymbol{\epsilon})\Big\}\bigcup \Big\{s\in 1-\mathcal{D}_{\chi}(\boldsymbol{\epsilon})\Big\}.
\end{equation}
\bigskip

Let $s\in \mathcal{R}(1;\chi,\boldsymbol{\epsilon})$ and $1\leq m\leq r-1.$ Let $j_m, j_{m-1},\cdots,j_1$ be $m$ integers such that $1\leq j_m<\cdots<j_1\leq r-1.$ Consider the distribution:
\begin{align*}
\mathcal{I}_{m,\chi}(s,\tau):=\sum_{\phi\in \mathfrak{B}_{P,\chi}}\int_{\mathcal{C}}\cdots\cdots\int_{\mathcal{C}}\underset{\kappa_{j_m}=s-1}{\Res}\cdots\underset{\kappa_{j_1}=s-1}{\Res}\mathcal{F}(\boldsymbol{\kappa};s)\frac{d\kappa_{r-1}\cdots d\kappa_{1}}{d\kappa_{j_m}\cdots d\kappa_{j_1}},
\end{align*}
where $\mathcal{F}(\boldsymbol{\kappa};s)=\mathcal{F}(\boldsymbol{\kappa};s,P,\chi):=R_{\varphi}(s,\boldsymbol{\kappa};\phi)\Lambda(s,\pi_{\boldsymbol{\kappa}}\otimes\tau\times\widetilde{\pi}_{-\boldsymbol{\kappa}})$. Then each $\mathcal{I}_{m,\chi}(s,\tau)$ is meromorphic in $\mathcal{R}(1;\chi,\boldsymbol{\epsilon})$ with a possible pole at $s=1.$

Let $n\leq 4.$ Let $\chi\in\mathfrak{X}_P.$ Assume that the adjoint L-function $L(s,\sigma,\Ad\otimes \tau)$ is holomorphic inside the strip $0<\Re(s)<1$ for any cuspidal representation  $\sigma\in\mathcal{A}_0\left(\GL(k,\mathbb{A}_F)\right),$ and any $k\leq n-1.$ Then according to Theorem H in loc. cit., for any $0\leq m\leq r-1,$ the function 
	\begin{align*}
	\sum_{\phi\in \mathfrak{B}_{P,\chi}}{\mathcal{I}_{m,\chi}(s)},\quad \ s\in \mathcal{R}(1;\chi,\boldsymbol{\epsilon}),
	\end{align*} 
	admits a meromorphic continuation to the area $\mathcal{R}(1/2;\tau)^-$, with possible simple poles at $s\in\{1/2, 2/3,\cdots, (n-1)/n,1\},$ where $\mathcal{R}(1/2;\tau)^-$ is defined in \eqref{R}. Moreover, for any $3\leq k\leq n,$ if $L_F((k-1)/k,\tau)=0,$ then $s=(k-1)/k$ is not a pole.

Recall that we need to investigate the analytic behavior of the function
\begin{align*}
\mathcal{Z}_{m,*}(s,\tau)=\sum_P\frac1{c_P}\sum_{\chi\in\mathfrak{X}_P}\sum_{\phi\in \mathfrak{B}_{P,\chi}}{\mathcal{I}_{m,\chi}(s)}\cdot{\Lambda(s,\tau)}^{-1},
\end{align*} 
where the sum over standard parabolic subgroups $P$ is finite while the sum over cuspidal data $\chi$ is infinite. According to Theorem F in loc. cit., $\mathcal{Z}_*(s,\tau)$ converges absolutely and locally normally in the region $\Re(s)>1.$ Since each summand $\sum_{\phi\in \mathfrak{B}_{P,\chi}}{\mathcal{I}_{m,\chi}(s)}\cdot{\Lambda(s,\tau)}^{-1}$ admits a meromorphic continuation to the region $\mathcal{R}(1/2;\tau)^-,$ with possible simple poles at $s\in\{1/2, 2/3,\cdots, (n-1)/n\}$ and a pole of order at most $4$ at $s=1,$ we can consider (at least formally) the distribution   
\begin{equation}\label{I}
\mathcal{Z}_m(s,\tau)=\sum_P\frac1{c_P}\sum_{\chi\in\mathfrak{X}_P}\sum_{\phi\in \mathfrak{B}_{P,\chi}}{\widetilde{\mathcal{I}}_{m,\chi}(s)}\cdot{\Lambda(s,\tau)}^{-1},
\end{equation} 
where $\widetilde{\mathcal{I}}_{m,\chi}(s)$ is the continuation of $\mathcal{I}_{m,\chi}(s,\tau).$ Then we only need to show that $(s-1/2)(s-2/3)(s-3/4)(s-1)^4\mathcal{Z}(s)$ converges absolutely and locally normally inside the domain $\mathcal{R}(1/2;\tau)^-.$

\bigskip

\begin{thmx}\label{78}
Let notation be as before. Let $0\leq m\leq r-1.$ Then $\mathcal{Z}_m(s,\tau)$ admits a meromorphic continuation to the domain $\mathcal{R}(1/2;\tau)^-,$ where it has possible poles at $s=1/2$ and $s=1.$ Moreover, if $s=1/2$ is a pole, then it must be simple.
\end{thmx}

\bigskip 
\subsection{Generic Characters for $G$}\label{3.2}
Let $v$ be a nonarchimedean place of $F.$ Let $P$ be a standard parabolic subgroup of $G.$ Fix a Levi decomposition of $P=MN$ with $M$ containing the maximal splitting torus $\mathbb{G}_m^n.$ Let $\sigma_v$ be an irreducible admissible unitary representation of $M(F_v)$ and fix $\lambda\in \mathfrak{a}_P^*(\mathbb{C})=\mathfrak{a}_P^*\otimes \mathbb{C}.$ We shall use $I(\lambda,\sigma_v)$ to denote the induced representation 
\begin{align*}
I(\lambda,\sigma_v)=\Ind_{M(F_v)N(F_v)}^{G(F_v)}\sigma_v\otimes \exp\langle\lambda,H_M(\cdot)\rangle\otimes\boldsymbol{1}.
\end{align*} 
Since $F_v$ is nonarchimedean, the space $V(\lambda,\sigma_v)$ of $I(\lambda,\sigma_v)$ consists of the space of locally constant functions from $G(F_v)$ into the space $\mathcal{H}(\sigma_v)$ of $\sigma_v,$ such that 
\begin{align*}
h_{v,\lambda}(m_vn_vg_v)=\sigma_v(m_v)\exp\langle\lambda+\rho_P,H_{M,v}(m_v)\rangle h_{v,\lambda}(g_v),\ h_{v,\lambda}\in V(\lambda,\sigma_v).
\end{align*}
The group acts on $V(\lambda,\sigma_v)$ via the right regular action. Define the Whittaker function for the representation $I(\lambda,\sigma_v)$ as follows:
\begin{align*}
W_v(h_{v,\lambda};\sigma_v)=\int_{N(F_v)}h_{v,\lambda}(w_0n_v)\theta(n_v)dn_v.
\end{align*}
\begin{lemma}\label{60''}
Let notation be as above, then there exists a test function $h_v\in\mathcal{H}(\sigma_v)$ such that for any $\lambda\in \mathfrak{a}_P^*(\mathbb{C}),$ the Whittaker function 
\begin{align*}
W_v(h_{v};\lambda,\sigma_v)=\int_{N(F_v)}h_{v}(w_0n_v)\exp\langle\lambda+\rho_P,H_{M,v}(w_0n_v)\rangle\theta(n_v)dn_v\neq 0.
\end{align*}
\end{lemma}
\begin{proof}
To construct such a $h_v,$ we start with the following auxiliary result: 
\begin{claim}\label{59'}
Let notation be as before, then there exists an $h_v^{\circ}\in\mathcal{H}(\sigma_v)$ such that 
\begin{equation}\label{155}
W_v(h_{v}^{\circ};\sigma_v):=\int_{N(F_v)}h_{v}^{\circ}(w_0n_v)\exp\langle\rho_P,H_{M,v}(w_0n_v)\rangle\theta(n_v)dn_v\neq 0.
\end{equation}
\end{claim}
Let $N^{-}$ be the opposite of $N,$ i.e., $N^{-}=w_0Nw_0^{-1}.$ Then one may take arbitrarily two functions $\varphi_1\in\mathcal{C}_c^{\infty}(P(F_v))$ and $\varphi_1\in\mathcal{C}_c^{\infty}(N^{-1}(F_v))$ to define 
\begin{align*}
\widetilde{\varphi}(g)=\begin{cases}
\varphi_1(p)\varphi_2(n^{-})h_v^{\circ},\ g=pn^{-}\in P(F_v)N^{-1}(F_v);\\
0,\ \text{otherwise}.
\end{cases}
\end{align*}
Now we let $h_v$ (depending on $\varphi_1$ and $\varphi_2$) be the function
\begin{align*}
h_v(g)=\int_{N(F_v)}\int_{M(F_v)}\sigma(m^{-1})\exp\langle-\lambda+\rho_P, H_M(m)\rangle\widetilde{\varphi}(mng)dmdn.
\end{align*}
Since $\exp\langle\rho_P, H_M(m)\rangle$ is the modular character, for any $m_1\in M(F_v),$ $n_1\in N(F_v),$ one has $d(m_1nm_1^{-1})=\exp\langle\rho_P, H_M(m_1)\rangle dn.$ Then by changing $m$ to $mm_1^{-1}$ and $n$ to $nn_1^{-1}$ we obtain that
\begin{align*}
h_v(m_1n_1g)&=\int_{N(F_v)}\int_{M(F_v)}\sigma(m^{-1})\exp\langle-\lambda+\rho_P, H_M(m)\rangle\widetilde{\varphi}(mnm_1n_1g)dmdn\\
&=\sigma(m_1)\exp\langle\lambda+\rho_P, H_M(m_1)\rangle h_v(g),
\end{align*} 
which implies that $h_v\in V(\lambda,\sigma).$ Now we have 
\begin{align*}
h_v(n^{-})=\int_{N(F_v)}\int_{M(F_v)}\sigma(m^{-1})\exp\langle-\lambda+\rho_P, H_M(m)\rangle\varphi_1(mn)\varphi_2(n^-)h_v^{\circ}dmdn.
\end{align*}
We will choose $\varphi_1$ so that $\varphi_1(mn)\sigma(m^{-1})h_v^{\circ}=\varphi_1(mn)h_v^{\circ}.$ Then we have
\begin{align*}
W_v(h_{v};\lambda,\sigma_v)=\int_{N^-(F_v)}\int_{N(F_v)}\int_{M(F_v)}\mathcal{F}(n^-,m,n) dn^-dmdn,
\end{align*}
where $\mathcal{F}(n^-,m,n):=\exp\langle-\lambda+\rho_P, H_M(m)\rangle W_v(h_{v}^{\circ};\sigma_v)\varphi_1(mn)\varphi_2(n^-)\theta(w_0^{-1}n^-w_0).$ Therefore, $W_v(h_{v};\lambda,\sigma_v)$ is equal to the product of $W_v(h_{v}^{\circ};\sigma_v)$ and 
\begin{align*}
\int_{M(F_v)N(F_v)}\exp\langle-\lambda+\rho_P, H_M(m)\rangle \varphi_1(mn)dmdn\int_{N^-(F_v)}\varphi_2(n^-)\theta(w_0^{-1}n^-w_0)dn^-.
\end{align*}
One can take appropriate $\varphi_1$ and $\varphi_2$ to make the above integral nonzero constant independent of $\lambda.$ Now Lemma \ref{60''} follows from Claim \ref{59'}.
\end{proof}
\begin{remark}
When $\Ind_{M(F_v)N(F_v)}^{G(F_v)}\sigma_v\otimes\boldsymbol{1}$ is unramified, we can simply take $h_v$ to be a spherical vector in $\mathcal{H}(\sigma_v).$ However, when $\Ind_{M(F_v)N(F_v)}^{G(F_v)}\sigma_v\otimes\boldsymbol{1}$ is ramified, one cannot take $h_v$ to be a new vector in $\mathcal{H}(\sigma_v)$ any more, since otherwise $W_v(h_v;\lambda,\sigma_v)$ would vanish identically.  
\end{remark}
\begin{proof}[Proof of Claim \ref{59'}]
Write $\pi_v$ for the representation $\Int_{M(F_v)N(F_v)}^{G(F_v)}\sigma_v\otimes\boldsymbol{1}.$ Let $\Phi_v\in \mathcal{S}(F_v^n),$ and $h_{1,v}, h_{2,v}\in\mathcal{H}(\sigma_v).$ Let $s\in\mathcal{C}$ such that $\Re(s)>1.$ Then we consider the local Rankin-Selberg integral $\Psi_v(s;h_{1,v},h_{2,v},\Phi_v)$ defined by
\begin{align*}
\int_{N(F_v)\backslash G(F_v)}W_v(x_v;h_{1,v},\sigma_v)\overline{W_v(x_v;h_{2,v},\sigma_v)}\Phi_v((0,\cdots,0,1) x_v)|\det x_v|_{F_v}^sdx_v,
\end{align*}
where for $1\leq j\leq 2,$ $W_v(x_v;h_{j,v},\sigma_v)$ is defined by 
\begin{align*}
\int_{N(F_v)}h_{j,v}(w_0n_vx_v)\exp\langle\rho_P,H_M(w_0n_vx_v)\rangle\theta(n_v)dn_v.
\end{align*}
By \cite{JPSS83}, there exists $h_{1,v}^{\circ}, h_{2,v}^{\circ}\in\mathcal{H}(\sigma_v)$ and $\Phi_v^{\circ}\in \mathcal{S}(F_v^n),$ such that the local Rankin-Selberg integral $\Psi_v(s;h_{1,v}^{\circ},h_{2,v}^{\circ},\Phi_v^{\circ})$ equals exactly the local L-function $L_v(s,\pi_v\times\widetilde{\pi}_v).$ One then applies the bound in  \cite{LRS99} to see that $|L_v(s,\pi_v\times\widetilde{\pi}_v)|>0.$ Hence $I_v(s;h_{1,v}^{\circ},h_{2,v}^{\circ},\Phi_v^{\circ})\neq0,$ which implies that there exists some $x_v\in G(F_v)$ such that $W_v(x_v;h_{1,v}^{\circ},\sigma_v)\neq 0.$ Then we can take $h_{v}^{\circ}=\pi_v(x_v)h_{1,v}^{\circ}$ to get \eqref{155}.
\end{proof}

Let $\epsilon_0>0$ be a small constant (smaller than $1/(n^2+1)$). Let $\mathcal{C}_{\epsilon_0}^+$ be the piecewise smooth curve consisted of three pieces: $\{s\in\mathbb{C}:\ \Re(s)=0,\ \Im(s)\geq \epsilon_0\},$ $\{s\in\mathbb{C}:\ \Re(s)\geq0,\  |s|=\epsilon_0\},$ and $\{s\in\mathbb{C}:\ \Re(s)=0,\ \Im(s)\geq \epsilon_0\}.$ Then by Lemma \ref{49lem}, for any $s\in \mathcal{C}_{\epsilon_0}^+$ and any cuspidal representations $\sigma$ and $\sigma'$ as above, 
\begin{equation}\label{157}
L(1\pm s,\sigma\times\sigma')\ll_{F,\epsilon_0} \left(1+|\epsilon_0(1-\epsilon_0)|^{-1}\right)C(\sigma\times\sigma')^{\frac{\beta_{m,m'}\mp\Re(s)}{2}+\epsilon_0},
\end{equation}
where the implied constant depends only on $F$ and $\epsilon_0.$ We will fix $\epsilon_0$ henceforth.

As before, we fix a proper parabolic subgroup $P\in\mathcal{P}$ of type $(n_1,n_2,\cdots,n_r).$ Let $\mathfrak{X}_P$ be the subset of cuspidal data $\chi=\{(M,\sigma)\}$ such that $M=M_P.$ For any meromorphic function $F$ and $\epsilon\geq0,$ we denote by $\mathcal{V}(F)$ the set of poles of $M$ and denote by $\mathcal{U}_{\epsilon}(F)$ the set $\{s\in\mathbb{C}:\ |s-\rho|> \epsilon,\ \forall\ \rho\in \mathcal{V}(F)\}.$

For any $a<b,$ write $\mathcal{S}_{(a,b)}$ for the strip $a<\Re(s)<b.$ Let $s\in \mathcal{S}_{(0,1)}$ and $1\leq m\leq r-1.$ Let $j_m, \cdots,j_1$ be $m$ integers such that $1\leq j_m<\cdots<j_1\leq r-1.$ For any $1\leq l\leq m,$ let $\delta_l(s)$ be of the form $a_ls+b_l,$ with $a_l, b_l\in \mathbb{Z};$ and for $l\in \{1,2,\cdots,r-1\}\setminus \{j_k:\ 1\leq k\leq m\},$ $\mathcal{C}_l\in \{\mathcal{C}_{\epsilon_0}^+, \mathcal{C}\}.$ We say $(\delta_m(s),\cdots,\delta_1(s))$ is $nice$ $with$ $respect$ $to$ $\chi=\Ind\sigma_1|\cdot|^{\lambda_1}\otimes\cdots\otimes \sigma_r|\cdot|^{\lambda_r}\in\mathfrak{X}_P$ if there exists a finite set of integers $\mathcal{L}$ (where elements might have multiplicities) and linear forms $c(s,\kappa_l)$ and $\tilde{c}(s,\kappa_l)$ with respect to $s$ and $\kappa_l,$ $l\in \mathcal{L},$ i.e., $c(s,\kappa_l)$ (resp. $\tilde{c}(s,\kappa_l)$) is of the form $a_ls+b_l\kappa_l+c_l$ with $b_l\neq0,$ (resp. $a_l's+b_l'\kappa_l+c_l'$ with $b_l'\neq0$), where the coefficients are integers, such that $\underset{\kappa_{j_m}=\delta_m(s)}{\Res}\cdots\underset{\kappa_{j_1}=\delta_1(s)}{\Res}\mathcal{G}(\boldsymbol{\kappa};s)$ is of the form 
\begin{equation}\label{192.}
R(s;\chi)\prod_{l\in\mathcal{L}}\frac{\Lambda(c(s,\kappa_l),\sigma_l\otimes\tau\times\widetilde{\sigma}'_l)}{\Lambda(\tilde{c}(s,\kappa_l),\sigma_l\times\widetilde{\sigma}'_l)}\prod_k{\Lambda(s,\sigma_k\otimes\tau\times\widetilde{\sigma}_k)},
\end{equation}
where $\mathcal{G}(\boldsymbol{\kappa};s)=\mathcal{G}(\boldsymbol{\kappa};s,P,\chi)$ is defined as 
\begin{align*}
\prod_{k=1}^r\Lambda(s,\sigma_k\otimes\tau\times\widetilde{\sigma}_k)\prod_{j=1}^{r-1}\prod_{i=1}^j\frac{\Lambda(s+\kappa_{i,j},\sigma_i\otimes\tau\times\widetilde{\sigma}_{j+1})\Lambda(s-\kappa_{i,j},\sigma_{j+1}\otimes\tau\times\widetilde{\sigma}_{i})}{\Lambda(1+\kappa_{i,j},\sigma_{i}\times\widetilde{\sigma}_{j+1})\Lambda(1-\kappa_{i,j},\sigma_{j+1}\times\widetilde{\sigma}_{i})};
\end{align*} 
$\tilde{c}(s,\kappa_l)\in 1-\mathcal{D}_{\chi}$ for any $s\in\mathcal{S}_{(1/3,1)},$ $\kappa_l\in \mathcal{D}_{\chi};$ and the function $R(s;\chi)$ is meromorphic satisfying that for any $s\in\mathcal{U}_{\epsilon}(R(\cdot;\chi)),$
\begin{equation}\label{154}
R(s;\chi)=\prod_{l'\in\mathcal{L}'}\frac{\Lambda(c_{l'}(s),\sigma_{l'}\otimes\tau\times\widetilde{\sigma}'_{l'})}{\Lambda(\tilde{c}_{l'}(s),\sigma_{l'}\times\widetilde{\sigma}'_{l'})},
\end{equation}
for some finite index set $\mathcal{L}'$ (with multiplicities) and linear forms $c_{l'}(s)$ and $\tilde{c}_{l'}(s).$

\begin{thm}\label{58''}
Let notation be as before. Assume that $(\delta_m(s),\cdots,\delta_1(s))$ is $nice$ $with$ $respect$ $to$ $\chi.$ Then for any test function $\varphi$ and $\Phi,$ we have
\begin{align*}
\sum_{\chi\in\mathfrak{X}_P}\sum_{\phi\in \mathfrak{B}_{P,\chi}}\int_{\mathcal{C}_1}\cdots\cdots\int_{\mathcal{C}_{r-1}}\bigg|\underset{\kappa_{j_m}=\delta_m(s)}{\Res}\cdots\underset{\kappa_{j_1}=\delta_1(s)}{\Res}\mathcal{F}(\boldsymbol{\kappa};s)\bigg|\frac{d\kappa_{r-1}\cdots d\kappa_{1}}{d\kappa_{j_m}\cdots d\kappa_{j_1}}<\infty,
\end{align*}
for any $s\in \mathcal{R}(1/2;\tau)^-$ outside the set $\mathcal{S}^{ex}\bigcup\cup_{\chi}\mathcal{U}_{0}(R(\cdot;\chi)),$ where
\begin{align*}
\mathcal{S}^{ex}=\bigcup_{l=1}^{r-1}\Big\{s\in \mathbb{C}:\ s\in \mathcal{R}(1/2;\tau)^-,\ c(s,\kappa_l)\in\{0, 1\}\ \text{for some}\ \kappa_l\in \mathcal{C}_l\Big\}.
\end{align*}
Moreover, the point-wise defined function 
\begin{align*}
\sum_{\chi\in\mathfrak{X}_P}\sum_{\phi\in \mathfrak{B}_{P,\chi}}\int_{\mathcal{C}_1}\cdots\cdots\int_{\mathcal{C}_{r-1}}\underset{\kappa_{j_m}=\delta_m(s)}{\Res}\cdots\underset{\kappa_{j_1}=\delta_1(s)}{\Res}\mathcal{F}(\boldsymbol{\kappa};s)\frac{d\kappa_{r-1}\cdots d\kappa_{1}}{d\kappa_{j_m}\cdots d\kappa_{j_1}}
\end{align*}
converges locally normally in the region $\mathcal{R}(1/2;\tau)^-\setminus \left(\mathcal{S}^{ex}\bigcup\cup_{\chi}\mathcal{U}_{0}(R(\cdot;\chi))\right);$ admits a meromorphic continuation to the area $\mathcal{R}(1/2;\tau)^-\setminus \mathcal{S}^{ex}.$
\end{thm}
\begin{proof}
Recall that we have written $\mathcal{F}(\boldsymbol{\kappa};s)=\mathcal{F}(\boldsymbol{\kappa};s,P,\chi)$ for the meromorphic function  $R_{\varphi}(s,\boldsymbol{\kappa};\phi)\Lambda(s,\pi_{\boldsymbol{\kappa}}\otimes\tau\times\widetilde{\pi}_{-\boldsymbol{\kappa}})$. Hence for any $\Re(s)>1,$ 
\begin{align*}
\mathcal{F}(\boldsymbol{\kappa};s)=\sum_{\phi_1\in \mathfrak{B}_{P,\chi}}\langle\mathcal{I}_P(\lambda,\varphi)\phi_1,\phi_2\rangle\cdot\Psi(s,W_1,W_2;\lambda),
\end{align*}
which admits a meromorphic continuation to the whole complex plane (see Theorem H in \cite{Yang19}). From the transform \eqref{145} we see that 
\begin{equation}
\begin{cases}
\lambda_1=\lambda_1(\boldsymbol{\kappa})=[(r-1)\kappa_1+(r-2)\kappa_2+\cdots+\kappa_{r-1}]\cdot{r}^{-1},\\
\lambda_j=\lambda_j(\boldsymbol{\kappa})\lambda_{j-1}-\kappa_{j-1},\ 2\leq j\leq r-1,\\
\lambda_r=\lambda_r(\boldsymbol{\kappa})-\lambda_1-\lambda_2-\cdots-\lambda_{r-1}.
\end{cases}
\end{equation}
Hence, we may write $\langle\mathcal{I}_P(\lambda,\varphi)\phi_1,\phi_2\rangle=\langle\mathcal{I}_P(\boldsymbol{\kappa},\varphi)\phi_1,\phi_2\rangle,$ and $\Psi(s,W_1,W_2;\lambda)=\Psi(s,W_1,W_2;\boldsymbol{\kappa}).$ Since $\langle\mathcal{I}_P(\boldsymbol{\kappa},\varphi)\phi_1,\phi_2\rangle$ is the Mellin inversion of some compact support smooth function, so it is entire with respect to $\boldsymbol{\kappa}.$ Therefore,
\begin{align*}
\underset{\kappa_{j_m}=\delta_m(s)}{\Res}\cdots\underset{\kappa_{j_1}=\delta_1(s)}{\Res}\mathcal{F}(\boldsymbol{\kappa};s)=\sum_{\phi_1}\langle\mathcal{I}_P(\boldsymbol{\kappa}_s,\varphi)\phi_1,\phi_2\rangle\underset{\kappa_{j_m}=\delta_m(s)}{\Res}\cdots\underset{\kappa_{j_1}=\delta_1(s)}{\Res}\Psi(s;\boldsymbol{\kappa}),
\end{align*}
where $\phi_1$ runs over $\mathfrak{B}_{P,\chi};$  $\boldsymbol{\kappa}_s=(\kappa_1,\cdots,\kappa_{j_{m}-1},\delta_m(s),\cdots,\kappa_{j_1-1},\delta_{1}(s),\cdots,\kappa_{r-1});$ $\Psi(s;\boldsymbol{\kappa})=\Psi(s,W_1,W_2;\boldsymbol{\kappa}).$  Let $\iota$ be the canonical isomorphism of vector spaces $\iota:\ i\mathfrak{a}_P^*/i\mathfrak{a}_G^*\xrightarrow{\sim}\mathbb{R}^{r-1}.$ Let $\{e_1, e_2,\cdots, e_{r-1}\}$ be an orthonormal basis of $\mathbb{R}^{r-1}.$ Set $\boldsymbol{\kappa}^{\circ}=(\kappa_1^{\circ},\kappa_2^{\circ},\cdots,\kappa_{r-1}^{\circ}),$ where 
\begin{align*}
\kappa_j^{\circ}=\begin{cases}
	\iota^{-1}(\epsilon_0e_j),\ \text{if $j\in \{j_1,\cdots, j_m\};$}\\
	\kappa_j,\ \text{otherwise}.
\end{cases}
\end{align*}
Let $\boldsymbol{\kappa}_s^{\circ}=\boldsymbol{\kappa}_s-\boldsymbol{\kappa}^{\circ}.$ Then $\langle\mathcal{I}_P(\boldsymbol{\kappa}_s,\varphi)\phi_1,\phi_2\rangle$ is equal to $\langle\mathcal{I}_P(\boldsymbol{\kappa}^{\circ}+\boldsymbol{\kappa}_s^{\circ},\varphi)\phi_1,\phi_2\rangle.$

Now we shall study analytic behavior of $\underset{\kappa_{j_m}=\delta_m(s)}{\Res}\cdots\underset{\kappa_{j_1}=\delta_1(s)}{\Res}\Psi(s,W_1,W_2;\boldsymbol{\kappa}).$ Fix arbitrarily an $s_0\in  \mathcal{S}_{(0,1)}\setminus \mathcal{S}^{ex},$ then there exists $\epsilon=\epsilon(s)>0$ such that for any $s$ such that $|s-s_0|\leq\epsilon$ one has $|c(s,\kappa_l)|\geq \epsilon$ and $|c(s,\kappa_l)-1|\geq \epsilon$ for any $l\in\mathcal{L}.$ Denote by $B_{\epsilon}(s_0)$ the open neighborhood $\{s\in\mathbb{C}:\ |s-s_0|\leq\epsilon\}.$ Let $Sin(s_0,\epsilon)$ be the collection of poles (with multiplicity) of $R(s;\chi)$ in $B_{\epsilon}(s_0).$ Let 
\begin{align*}
R_{\Sin}(s;\chi):=\prod_{\rho\in \Sin(s_0,\epsilon)}(s-\rho).
\end{align*}
Then $R_{\Sin}(s;\chi)$ is a well defined polynomial since $\Sin(s_0,\epsilon)$ is finite. Note that the meromorphic function $R_{\Sin}(s;\chi)R(s;\chi)$ is holomorphic in $B_{\epsilon}(s_0).$ Then the function  $R_{\Sin}(s;\chi)\cdot\underset{\kappa_{j_m}=\delta_m(s)}{\Res}\cdots\underset{\kappa_{j_1}=\delta_1(s)}{\Res}\Psi(s,W_1,W_2;\boldsymbol{\kappa})$ is holomorphic in $B_{\epsilon}(s_0).$

 \begin{itemize}
	\item[Case 1:]If $c(s,\delta_{j}(s))=1,$ then $L(c(s,\kappa_l),\sigma\otimes\tau\times\sigma')$ has a simple pole at $\kappa_l=\delta_j(s).$ Let $C_{\epsilon}=\{s\in\mathbb{C}:\ |s-1|=\epsilon\}.$ By trianGLe inequality, we have that  $|\Res_{\kappa_l=\delta_j(s)}L(c(s,\kappa_l),\sigma\otimes\tau\times\sigma')|=|\Res_{s=1}L(s,\sigma\otimes\tau\times\sigma')|\leq(2\pi)^{-1}\cdot\int_{C_{\epsilon}}|L(s,\sigma\otimes\tau\times\sigma')||ds|,$ which is dominated, according to Lemma \ref{49lem}, by
	\begin{align*}
	\int_{C_{\epsilon}}\left(1+|s(s-1)|^{-1}\right)C(\sigma\otimes \tau\times\sigma')^{\frac{2-\Re(s)}{2}+\epsilon}|ds|\ll C(\sigma\otimes \tau\times\sigma').
	\end{align*} 
	where the implied constant is absolute, depending only on the base field $F$. In this case, the archimedean $L$-factor becomes 
	\begin{align*}
	L_{\infty}(1,\sigma\otimes\tau\times\sigma')=\prod_v\prod_{i=1}^{n_1}\prod_{j=1}^{n_1'}\Gamma_{F_v}\left(1+\mu_{\sigma\otimes\tau\times \sigma';v,i,j}\right).
	\end{align*}
	Note that for each local factor $\Gamma_{F_v}\left(1+\mu_{\sigma\otimes\tau\times \sigma';v,i,j}\right),$ one can apply Lemma \ref{48'} to show that $\Gamma_{F_v}\left(1+\mu_{\sigma\otimes\tau\times \sigma';v,i,j}\right)\asymp_s \Gamma_{F_v}\left(s+\mu_{\sigma\otimes\tau\times \sigma';v,i,j}\right),$ where the implied constant depends only on $s$. Hence, $L_{\infty}(1,\sigma\otimes\tau\times\sigma')\asymp_s L_{\infty}(s,\sigma\otimes\tau\times\sigma'),$ with the implied constant relying only on $s$. Hence,
	\begin{equation}\label{161'}
	\underset{\kappa_{l}=\delta_j(s)}{\Res}\Lambda(c(s,\kappa_l),\sigma\otimes\tau\times\sigma')\ll_s C(\sigma\otimes\tau\times\sigma';\gamma)|L_{\infty}(s,\sigma\otimes\tau\times\sigma')|.
	\end{equation}
	\item[Case 2:] If $c(s,\delta_{j}(s))=0,$ then archimedean factor $L_{\infty}(c(s,\kappa_l),\sigma\otimes\tau\times\sigma')$ has a possible simple pole at $\kappa_l=\delta_j(s).$ Then one has that $\Res_{\kappa_l=\delta_j(s)}\Lambda(c(s,\kappa_l),\sigma\otimes\tau\times\sigma')=\Res_{c(s,\kappa_l)=0}\Lambda(c(s,\kappa_l),\sigma\otimes\tau\times\sigma')=L(0,\sigma\otimes\tau\times\sigma')\Res_{s=0}L_{\infty}(s,\sigma\otimes\tau\times\sigma').$ Note that 
	\begin{align*}
	L_{\infty}(s,\sigma\otimes\tau\times\sigma')=\prod_v\prod_{i=1}^{n_1}\prod_{j=1}^{n_1'}\Gamma_{F_v}\left(s+\mu_{\sigma\otimes\tau\times \sigma';v,i,j}\right).
	\end{align*}
	Since $\tau$ is unitary, by \cite{LRS99} one has that $\Re(u_{\sigma\otimes\tau\times \sigma';v,i,j})\geq -3/5>-1,$ for any $v, i, j$ as above. Note that $\Gamma(s)$ only has simple poles at $s=-k,$ $k\in\mathbb{N}_{\geq 0}.$ Hence, there is a unique archimedean place $v_0$ and a unique Satake parameter $u_{\sigma\otimes\tau\times \sigma';v_0,i_0,j_0}$ such that $\Gamma_{F_v}\left(s+\mu_{\sigma\otimes\tau\times \sigma';v_0,i_0,j_0}\right)$ has a simple pole at $s=0.$ Hence $\mu_{\sigma\otimes\tau\times \sigma';v_0,i_0,j_0}=0.$ The residue is $\Res_{s=0}\Gamma_{F_v}\left(s\right)=1.$ In this case, since $\mu_{\sigma\otimes\tau\times \sigma';v_0,i_0,j_0}=0,$ Stirling formula implies that 
	\begin{equation}\label{159}
	|\Res_{s=0}L_{\infty}(s,\sigma\otimes\tau\times\sigma')|\asymp |L_{\infty}(s,\sigma\otimes\tau\times\sigma')|,
	\end{equation}
	where the implies constant is absolute. Since in this case we have $\sigma\otimes\tau\simeq\sigma',$ $L(s,\sigma\otimes\tau\times\sigma')$ has simple poles precisely at $s=1.$ Consider instead the function $f(s)=(s-1)(s+2)^{-(5+\beta_{n_1,n_1'}-\beta)/2}L(s,\sigma\otimes\tau\times\sigma'),$ where $\beta_{n_1,n_1'}:=1-1/(n_1^2+1)-1/(n_1'^2+1)$ and  $\beta=\Re(s).$ Then clearly $f(s)$ is holomorphic and of order 1 in the right half plane $\Re(s)>-\beta_{n_1,n_1'}.$ Hence by Phragm\'{e}n-Lindel\"{o}f principle we have that $f(s)$ is bounded by $O_{\epsilon}\left(C(\sigma\otimes\tau\times\sigma';\gamma)^{(1+\beta_{n_1,n_1'}-\beta)/{2}+\epsilon}\right)$ in the strip $-\beta_{n_1,n_1'}\leq\Re(s)\leq 1+\beta_{n_1,n_1'},$ leading to the estimate
	\begin{equation}\label{160}
	L(0,\sigma\otimes\tau\times\sigma')\ll_{\epsilon}C(\sigma\otimes\tau\times\sigma';\gamma)^{\epsilon+1/2+\beta_{n_1,n_1'}/2},
	\end{equation}
	where the implied constant is absolute. Hence, combining the estimates \eqref{159} and \eqref{160} we then obtain 
	\begin{equation}\label{161}
	\underset{\kappa_{l}=\delta_j(s)}{\Res}\Lambda(c(s,\kappa_l),\sigma\otimes\tau\times\sigma')\ll_{\epsilon}C(\sigma\otimes\tau\times\sigma';\gamma)^{4/5+\epsilon}|L_{\infty}(s,\sigma\otimes\tau\times\sigma')|.
	\end{equation}
\end{itemize}
In all, combining \eqref{161'} and \eqref{161}, we conclude that 
\begin{equation}\label{163}
\underset{\kappa_{l}=\delta_j(s)}{\Res}\Lambda(c(s,\kappa_l),\sigma\otimes\tau\times\sigma')\ll_{F,\epsilon} C(\sigma\otimes\tau\times\sigma';\gamma)|L_{\infty}(s,\sigma\otimes\tau\times\sigma')|,
\end{equation}
where the implied constant depends only on $\epsilon$ and the base field $F.$ 
\medskip

Let $\Phi=\Phi_{\infty}\cdot\prod_{v<\infty}\Phi_{v}\in\mathcal{S}_0(\mathbb{A}_F^n)$ be a test function, where $\Phi_{\infty}=\prod_{v\mid\infty}\Phi_v.$ Let $x_v=(x_{v,1},x_{v,2},\cdots,x_{v,n})\in F_v^n,$ then by definition, $\Phi_v$ is of the form
\begin{equation}\label{166}
\Phi_v(x_v)=e^{-\pi \sum_{j=1}^nx_{v,j}^2}\cdot \sum_{k=1}^mQ_k(x_{v,1},x_{v,2},\cdots,x_{v,n}),
\end{equation}
where $F_v\simeq \mathbb{R},$ $Q_k(x_{v,1},x_{v,2},\cdots,x_{v,n})\in \mathbb{C}[x_{v,1},x_{v,2},\cdots,x_{v,n}]$ are monomials; and
\begin{equation}\label{167}
\Phi_v(x_v)=e^{-2\pi \sum_{j=1}^nx_{v,j}\bar{x}_{v,j}}\cdot\sum_{k=1}^m Q_k(x_{v,1},\bar{x}_{v,1},x_{v,2},\bar{x}_{v,2},\cdots,x_{v,n},\bar{x}_{v,n}),
\end{equation}
where $F_v\simeq \mathbb{C}$ and $Q_k(x_{v,1},\bar{x}_{v,1},x_{v,2},\bar{x}_{v,2},\cdots,x_{v,n},\bar{x}_{v,n})$ are monomials in the ring $\mathbb{C}[x_{v,1},\bar{x}_{v,1},x_{v,2},\bar{x}_{v,2},\cdots,x_{v,n},\bar{x}_{v,n}].$ Thus there exists a finite index set $J$ such that 
\begin{align*}
\Phi_{\infty}(x_{\infty})=\sum_{\boldsymbol{j}=(j_v)_{v\mid\infty}\in J}\prod_{v\mid\infty}\Phi_{v,j_v}(x_v),\quad x_{\infty}=\prod_{v\mid\infty}x_v\in G(\mathbb{A}_{F,\infty}),
\end{align*}
where each $\Phi_{v,j_v}$ is of the form in \eqref{166} or \eqref{167} with $m=1.$ Let $\Phi_{\infty,\boldsymbol{j}}=\prod_{v\mid\infty}\Phi_{v,j-v},$ $\boldsymbol{j}=(j_v)_{v\mid\infty}\in J.$ Then $\Phi$ is equal to the sum over $\boldsymbol{j}\in J$ of each $\Phi_{\boldsymbol{j}}=\Phi_{\infty,\boldsymbol{j}}\prod_{v<\infty}\Phi_{v}\in\mathcal{S}_0(\mathbb{A}_F^n).$ According to \cite{Jac09}, $\Psi_v\left(s,W_{1,v},W_{2,v};\boldsymbol{\kappa},\Phi_{v,j}\right)$ converges absolutely in $\Re(s')>0$ for each $v\mid\infty$ and $j\in J.$ Hence, one has that
\begin{equation}\label{169}
\Big|\prod_{v\mid\infty}\Psi_v\left(s,W_{1,v},W_{2,v};\boldsymbol{\kappa},\Phi_{v}\right)\Big|\leq \sum_{\boldsymbol{j}\in J}\Big|\prod_{v\mid\infty}\Psi_v\left(s,W_{1,v},W_{2,v};\boldsymbol{\kappa},\Phi_{v,j_v}\right)\Big|.
\end{equation}

Since each $\Phi_{v,j_v}$ is a monomial multiplying an exponential function with negative exponent, $\Psi_v\left(s,W_{1,v},W_{2,v};\boldsymbol{0},\Phi_{v,j_v}\right)$ is in fact of the form $c_1\pi^{c_2s}\prod_i\prod_j\Gamma(s+\nu_{i,j}),$ where $c_1=c_1(v),$ $c_2=c_2(v)$ and $\nu_{i,j}=\nu_{i,j}(v)$ are some constants and the product is finite. Hence, $\Psi_v\left(s,W_{1,v},W_{2,v};\boldsymbol{\kappa},\Phi_{v,j_v}\right)$ is in fact of the form $c_1\pi^{c_2s}\prod_i\prod_j\Gamma(s+\lambda_i-\lambda_j+\nu_{i,j}).$ Since the local integral  $\Psi_v\left(s,W_{1,v},W_{2,v};\boldsymbol{\kappa},\Phi_{v,j}\right)$ converges absolutely in $\Re(s)>0$ for each $v\mid\infty,$ $\boldsymbol{\kappa}\in i\mathfrak{a}_P^*/i\mathfrak{a}_G^*$ and $j\in J,$ then there is no pole in the right half plane $\Re(s)>0.$ So one must have that $\Re(\nu_i)\geq0.$ Also, note that for each archimedean place $v,$ there exists a polynomial $Q_1(s,\boldsymbol{\kappa})\in\mathbb{C}[s,\kappa_1,\cdots,\kappa_{r-1}]$ (see loc. cit.) with integers $n_{i,j}$ and $N_{i,j}$ depending on $\pi_{\infty}$ and $\boldsymbol{\kappa},$ such that 
\begin{align*}
L_v(s,\pi_{\boldsymbol{\kappa},v}\otimes\tau_v\times\widetilde{\pi}_{-\boldsymbol{\kappa},v})=Q_1(s,\boldsymbol{\kappa})\prod_{i=1}^r\prod_{j=1}^rL_v(s+\kappa_{i,j-1},\sigma_{v,i}\otimes\tau_v\times\widetilde{\sigma}_{v,j}),
\end{align*}
where $\Re(s)>\beta_{n_i,n_j}:=1-1/(n_i^2+1)-1/(n_j^2+1).$ Since each $\sigma_{v,j}$ is unitary, $L_v(s+\kappa_{i,j-1},\sigma_{v,i}\otimes\tau_v\times\widetilde{\sigma}_{v,j})$ is holomorphic when $\Re(s)>\beta_{n_i,n_j},$ then $Q_1(s,\boldsymbol{\kappa})$ is nonvanishing in $\Re(s)>\beta_{n,n}=1-2/(n^2+1),$ and each zero of $Q_1(s,\boldsymbol{\kappa})$ must be a pole of some  $L_v(s+\kappa_{i,j-1},\sigma_{v,i}\otimes\tau_v\times\widetilde{\sigma}_{v,j})$ (after meromorphic continuation), for some $1\leq i,j \leq r.$ Let $\mu_{p_i,q_j},$ $1\leq p_i\leq n_i,$ $1\leq q_j\leq n_j,$ be Satake parameters such that $L_v(s+\lambda_i-\lambda_j,\sigma_{v,i}\otimes\tau_v\times\widetilde{\sigma}_{v,j})=c_{1,i,j}\pi^{c_{2,i,j}s}\prod_{p_i}\prod_{q_j}\Gamma(s+\lambda_i-\lambda_j+\mu_{p_i,q_j}).$ Then $\Re(\mu_{p_i,q_j})\geq\beta_{n_i,n_j}.$ Then there exist constants $c_{\chi},$ nonnegative integers $m_{p_i,q_j}$ and exponents $e_{p_i,q_j}\in\mathbb{N}_{\geq 0}$ such that 
\begin{equation}\label{169_1}
Q_1(s,\boldsymbol{\kappa})=c_{\chi}\prod_{i=1}^r\prod_{j=1}^r\prod_{p_i=1}^{n_i}\prod_{q_j=1}^{n_j}(s+\lambda_i-\lambda_j+\mu_{p_i,q_j}+m_{p_i,q_j})^{e_{p_i,q_j}}.
\end{equation}
In conclusion, when $\Re(s)>1,$ $\Psi_v\left(s,W_{1,v},W_{2,v};\boldsymbol{\kappa},\Phi_{v,j_v}\right)$ is equal to product of the meromorphic function $Q_1(s,\boldsymbol{\kappa})\mathcal{H}_{v}(s,\boldsymbol{\kappa})$ and the meromorphic function 
\begin{align*}
\underset{\substack{1\leq i,j\leq n\\ i\neq j}}{\prod\prod}\frac{1}{L_v(1+\kappa_{i,j-1},\sigma_{v,i}\times\widetilde{\sigma}_{v,j})}\prod_{p=1}^r\prod_{q=1}^rL_v(s+\lambda_p-\lambda_q,\sigma_{v,p}\otimes\tau_v\times\widetilde{\sigma}_{v,q}),
\end{align*}
where $\mathcal{H}_{v}(s,\boldsymbol{\kappa})=\mathcal{H}_{v}(s,\lambda)$ defined just before \eqref{ar}, depending on $\pi_{v}$ and $\Phi_{v,j_{v}}.$ We thus obtain meromorphic continuation of $\Psi_v\left(s,W_{1,v},W_{2,v};\boldsymbol{\kappa},\Phi_{v,j_v}\right)$ to the whole complex plane. Now we identify $\Psi_v\left(s,W_{1,v},W_{2,v};\boldsymbol{\kappa},\Phi_{v,j_v}\right)$ with its continuation. Then by \eqref{169_1} and preceding analysis we have 
\begin{align*}
\Psi_v\left(s,W_{1,v},W_{2,v};\boldsymbol{\kappa},\Phi_{v,j_v}\right)=Q_2(s,\boldsymbol{\kappa})\prod_{p=1}^r\prod_{q=1}^rL_v(s+\kappa_{p,q-1},\sigma_{v,p}\otimes\tau_v\times\widetilde{\sigma}_{v,q}),
\end{align*}
where for any $i< j,$ $\kappa_{i,j-1}=\sum_{k=i}^{j-1}\langle e_k,\boldsymbol{\kappa}\rangle;$ and $Q_2(s,\boldsymbol{\kappa})$ is equal to the product of $c_{\chi}\mathcal{H}_{v}(s,\boldsymbol{\kappa})$ and the function 
\begin{align*}
\prod_{i'\in I}\prod_{j'\in J}\frac{1}{L_v(1+\kappa_{i',j'-1},\sigma_{v,i'}\times\widetilde{\sigma}_{v,j'})}\prod_{i=1}^r\prod_{j=1}^r\prod_{p_i=1}^{n_i}\prod_{q_j=1}^{n_j}(s_{i,j}+\mu'_{i,j,p_i,q_j})^{e_{p_i,q_j}},
\end{align*}
where $I$ and $J$ are some finite set of indexes; $s_{i,j}=s+\lambda_i-\lambda_j$ and $\mu'_{i,j,p_i,q_j}=\mu_{p_i,q_j}+m_{p_i,q_j}.$ One can check directly the type of residues in the proof of Theorem H in \cite{Yang19} to conclude that for $n\leq 4$ the function 
$$
\mathcal{H}_{v}(s,\boldsymbol{\kappa})\prod_{i'\in I}\prod_{j'\in J}{L_v(1+\kappa_{i',j'-1},\sigma_{v,i'}\times\widetilde{\sigma}_{v,j'})}^{-1}
$$ 
is entire as a function of $\boldsymbol{\kappa}^{\circ}$ and as a function of $s$ it is nonvanishing in $\Re(s)\geq c_0'>0$ for some absolute constant $c_0'.$ The existence of $c_0'$ comes from the fact that $\Re(\boldsymbol{\kappa})$ lies in the box $[-4,4]^{n-1}.$

Since $|\Re(\mu'_{i,j,p_i,q_j})|\leq \beta_{n_i,n_j},$ there exists some $c_0>0$ such that for any $s\geq c_0,$ $\big|Q_2(s,\boldsymbol{\kappa})\big|$ is bounded by the product of $|c_{\chi}\mathcal{H}_{v}(s,\boldsymbol{\kappa})|$ and 
\begin{align*}
\prod_{i'\in I}\prod_{j'\in J}\frac{1}{L_v(1+\kappa_{i',j'-1},\sigma_{v,i'}\times\widetilde{\sigma}_{v,j'})}\prod_{i=1}^r\prod_{j=1}^r\prod_{p_i=1}^{n_i}\prod_{q_j=1}^{n_j}\big|s_{i,j}+\mu'_{i,j,p_i,q_j}\big|^{e_{p_i,q_j}},
\end{align*} 
where $s_{i,j}=s+\kappa_{i,j,s}.$ Recall that we have restricted $s$ in a fixed compact set, then $\Im(s+\kappa_{i,j,s})\asymp\Im({\kappa}^{\circ}_{i,j}),$ where for any $i< j,$ $\kappa_{i,j-1}^{\circ}=\langle e_i,\boldsymbol{\kappa}^{\circ}\rangle+\cdots+\langle e_{j-1},\boldsymbol{\kappa}^{\circ}\rangle.$ Therefore, we can take $c_0$ to be large enough (depending only on the fixed neighborhood of $s$) to get that 
\begin{equation}\label{171'}
\big|Q_2(s,\boldsymbol{\kappa})\big|\leq \big|Q_2(\Re(s'),\boldsymbol{\kappa}^{\circ})\big|,
\end{equation} 
where $s'$ is any complex number such that $\Im(s')=\Im({\kappa}^{\circ}_{i,j})$ and $\Re(s')\geq c_0.$

Note that the function $c(s,\kappa_l)(c(s,\kappa_l)-1)\Lambda(c(s,\kappa_l),\sigma\otimes\tau\times\widetilde{\sigma}')$ is entire. Then by Phragm\'{e}n-Lindel\"{o}f principle and the functional equation we have 
\begin{align*}
&\big|c(s,\kappa_l)(c(s,\kappa_l)-1)\Lambda(c(s,\kappa_l),\sigma\otimes\tau\times\widetilde{\sigma}')\big|\\
\ll& C(\sigma\otimes\tau\times\widetilde{\sigma}';s_2)^{\alpha(s_2)}\big|c(s_2,\kappa_l)(c(s_2,\kappa_l)-1)\Lambda(c(s_2,\kappa_l),\sigma\otimes\tau\times\widetilde{\sigma}')\big|,
\end{align*}
where $s_2$ is the unique complex number such that $\Re(c(s_2,\kappa_l))= 2|\Re(c(s,\kappa_l))|+2$ and $\Im(c(s_2,\kappa_l))=\Im(c(s,\kappa_l));$ and $\alpha(s_2)$ is positive depending only on $\Re(s_2).$ 

Let $s\in B_{\epsilon}(s_0).$ Then by our definition $\min\{|c(s,\kappa_l)|,|c(s,\kappa_l)-1|\}\geq \epsilon,$ for any $\kappa_l\in \mathcal{C}_l.$ Then one has, for any $s\in B_{\epsilon}(s_0),$ that
\begin{align*}
\big|\Lambda(c(s,\kappa_l),\sigma\otimes\tau\times\widetilde{\sigma}')\big|
\ll_{\epsilon} C(\sigma\otimes\tau\times\widetilde{\sigma}';s_2)^{\alpha(s_2)+1}\big|L_{\infty}(c(s_2,\kappa_l),\sigma\otimes\tau\times\widetilde{\sigma}')\big|,
\end{align*}
since $L(c(s_2,\kappa_l),\sigma\otimes\tau\times\widetilde{\sigma}')\ll 1.$ Therefore, we then obtain that 
\begin{equation}\label{1,}
\big|\Lambda(c(s,\kappa_l),\sigma\otimes\tau\times\widetilde{\sigma}')\big|
\ll_{\epsilon} C(\sigma\otimes\tau\times\widetilde{\sigma}';s')^{-N}\big|L_{\infty}(s',\sigma\otimes\tau\times\widetilde{\sigma}')\big|,
\end{equation}
where $N>0$ is large and $s'$ is such that $\Re(s')=4N+2$ and $\Im(c(s',\kappa_l))=\Im(c(s,\kappa_l)).$ Moreover, the implied constant in \eqref{1,} is independent of $N.$ Note that $L(s',\sigma\otimes\tau\times\widetilde{\sigma}')\gg1.$ One can deduce from \eqref{1,} that 
\begin{equation}\label{170'}
\prod_{l\in\mathcal{L}}\big|\Lambda(c(s,\kappa_l),\sigma_l\otimes\tau\times\widetilde{\sigma}'_l)\big|
\ll_{\epsilon} \prod_{l\in\mathcal{L}}C(\sigma_l\otimes\tau\times\widetilde{\sigma}'_l;s')^{-N}\big|\Lambda(s',\sigma_l\otimes\tau\times\widetilde{\sigma}'_l)\big|.
\end{equation}
A similar argument applied to \eqref{154} implies that 
\begin{equation}\label{170}
\big|R_{\Sin}(s;\chi)\big|\cdot\prod_{l'\in\mathcal{L}'}\big|\Lambda(c_{l'}(s),\sigma_{l'}\otimes\tau\times\widetilde{\sigma}'_{l'})\big|\ll \prod_{l'\in\mathcal{L}'}\frac{\big|\Lambda(s',\sigma_{l'}\otimes\tau\times\widetilde{\sigma}'_{l'})\big|}{C(\sigma_{l'}\otimes\tau\times\widetilde{\sigma}'_{l'};s')^{N}},
\end{equation}
where $s\in B_{\epsilon}(s_0);$ the parameters $N$ and $s'$ are defined as in \eqref{170'}.
\medskip 

Let $S(\varphi,\Phi)$ be the finite set of nonarchimedean places such that both $\varphi_v$ and $\Phi_v=\Phi_v^{\circ}$ are the characteristic function of $G(\mathcal{O}_{F,v})$ outside $\Sigma_{F,\infty}\cup S(\pi,\Phi).$ Note that when $\varphi_v$ is $G(\mathcal{O}_{F,v})$-invariant, then $\pi_{v,\lambda}$ is unramified. So the cardinality of the finite set $S(\pi,\Phi)$ is bounded in terms of $\tau,$ $\Phi$ and the $K$-finite type of the test function $\varphi.$ Namely, there exists a finite set $S_{\varphi,\tau,\Phi}$ of prime ideals of the base field $F$ such that for any $\pi$ induced from some cuspidal datum $\chi\in\mathfrak{X}_P,$ one has $S(\pi,\Phi)\subseteq S_{\varphi,\tau,\Phi}.$ Let $\chi=(\sigma_1,\cdots,\sigma_r),$ where each $\sigma_j$ is a cuspidal representation of some $\GL(n_j,\mathbb{A}_F).$ By spectral expansion,  For each $v\in S_{\varphi,\tau,\Phi},$ every possible local component $\sigma_{j,v}$ has bounded conductor in terms of $K_{v}$-type of $\varphi_v,$ $\tau_v$ and $\Phi_v.$ Then there can only be finitely many (depending on $\varphi_v,$ $\tau_v$ and $\Phi_v$) such $\sigma_{j,v}$'s. Hence there are only finitely many possible $\pi_{v}$'s,  for any $v\in S_{\varphi,\tau,\Phi}.$ Moreover, for each $\pi_v,$ there are finitely many vectors in $\pi_v$ having the given $K_v$-type. Hence $\#\mathfrak{B}_{P,\chi,v}<\infty$ for each $v\in S_{\varphi,\tau,\Phi}.$ Denote by $D(v;\varphi_v,\tau_v,\Phi_v)$ the sum of all finitely many possible $\#\mathfrak{B}_{P,\chi,v}.$ Let $D_{\varphi,\tau,\Phi}$ be the product of $D(v;\varphi_v,\tau_v,\Phi_v)$ over $v\in S_{\varphi,\tau,\Phi}.$ Then $D_{\varphi,\tau,\Phi}$ is an integer depending only on $\tau$ and the test functions $\varphi$ and $\Phi.$ Note that for each $v\in S_{\varphi,\tau,\Phi},$ the local Whittaker function $W_{v}(x;\pi_v,\boldsymbol{\kappa})$ is dominated by a gauge $\xi_{\pi_v}$ uniformly in fixed strips $c_1\leq \Re(\boldsymbol{\kappa}_j)\leq c_2,$ $1\leq j\leq r-1.$ Hence we have
\begin{align*}
\Psi_v\left(s,W_{1,v},W_{2,v};\boldsymbol{\kappa}_s,\Phi_{v}\right)\ll \int_{N(F_v)\backslash G(F_v)}\Big|\xi_{\alpha}(x_v)\overline{\xi_{\beta}(x_v)}\Phi(\eta x_v)|\det x_v|_{F_v}^{\Re(s)}\Big|dx_v,
\end{align*}
which is finite since $\xi_{\alpha}$ and $\xi_{\beta}$ are Bruhat-Schwartz functions. Given $s\in \mathcal{S}_{(0,1)}\setminus \mathcal{S}^{ex},$ since there are only finitely many possible local Rankin-Selberg integrals $\Psi_v\left(s,W_{1,v},W_{2,v};\boldsymbol{\kappa}_s,\Phi_{v}\right),$ $v\in S_{\varphi,\tau,\Phi},$ and each is finite, we see that 
\begin{align*}
C_{\varphi,\tau,\Phi}:=\prod_{v\in S_{\varphi,\tau,\Phi}}\max_{\phi_{1,v}\in\mathfrak{B}_{P,\chi,v}}\max_{\phi_{2,v}\in\mathfrak{B}_{P,\chi,v}}\Big|\Psi_v\left(s,W_{1,v},W_{2,v};\boldsymbol{\kappa}_s,\Phi_{v}\right)\Big|<\infty.
\end{align*}
Since $\boldsymbol{\kappa}_s=\boldsymbol{\kappa}^{\circ}+\boldsymbol{\kappa}_s^{\circ},$ then $\langle\mathcal{I}_{P,v}(\boldsymbol{\kappa}_s,\varphi_v)\phi_{1,v},\phi_{2,v}\rangle_v=\langle\mathcal{I}_{P,v}(\boldsymbol{\kappa}^{\circ}+\boldsymbol{\kappa}_s^{\circ},\varphi_v)\phi_{1,v},\phi_{2,v}\rangle_v=\langle\mathcal{I}_{P,v}(\boldsymbol{\kappa}^{\circ},\varphi_ve^{(\boldsymbol{\kappa}_s^{\circ}+\rho_P)H_P})\phi_{1,v},\phi_{2,v}\rangle_v.$ Noting that $\varphi_ve^{(\boldsymbol{\kappa}_s^{\circ}+\rho_P)H_P}$ is Bruhat-Schwartz and the representation $\mathcal{I}_{P,v}(\boldsymbol{\kappa}^{\circ},\varphi_ve^{(\boldsymbol{\kappa}_s^{\circ}+\rho_P)H_P})$ is unitary, we have by trianGLe inequality that $|\langle\mathcal{I}_{P,v}(\boldsymbol{\kappa}_s,\varphi_v)\phi_{1,v},\phi_{2,v}\rangle_v|\leq \sqrt{\langle\phi_{1,v},\phi_{1,v}\rangle_v\langle\phi_{2,v},\phi_{2,v}\rangle_v}=1.$ Applying estimates for Satake parameters to the local Eulerian product we get that $|L(c(s,\kappa_l),\sigma_v\otimes\tau_v\times\sigma_v')|^{-1}\leq(1+N_{F/\mathbb{Q}}(\mathfrak{p})^{M_1})^{M_2},$ for some absolute positive constants $M_1$ and $M_2.$ Hence by definition of $R_v\left(s,W_{1,v},W_{2,v};\boldsymbol{\kappa}_s,\Phi_{v}\right)$ we have
\begin{equation}\label{171}
\prod_{v\in S_{\varphi,\tau,\Phi}}\big|\langle\mathcal{I}_{P,v}(\boldsymbol{\kappa}_s,\varphi_v)\phi_{1,v},\phi_{2,v}\rangle_vR_v\left(s,W_{1,v},W_{2,v};\boldsymbol{\kappa}_s,\Phi_{v}\right)\big|\ll_F C_{\varphi,\tau,\Phi},
\end{equation}
where the implied constant depends only on the base field $F.$ Since $c'(s,\kappa_l)\in 1-\mathcal{D}_{\chi},$ $\Re(c'(s,\kappa_l))>4/5.$ Then applying the upper bounds for Satake parameters we get $|L_v(c'(s,\kappa_l),\sigma\times\sigma')|\leq \prod\prod|1-N_{F/\mathbb{Q}}(\mathfrak{p})^{-1/2}|^{-1}<\infty.$ Therefore,
\begin{equation}\label{174}
\big|L_{S_{\varphi,\tau,\Phi}}(c'(s,\kappa_l),\sigma\times\sigma'))\big|=\prod_{v\in S_{\varphi,\tau,\Phi}}\big|L_v(c'(s,\kappa_l),\sigma\times\sigma'))\big|\ll1,
\end{equation}
where the implied constant depends only on $F,$ $\tau,$ $\varphi$ and $\Phi.$ 
\medskip

Let $\sigma$ and $\sigma'$ be cuspidal representations of $\GL(n_1,\mathbb{A}_F)$ and $\GL(n_1',\mathbb{A}_F),$ respectively. Since $|c(s,\kappa_l)|\geq \epsilon$ and $|c(s,\kappa)-1|\geq \epsilon,$ then by convexity bound,
\begin{equation}\label{173'}
L(c(s,\kappa_l),\sigma\otimes\tau\times\sigma')\ll_{F,\epsilon_0,\epsilon} C(\sigma\otimes\tau\times\sigma')^{\frac{1+\beta_{n_1,n_1'}-\Re(s)}{2}+\epsilon}.
\end{equation}
Also, since $c'(s,\kappa_l)\in 1-\mathcal{D}_{\chi},$ then we can apply the result in \cite{Lap13} to deduce 
\begin{equation}\label{175}
L(c'(s,\kappa_l),\sigma\times\sigma')\gg_{F,n} C(\sigma\times\sigma')^{-c_l}C(\sigma\times\sigma)^{-c_l}C(\sigma'\times\sigma')^{-c_l},
\end{equation}
where $c_l>0$ is absolute. Denote  $R(s;\chi)\underset{\kappa_{j_m}=\delta_m(s)}{\Res}\cdots\underset{\kappa_{j_1}=\delta_1(s)}{\Res}\Psi(s,W_1,W_2;\boldsymbol{\kappa})$ by $\Res\Psi_{12}(s).$ Gather \eqref{157}, \eqref{163}, \eqref{169}, \eqref{170'}, \eqref{170}, \eqref{171'}, \eqref{174}, \eqref{173'} and \eqref{175} to obtain
\begin{equation}\label{174'}
\big|\Res\Psi_{12}(s)\big|\ll_{F,\epsilon_0,\epsilon}C_{\varphi,\tau,\Phi}\sum_{\boldsymbol{j}\in J}\Big|\prod_{v\mid\infty}\Psi_v\left(s',W_{1,v},W_{2,v};\boldsymbol{\kappa}^{\circ},\Phi_{v,j_v}\right)\Big|\cdot H_{\chi}(\boldsymbol{\kappa}^{\circ}),
\end{equation}
where $H_{\chi}(\boldsymbol{\kappa}^{\circ})$ is a function depending on $\chi$ and $\boldsymbol{\kappa}^{\circ},$ and it is defined by
\begin{align*}
H_{\chi}(\boldsymbol{\kappa}^{\circ})=\frac{\prod_{k=1}^r\prod_{j=1}^{r-1}\prod_{i=1}^j[C(\sigma_k\otimes\tau\times\widetilde{\sigma}_{k})C(\sigma_i\otimes\tau\times\widetilde{\sigma}_{j+1})C(\sigma_{j+1}\otimes \tau\times\widetilde{\sigma}_{i})]^{N}}{\prod_{q=1}^{r-1}\prod_{p=1}^q\big|L(1+\kappa_{p,q}^{\circ},\sigma_{p}\times\widetilde{\sigma}_{q+1})L(1-\kappa_{p,q}^{\circ},\sigma_{q+1}\times\widetilde{\sigma}_{p})\big|},
\end{align*}
where $N$ is an absolute constant. Let $s_0'> 4N+1$ be a large enough (depending at most possibly on $\epsilon$) real number. Then substituting Stirling formula into the estimate \eqref{174'} we have that 
\begin{equation}\label{176}
\big|\Res\Psi_{12}(s)\big|\ll_{\epsilon} \frac{C_{\varphi,\tau,\Phi}\sum_{\boldsymbol{j}\in J}\big|\prod_{v\mid\infty}\Psi_v\left(s_0',W_{1,v},W_{2,v};\boldsymbol{\kappa}^{\circ},\Phi_{v,j_v}\right)\big|}{\prod_{j=1}^{r-1}\prod_{i=1}^j\big|L(1+\kappa_{i,j}^{\circ},\sigma_{i}\times\widetilde{\sigma}_{j+1})L(1-\kappa_{i,j}^{\circ},\sigma_{j+1}\times\widetilde{\sigma}_{i})\big|}.
\end{equation}
Since $L_S(s',\sigma\times\sigma')=\prod_{v\in S_{\varphi,\tau,\Phi}}L_v(s',\sigma_v\times\sigma_v')\gg1$ when $\Re(s')>4/5,$ where the implied constant is absolute, then from \eqref{176} we deduce that 
 \begin{equation}\label{177}
 \big|\Res\Psi\big|\ll_{\epsilon} \frac{C_{\varphi,\tau,\Phi}\sum_{\boldsymbol{j}\in J}\big|\prod_{v\mid\infty}\Psi_v\left(s_0',W_{1,v},W_{2,v};\boldsymbol{\kappa}^{\circ},\Phi_{v,j_v}\right)\big|}{\prod_{j=1}^{r-1}\prod_{i=1}^j\big|L^S(1+\kappa_{i,j}^{\circ},\sigma_{i}\times\widetilde{\sigma}_{j+1})L^S(1-\kappa_{i,j}^{\circ},\sigma_{j+1}\times\widetilde{\sigma}_{i})\big|},
 \end{equation}
where $\Res\Psi=\Res\Psi_{12}(s)$ and $L^S(s',\sigma\times\sigma')=\prod_{v\notin \Sigma_{F,\infty}\cup S_{\varphi,\tau,\Phi}}L_v(s',\sigma_v\times\sigma_v')$ is the partial $L$-function, and the implied constant in \eqref{177} depends only on $F$ and $\epsilon.$
\medskip 

On the other hand, for any $v\in S_{\varphi,\tau,\Phi},$ by Lemma \ref{60''}, there exists some $\phi_v^{\circ}\in \mathcal{H}(\sigma_{1,v},\cdots, \sigma_{r,v})$ such that $W(e; \phi_v^{\circ},\boldsymbol{\kappa})\neq0,$ for any $\boldsymbol{\kappa}\in i\mathfrak{a}_P^*/i\mathfrak{a}_G^*.$ Since $\Phi_v$ is a Schwartz-Bruhat function, we can write $\Phi_v$ as a finite sum of $\Phi_{v,l},$ where each $\Phi_{v,l}$ is a constant multiplying a characteristic function of some connected compact subset of $F_v^n.$ Then the Fourier transform of $\Phi_{v,l}$ is of the same form. Let the integral $\Psi_{v}^{*}\left(s',W_{v}^{\circ},W_{v}^{\circ};\boldsymbol{\kappa},\Phi_{v,l}\right)$ be defined by
\begin{align*}
\int_{N(F_v)\backslash G(F_v)} W(x_v; \phi_v^{\circ},\boldsymbol{\kappa})\overline{W(x_v; \phi_v^{\circ},-\bar{\boldsymbol{\kappa}})}\cdot\Phi_{v,l}(\eta x_v)|\det x_v|_{F_v}^{s'}dx_v.
\end{align*}
If $\Psi_{v}^{*}\left(s',W_{v}^{\circ},W_{v}^{\circ};\boldsymbol{\kappa},\Phi_{v,l}\right)=0$ for some $\boldsymbol{\kappa}\in i\mathfrak{a}_P^*/i\mathfrak{a}_G^*$ and some $s'>3,$ then $0=|\Psi_{v}^{*}\left(s',W_{v}^{\circ},W_{v}^{\circ};\boldsymbol{\kappa},\Phi_{v,l}\right)|=\Psi_{v}^{*}\left(s',W_{v}^{\circ},W_{v}^{\circ};\boldsymbol{\kappa},|\Phi_{v,l}|\right),$ which amounts to that 
\begin{align*}
\int_{N(F_v)\backslash G(F_v)} \big|W(x_v; \phi_v^{\circ},\boldsymbol{\kappa})\big|^2\cdot|\Phi_{v,l}(\eta x_v)|\cdot|\det x_v|_{F_v}^{s'}dx_v=0.
\end{align*}
Since $W(x_v; \phi_v^{\circ},\boldsymbol{\kappa})$ is a continuous function of $x_v,$ then $W(x_v; \phi_v^{\circ},\boldsymbol{\kappa})=0$ for any $x_v.$ In particular, $W(e; \phi_v^{\circ},\boldsymbol{\kappa})=0,$ which is a contradiction. Hence, one sees that  $\Psi_{v}^{*}\left(s',W_{v}^{\circ},W_{v}^{\circ};\boldsymbol{\kappa},\Phi_{v,l}\right)\neq0$ for any $\boldsymbol{\kappa}\in i\mathfrak{a}_P^*/i\mathfrak{a}_G^*$ and any $s'>3.$ Note that by Proposition 43 in \cite{Yang19}, for any $s'>3,$ we have
\begin{align*}
\frac{\Psi_v\left(s',W_{v}^{\circ},W_{v}^{\circ};\boldsymbol{\kappa},\Phi_{v,l}\right)}{L_v(s',\pi_{\boldsymbol{\kappa},v}\otimes\tau_v\times\widetilde{\pi}_{-\boldsymbol{\kappa},v})}\in \mathbb{C}[q_v^{s'}, q_v^{-s'}, q_v^{\lambda_i}, q_v^{-\lambda_i}:\ 1\leq i\leq r].
\end{align*}

Then for fix $s'>3$ and for any $\boldsymbol{\kappa},$  $\Psi_v\left(s',W_{v}^{\circ},W_{v}^{\circ};\boldsymbol{\kappa},\Phi_{v,l}\right)L_v(s',\pi_{\boldsymbol{\kappa},v}\otimes\tau_v\times\widetilde{\pi}_{-\boldsymbol{\kappa},v})^{-1}$ is a polynomial nonvanishing in a compact domain. Then there exists a positive constant $C_{v,s'}'=C(s';\varphi_v,\Phi_v,\tau_v)$ such that for any $\boldsymbol{\kappa},$ one has  $|\Psi_v\left(s',W_{v}^{\circ},W_{v}^{\circ};\boldsymbol{\kappa},\Phi_{v,l}\right)L_v(s',\pi_{\boldsymbol{\kappa},v}\otimes\tau_v\times\widetilde{\pi}_{-\boldsymbol{\kappa},v})^{-1}|\geq C_{v,s'}'.$ Since $s'>3,$ we have 
\begin{align*}
\big|L_v(s',\pi_{\boldsymbol{\kappa},v}\otimes\tau_v\times\widetilde{\pi}_{-\boldsymbol{\kappa},v})\big|&\geq \prod_{i=1}^r\prod_{j=1}^r\prod_{k=1}^{n_i}\prod_{l=1}^{n_j}\Big|1+|St_{\sigma\times \sigma',k,l}(\mathfrak{p})|\cdot N_{F/\mathbb{Q}}(\mathfrak{p})^{-s'}\Big|^{-1}\\
&\geq \prod_{i=1}^r\prod_{j=1}^r\Big|1+N_{F/\mathbb{Q}}(\mathfrak{p})^{-s'+1-(1+n_i^2)^{-1}-(1+n_j^2)^{-1}}\Big|^{-2n},
\end{align*}
where the right hand side is larger than $\prod_{i=1}^r\prod_{j=1}^re^{-2n\sum_pp^{-2}}\geq e^{-4nr^2}.$ Let $C_{v,s'}=e^{-4nr^2}C_{v,s'}'.$ Then we have that $|\Psi_v\left(s,W_{v}^{\circ},W_{v}^{\circ};\boldsymbol{\kappa},\Phi_{v,l}\right)|\geq C_v>0,$ for any $v\in S_{\varphi,\tau,\Phi}.$ Let $C_{\varphi,\tau,\Phi}^{\circ}(s')$ be the product of $C_{v,s'}$ over $v\in S_{\varphi,\tau,\Phi}.$ Denote by $\Psi_{S_{\varphi,\tau,\Phi}}\left(s',W_{v}^{\circ},W_{v}^{\circ};\boldsymbol{\kappa},\Phi_{v,l}\right)$ the product of local Rankin-Selberg integrals $\Psi_{v}\left(s',W_{v}^{\circ},W_{v}^{\circ};\boldsymbol{\kappa},\Phi_{v,l}\right)$ over $v\in S_{\varphi,\tau,\Phi}.$ Then 
\begin{equation}\label{172'}
\big|\Psi_{S_{\varphi,\tau,\Phi}}\left(s',W_{v}^{\circ},W_{v}^{\circ};\boldsymbol{\kappa},\Phi_{v,l}\right)\big|\geq C_{\varphi,\tau,\Phi}^{\circ}(s')>0.
\end{equation}

For each $v\in S_{\varphi,\tau,\Phi},$ let $\varphi_v^{\circ}$ be a fundamental idempotent with respect to a small compact subgroup such that $\phi_v^{\circ}$ is right $\supp \varphi_v^{\circ}$-invariant. Then $\mathcal{I}_{P,v}(\boldsymbol{\kappa},\varphi_v^{\circ})\phi_{v}^{\circ}=\phi_{v}^{\circ}.$ Hence we get  $\langle\mathcal{I}_{P,v}(\boldsymbol{\kappa},\varphi_v^{\circ})\phi_{v}^{\circ},\phi_{v}^{\circ}\rangle_v=\langle\phi_{v}^{\circ},\phi_{v}^{\circ}\rangle_v=1.$ Therefore, by \eqref{172'},
\begin{equation}\label{173}
\prod_{v\in S_{\varphi,\tau,\Phi}}\big|\langle\mathcal{I}_{P,v}(\boldsymbol{\kappa},\varphi_v^{\circ})\phi_{v}^{\circ},\phi_{v}^{\circ}\rangle_v\Psi_v\left(s',W_{v}^{\circ},W_{v}^{\circ};\boldsymbol{\kappa},\Phi_{v,l}\right)\big|\geq C_{\varphi,\tau,\Phi}^{\circ}(s')>0.
\end{equation}

When $v$ is a nonarchimedean place and $v\notin S_{\varphi,\tau,\Phi},$ then each $\pi_v$ is unramified and $\Phi_v$ is the characteristic function of $G(\mathcal{O}_{F_v}).$ Then by Proposition 43 in loc. cit., when $\Re(s')>1,$ the local Rankin-Selberg integral $\Psi_v\left(s',W_{1,v},W_{2,v};\boldsymbol{\kappa},\Phi_{v}\right)$ is equal to the product $\prod_{k=1}^rL_v(s',\sigma_{k,v}\otimes\tau_{v}\times\widetilde{\sigma}_{k,v})$ multiplying 
\begin{align*}
\prod_{j=1}^{r-1}\prod_{i=1}^j\frac{L_v(s'+\kappa_{i,j},\sigma_{i,v}\otimes\tau_v\times\widetilde{\sigma}_{j+1,,v})L_v(s'-\kappa_{i,j},\sigma_{j+1,,v}\otimes\tau_v\times\widetilde{\sigma}_{i,v})}{L_v(1+\kappa_{i,j},\sigma_{i,v}\times\widetilde{\sigma}_{j+1,v})L_v(1-\kappa_{i,j},\sigma_{j+1,v}\times\widetilde{\sigma}_{i,v})}.
\end{align*}

Let $\Re(s')>3$ and $\boldsymbol{\kappa}\in i\mathfrak{a}_P^*/i\mathfrak{a}_G^*.$ Denote by the partial Rankin-Selberg integral  $\Psi^{S}\left(s',W_{1},W_{2};\boldsymbol{\kappa},\Phi\right)$ the product of each local integral $\Psi_v\left(s',W_{1,v},W_{2,v};\boldsymbol{\kappa},\Phi_{v}\right),$ where $v$ is a nonarchimedean place and $v\notin S_{\varphi,\tau,\Phi}.$ Similarly we define the partial L-function $L^S(s',\sigma\times\sigma').$ Then for any cuspidal representations $\sigma$ (resp. $\sigma'$) of $\GL(n_1,\mathbb{A}_F)$ (resp. $\GL(n_1',\mathbb{A}_F)$), we have 
\begin{align*}
\big|L^S(s',\sigma\otimes\tau\times\sigma')\big|&=\prod_{\mathfrak{p}\notin S_{\varphi,\tau,\Phi}}\prod_{k=1}^{n_1}\prod_{l=1}^{n_1'}\Big|1-St_{\sigma\otimes\tau\times \sigma',k,l}(\mathfrak{p})N_{F/\mathbb{Q}}(\mathfrak{p})^{-s'}\Big|^{-1}\\
&\geq \prod_{\mathfrak{p}\notin S_{\varphi,\tau,\Phi}}\prod_{k=1}^{n_1}\prod_{l=1}^{n_1'}\Big|1+N_{F/\mathbb{Q}}(\mathfrak{p})^{-\Re(s')+1-(1+n_1^2)^{-1}+(1+n_1'^2)^{-1}}\Big|^{-1}\\
&\geq \prod_{p}\Big|1+p^{-2d}\Big|^{-n_1-n_1'}\geq e^{-(n_1+n_1')\sum_{p}p^{-2d}}\geq e^{-2(n_1+n_1')},
\end{align*}
where $d=[F:\mathbb{Q}].$ Therefore, for any $\phi_1, \phi_2\in \mathfrak{B}_{P,\chi},$ we have that 
\begin{equation}\label{172}
\big|\Psi^{S}_{12}(s')\big|\gg \prod_{j=1}^{r-1}\prod_{i=1}^j \big|L^S(1+\kappa_{i,j},\sigma_{i}\times\widetilde{\sigma}_{j+1})L^S(1-\kappa_{i,j},\sigma_{j+1}\times\widetilde{\sigma}_{i})\big|^{-1},
\end{equation}
where $\Psi^{S}_{12}(s')=\Psi^{S}\left(s',W_{1},W_{2};\boldsymbol{\kappa},\Phi\right)$ and the implied constant depends only on $n.$
\medskip 

Let $\Res\mathcal{F}(s)=R(s;\chi)\underset{\kappa_{j_m}=\delta_m(s)}{\Res}\cdots\underset{\kappa_{j_1}=\delta_1(s)}{\Res}\mathcal{F}(\boldsymbol{\kappa};s).$ Now combining \eqref{177}, \eqref{173} and \eqref{172} we then obtain that 
\begin{align*}
\frac{\Res\mathcal{F}(s)}{C_{\varphi,\tau,\Phi}}&\ll\sum_{\boldsymbol{j}\in J}\sum_{\phi_1'}\frac{\big|\langle\mathcal{I}_P(\boldsymbol{\kappa}^{\circ}+\boldsymbol{\kappa}_s^{\circ},\varphi)\phi_1,\phi_2\rangle\prod_{v\mid\infty}\Psi_v\left(s_0',W_{1,v},W_{2,v};\boldsymbol{\kappa}^{\circ},\Phi_{v,j_v}\right)\big|}{\prod_{j=1}^{r-1}\prod_{i=1}^j\big|L^S(1+\kappa_{i,j},\sigma_{i}\times\widetilde{\sigma}_{j+1})L^S(1-\kappa_{i,j},\sigma_{j+1}\times\widetilde{\sigma}_{i})\big|}\\
&\ll\sum_{\boldsymbol{j}\in J}\sum_{\phi_1'}\Big|\langle\mathcal{I}_P(\boldsymbol{\kappa}^{\circ}+\boldsymbol{\kappa}_s^{\circ},\varphi)\phi_1,\phi_2\rangle\Psi^{S_{\varphi,\tau,\Phi}}\left(s',W_{1},W_{2};\boldsymbol{\kappa},\Phi\right)\Big|\\
&\ll\sum_{\boldsymbol{j}\in J}\sum_{\phi_1'}\Big|\frac{\langle\mathcal{I}_P(\boldsymbol{\kappa}^{\circ}+\boldsymbol{\kappa}_s^{\circ},\varphi)\phi_1,\phi_2\rangle\Psi^{S_{\varphi,\tau,\Phi}}\left(s',W_{1},W_{2};\boldsymbol{\kappa},\Phi\right)}{C_{\varphi,\tau,\Phi}^{\circ}(s')\Psi_{S_{\varphi,\tau,\Phi}}\left(s',W_{v}^{\circ},W_{v}^{\circ};\boldsymbol{\kappa},\Phi_{v,l}\right)^{-1}}\Big|\\
&\ll\sum_{\boldsymbol{j}\in J}\sum_{\phi_1}\Big|\langle\mathcal{I}_P(\boldsymbol{\kappa}^{\circ}+\boldsymbol{\kappa}_s^{\circ},\varphi)\phi_1,\phi_2\rangle\Psi\left(s',W_{1},W_{2};\boldsymbol{\kappa},\Phi\right)\Big|,
\end{align*}
where $\phi_1'$ runs over $\mathcal{B}$ such that $\phi_{1,v}=\phi_v^{\circ},$ $v\in S_{\varphi,\tau,\Phi}.$ Now Theorem \ref{58''} follows from Corollary 41 in \cite{Yang19}.
\end{proof}

\begin{proof}[Proof of Theorem \ref{78}]
According to Theorem H in loc. cit. the function $(s-1/2)\cdot \mathcal{I}_{m,\chi}(s)\cdot\Lambda(s,\tau)^{-1}$ is holomorphic in the region $\mathcal{S}_{(1/3,\infty)}$ for each $0\leq m\leq r-1.$ Invoking the computation in the appendix of loc. cit. with Theorem G in loc. cit. and Theorem \ref{58''} we see that 
$$
\widetilde{\mathcal{Z}}_m(s)=(s-1/2)(s-1)^n\cdot\sum_P\frac1{c_P}\sum_{\chi\in\mathfrak{X}_P}\sum_{\phi\in \mathfrak{B}_{P,\chi}}{\widetilde{\mathcal{I}}_{m,\chi}(s)}\cdot{\Lambda(s,\tau)}^{-1}
$$
converges locally normally in the region $\mathcal{S}_{(1/3,\infty)}\setminus \{s:\ \Re(s)=1/2,\cdots,(n-1)/n,1\}.$ Let $s_0$ be such that $\Re(s_0)=\beta,$ where $\beta\in\{1/2,\cdots,(n-1)/n,1\}.$ Let $\epsilon>0$ be sufficiently small. Let $U_{\epsilon}(s_0)=\{s:\ |s-s_0|<\epsilon\}.$ We shall prove that $\widetilde{\mathcal{Z}}_m(s)$ converges uniformly in the region $U_{\epsilon}(s_0),$ which follows clearly from Corollary 41 in loc. cit. and the following Claim \ref{80.}.
\end{proof}
\begin{claim}\label{80.}
	Let $s\in U_{\epsilon}(s_0).$ Then $(s-1)^n\widetilde{\mathcal{I}}_{m,\chi}(s)\cdot{\Lambda(s,\tau)}^{-1}$ is bounded uniformly by a finite sum of $\big|\langle\mathcal{I}_P(\boldsymbol{\kappa},\varphi)\phi,\phi\rangle\Psi\left(s',W,W;\boldsymbol{\kappa},\Phi\right)\big|,$ where $\Re(s')=\Re(s_0)$ and $\Re(s')$ is large (depending on $s_0$), and the sum depends only on the test functions $\varphi$ and $\Phi.$ 
\end{claim}
\begin{proof}[Proof of Claim \ref{80.}]
Recall that for $\chi\in\mathfrak{X}_P$ and $\beta\in\mathbb{R},$ we set $\mathcal{R}_{\chi}(\beta):=(\beta-\mathcal{D}_{\chi}(\boldsymbol{\epsilon}))\cup(\beta-\mathcal{D}_{\chi}(\boldsymbol{\epsilon})).$ Then there are only finitely many $\chi$ such that $\mathcal{R}_{\chi}(1)\supseteq U_{\epsilon}(s_0).$ The contribution from these $\chi$'s is clearly convergent uniformly. Let $\chi$ be such that $\mathcal{R}_{\chi}(\beta)\nsupseteq U_{\epsilon}(s_0).$ Then we can divide $U_{\epsilon}(s_0)$ into three parts $U_{\epsilon}(s_0)^-\cup U_{\epsilon}(s_0)^0\cup U_{\epsilon}(s_0)^+,$ where $U_{\epsilon}(s_0)^0=U_{\epsilon}(s_0)\cap \mathcal{R}_{\chi}(\beta),$ $\mathcal{R}_{\chi}(\beta)^-=\left(\mathcal{R}_{\chi}(\beta)\setminus \mathcal{R}_{\chi}(\beta)\right)\cap \mathcal{S}_{(0,\beta)},$ and $\mathcal{R}_{\chi}(\beta)^-=\left(\mathcal{R}_{\chi}(\beta)\setminus \mathcal{R}_{\chi}(\beta)\right)\cap \mathcal{S}_{(\beta,2)}.$ 

Note that by \eqref{1} we see that  $R(s,W_1, W_2;\boldsymbol{\kappa},\phi)\Lambda(s,\pi_{\boldsymbol{\kappa}}\otimes\tau\times\widetilde{\pi}_{-\boldsymbol{\kappa}})$ is equal to $F(s,\boldsymbol{\kappa};\chi)\mathcal{G}(\boldsymbol{\kappa};s,P,\chi),$ where $\mathcal{G}(\boldsymbol{\kappa};s,P,\chi)$ is defined to be  
\begin{align*}
\prod_{k=1}^r\Lambda(s,\sigma_k\otimes\tau\times\widetilde{\sigma}_k)\prod_{j=1}^{r-1}\prod_{i=1}^j\frac{\Lambda(s+\kappa_{i,j},\sigma_i\otimes\tau\times\widetilde{\sigma}_{j+1})\Lambda(s-\kappa_{i,j},\sigma_{j+1}\otimes\tau\times\widetilde{\sigma}_{i})}{\Lambda(1+\kappa_{i,j},\sigma_{i}\times\widetilde{\sigma}_{j+1})\Lambda(1-\kappa_{i,j},\sigma_{j+1}\times\widetilde{\sigma}_{i})}.
\end{align*}
Then $\underset{\kappa_{j_m}=\delta_m(s)}{\Res}\cdots\underset{\kappa_{j_1}=\delta_1(s)}{\Res}R(s,W_1, W_2;\boldsymbol{\kappa},\phi)\Lambda(s,\pi_{\boldsymbol{\kappa}}\otimes\tau\times\widetilde{\pi}_{-\boldsymbol{\kappa}})$ is equal to the function $F(s,\boldsymbol{\kappa}_s;\chi)\underset{\kappa_{j_m}=\delta_m(s)}{\Res}\cdots\underset{\kappa_{j_1}=\delta_1(s)}{\Res}\mathcal{G}(\boldsymbol{\kappa};s,P,\chi),$ where $\boldsymbol{\kappa}_s=(\kappa_1,\cdots,\kappa_{r-1})$ with $\kappa_{j}=\delta_j(s),$ $j=j_1,\cdots,j_m.$ It follows from the proof of Theorem \ref{47'} (resp.  Theorem \ref{58''}) that $|F(s,\boldsymbol{\kappa};\chi)|\ll_{\epsilon} |F(s',\boldsymbol{\kappa};\chi)|$ (resp. $|F(s,\boldsymbol{\kappa}_s;\chi)|\ll_{\epsilon} F(s',\boldsymbol{\kappa}^{\circ};\chi)$) for $s'$ such that $\Im(s')=\Im(s_0)$ and $\Re(s')$ is large enough. 

Also, by the definition of $\widetilde{\mathcal{I}}_{m,\chi}(s)$ we see that $\underset{\kappa_{j_m}=\delta_m(s)}{\Res}\cdots\underset{\kappa_{j_1}=\delta_1(s)}{\Res}\mathcal{G}(\boldsymbol{\kappa};s,P,\chi)$ is of the form \eqref{192.} and $(\delta_1(s),\cdots,\delta_{m}(s))$ is $nice$ $with$ $respect$ $to$ $\chi\in\mathfrak{X}_P.$ Let $s\in \mathcal{R}_{\chi}(\beta)^0,$ then one sees from the explicit construction of $\widetilde{\mathcal{I}}_{m,\chi}(s)$ that $\kappa_{i,j}\notin (\beta-\mathcal{D}_{\chi}(\boldsymbol{n\epsilon}))\cup(\beta-\mathcal{D}_{\chi}(\boldsymbol{n\epsilon})).$ While $s\in \mathcal{R}_{\chi}(\beta)^-\cup \mathcal{R}_{\chi}(\beta)^+,$ $\Re(\kappa_{i,j})=0.$ In all, one has $\min\{|s\pm\kappa_{i,j}|,|s\pm\kappa_{i,j}-1|\}\gg_{\tau} C(\sigma_i\times\sigma_j;s_0)^{-N},$ and $\min\{|c(s,\kappa_l)|,|c(s,\kappa_l)-1|\}\gg_{\tau}C(\sigma_l\times\sigma_l')^{-N}$ (see \eqref{192.} for the notation here), where $N$ is a positive absolute constant coming from definition of the zero-free region (see \eqref{A} and \eqref{B}). Now apply preconvex bound to see that $|\Lambda(s\pm\kappa_{i,j},\sigma_i\otimes\tau\times\widetilde{\sigma}_{j+1})|\ll \Lambda(s'\pm\kappa_{i,j},\sigma_i\otimes\tau\times\widetilde{\sigma}_{j+1})$ and $|(s-1/2)(s-1)^n\Lambda(s,\sigma_k\otimes\tau\times\widetilde{\sigma}_k)|\leq |\Lambda(s',\sigma_k\otimes\tau\times\widetilde{\sigma}_k)|.$ One then concludes Claim \ref{80.} for $m=0$ form the proof of Theorem G in \cite{Yang19}. Likewise, one has bounds for $|\Lambda(c(s,\kappa_l),\sigma_l\otimes\tau\times\widetilde{\sigma}'_l)|\ll |\Lambda(c(s',\kappa_l),\sigma_l\otimes\tau\times\widetilde{\sigma}'_l)|$ and  $|(s-1/2)(s-2/3)(s-3/4)(s-1)^n\Lambda(c_{l'}(s),\sigma_{l'}\otimes\tau\times\widetilde{\sigma}'_{l'})|\ll|\Lambda(c_{l'}(s'),\sigma_{l'}\otimes\tau\times\widetilde{\sigma}'_{l'})|.$ Then the $m\geq 1$ part of Claim \ref{80.} follows from the proof of Theorem \ref{58''} and the fact that $(s-1/2)(s-1)^n\widetilde{\mathcal{I}}_{m,\chi}(s)\cdot{\Lambda(s,\tau)}^{-1}$ is holomorphic at $s\in \{2/3,\cdots,(n-1)/n\}.$
\end{proof}

\section{Proof of Main Theorems}\label{8.}
\begin{prop}\label{76lem}
Let $n\geq 1$ be an integer. Let $\pi$ be an cuspidal representation of $\GL(n,\mathbb{A}_F)$ and $\tau$ be a Hecke character on $F^{\times}\backslash\mathbb{A}_F^{\times},$ where $F$ is a number field. Assume $\tau$ has order at most 2. Then the root number of $\Lambda_F(s,\pi,Ad\otimes\tau)$ is 1.
\end{prop}
\begin{proof}
Denote by $W(\pi,\Ad\otimes\tau)=\prod_{v\in\Sigma_F}W(\pi_v,\Ad\otimes\tau_v)$ the root number associated with $\Lambda_F(s,\pi,\Ad\otimes\tau),$ where $W(\pi_v,\Ad\otimes\tau_v)$ are local root number in the functional equation of $L(s,\pi_v,\Ad\otimes\tau_v).$ First we deal with the case where $\tau$ is trivial. The general case will be reduced to this special case by base change. Write $W(\pi,\Ad)=\prod_{v\in\Sigma_F}W(\pi_v,\Ad).$ According to \cite{BH99}, for any $v\in\Sigma_{F,fin}$ and any irreducible admissible representation $\pi_v$ of $\GL(n,F_v),$ one has that $W(\pi_v,\Ad)=w_{\pi_v}(-1)^{n-1},$ where $w_{\pi_v}$ is the central character of $\pi_v.$ 

Hence, we need to compute archimedean root numbers $W(\pi_v,\Ad).$ Since our approach is using Langlands classification (see \cite{Kna94}), we will separate the cases where the place $v$ is archimedean or finite.  
\begin{itemize}
\item[Case 1:] Assume that $F_v\simeq \mathbb{C}.$ One has $W_{F_v}\simeq\mathbb{C}^{\times}.$ So all irreducible representations are one dimensional. We may write any such characters as $\chi_{k,\nu}(z)=(z/|z|)^k|z|_{\mathbb{C}}^{\nu}=(z/|z|)^k|z|^{2\nu},$ for $k\in \mathbb{Z}$ and $\nu\in\mathbb{C}.$ The root number associated to this character is $W(\chi_{k,\nu})=i^{|k|}.$ Since $\chi_{k,\nu}\otimes \chi_{k',\nu'}=\chi_{k+k',\nu+\nu'},$ we then have $W(\chi_{k+k',\nu+\nu'})=i^{|k+k'|}.$ Let $\oplus_{j=1}^n\chi_{k_j,\nu_j}$ be the representation corresponding to $\pi_v.$ Then $\Ad\pi_v$ corresponds to 
\begin{align*}
\Ad(\oplus_{j=1}^n\chi_{k_j,\nu_j})=(n-1)\boldsymbol{1}\bigoplus\oplus_{l=1}^{n-1}\oplus_{j=l+1}^n \chi_{k_l,\nu_l}\otimes{\chi}_{k_j,\nu_j}^{-1}\oplus \chi_{k_j,\nu_j}\otimes{\chi}_{k_l,\nu_l}^{-1},
\end{align*}
where $\boldsymbol{1}$ is the trivial representation of $W_{F_v}.$ Therefore, we have 
\begin{align*}
W(\pi_v,\Ad)&=W(\boldsymbol{1})^{n-1}\prod_{l=1}^{n-1}\prod_{j=l+1}^nW(\chi_{k_l-k_j,\nu_l-\nu_j})W(\chi_{k_j-k_l,\nu_j-\nu_l})\\
&=\prod_{l=1}^{n-1}\prod_{j=l+1}^n(-1)^{|k_l-k_j|}=\prod_{l=1}^{n-1}\prod_{j=l+1}^n(-1)^{k_l+k_j}.
\end{align*}
By comparing multiplicity of each $k_j$ one concludes that $\sum_{1\leq l<j\leq n}(k_l-k_j)\cong (n-1)\sum_{j=1}^nk_j\mod2.$ Consequently, we have
\begin{equation}\label{299}
W(\pi_v,\Ad)=\prod_{j=1}^n(-1)^{(n-1)k_j}=w_{\pi_v}(-1)^{n-1}.
\end{equation}
\item[Case 2:] Assume that $F_v\simeq \mathbb{R}.$ One has $W_{F_v}\simeq\mathbb{C}^{\times}\sqcup\boldsymbol{j}\mathbb{C}^{\times},$ where $\boldsymbol{j}^2=-1$ and $\boldsymbol{j}z\boldsymbol{j}^{-1}=\bar{z}$ for any $z\in\mathbb{C}^{\times}.$ Hence each irreducible representation $\sigma$ of $W_{F_v}$ is of dimension 1 or 2. If $\dim\sigma=1,$ then its restriction to $\mathbb{C}^{\times}$ is of the form $\chi_{0,\nu}$ for some $\nu\in\mathbb{C}$ (see (3.2) of \cite{Kna94}). If $\sigma(\boldsymbol{j})=1,$ then $W(\sigma),$ the root number associated to $\sigma,$ is trivial. If $\sigma(\boldsymbol{j})=-1,$ then $W(\sigma)=i.$ If $\dim\sigma=2,$ then it is the induction of $\chi_{k,\nu}$ from $\mathbb{C}^{\times}$ to $\GL(2,\mathbb{R}),$ where $k\in\mathbb{N}_{\geq 1}$ and $\nu\in\mathbb{C}.$ In this case, the root number $W(\sigma)=i^k.$ Let $\sigma_1$ and $\sigma_2$ be two irreducible representations of $W_{F_v}.$ We shall examine the tensor product parameters $\sigma_1\otimes\widetilde{\sigma}_2.$
\begin{itemize}
\item[(a)] If $\dim\sigma_1=\dim\sigma_2=1,$ then so is $\sigma_1\otimes\widetilde{\sigma}_2.$ Let $\sigma_1=\chi_{0,\nu_1}$ and $\sigma_2=\chi_{0,\nu_2}.$ Then $\sigma_1\otimes\widetilde{\sigma}_1=\chi_{1-\sigma_1(\boldsymbol{j})\sigma_1(\boldsymbol{j}),\nu_1+\nu_1}=\chi_{0,2\nu_1}.$ Thus one has the formula  $W(\sigma_1\otimes\widetilde{\sigma}_1)=1,$ $W(\sigma_1\otimes\widetilde{\sigma}_2)=i^{{1-\sigma_1(\boldsymbol{j})\sigma_2(\boldsymbol{j})}},$ and 
\begin{equation}\label{300.}
W(\sigma_1\otimes\widetilde{\sigma}_2)W(\sigma_2\otimes\widetilde{\sigma}_1)=i^{{1-\sigma_1(\boldsymbol{j})\sigma_2(\boldsymbol{j})}}i^{{1-\sigma_2(\boldsymbol{j})\sigma_1(\boldsymbol{j})}}=(-1)^{{1-\sigma_1(\boldsymbol{j})\sigma_2(\boldsymbol{j})}}=1.
\end{equation}
\item[(b)] If $\dim\sigma_1\cdot \dim\sigma_2=2,$ then $\sigma_1\otimes\widetilde{\sigma}_2$ is irreducible and two dimensional. Let $\sigma_1=\chi_{0,\nu_1}$ and $\sigma_2$ be induced from $\mathbb{C}^{\times}$ by $\chi_{k_2,\nu_2},$ where $k_2\in\mathbb{N}.$ Then $\sigma_1\otimes\widetilde{\sigma}_2$ is induced from $\mathbb{C}^{\times}$ by $\chi_{k_2,\nu_1+\nu_2}.$ Thus $W(\sigma_1\otimes\widetilde{\sigma}_2)=i^{k_2}$ and 
\begin{equation}\label{301}
W(\sigma_1\otimes\widetilde{\sigma}_2)W(\sigma_2\otimes\widetilde{\sigma}_1)=i^{k_2}i^{k_2}=(-1)^{k_2}.
\end{equation}
\item[(c)] If $\dim\sigma_1=\dim\sigma_2=2,$ then we may assume that $\sigma_1$ is induced from $\mathbb{C}^{\times}$ by $\chi_{k_1,\nu_1}$ and $\sigma_2$ is induced from $\mathbb{C}^{\times}$ by $\chi_{k_2,\nu_2}.$ Then $\sigma_1\otimes\widetilde{\sigma}_2$ is the direct sum of two two-dimensional representations, induced from $\mathbb{C}^{\times}$ from the characters $\chi_{k_1,\nu_1}\chi_{-k_2,-\nu_2}=\chi_{k_1-k_2,\nu_1-\mu_2}$ and $\chi_{-k_1,-\nu_1}\chi_{-k_2,-\nu_2}=\chi_{-k_1-k_2,-\nu_1-\mu_2}.$ Note that the former representation id reducible when $k_1=k_2.$ It then follows that $W(\sigma_1\otimes\widetilde{\sigma}_1)=i^{2|k_1|}=(-1)^{k_1},$ and 
\begin{equation}\label{302}
W(\sigma_1\otimes\widetilde{\sigma}_2)W(\sigma_2\otimes\widetilde{\sigma}_1)=i^{2|k_1-k_2|+2|k_1+k_2|}=1.
\end{equation}
\end{itemize}
Let $\oplus_{j=1}^r\oplus_{j=1}^r\sigma_j$ be the representation corresponding to $\pi_v,$ where $\dim\sigma\in\{1,2\}$ and $\sum_{j=1}^r\dim\sigma_j=n.$ Assume further that $\dim\sigma_k=1,$ for $1\leq k\leq r_0\leq r;$ and $\dim\sigma_k=2,$ for $r_0<k\leq r.$ For $1\leq k\leq r_0,$ write $\sigma_k=\chi_{w_k,\nu_k};$ for $r_0<k\leq r,$ we may assume $\sigma_k$ is induced from $\mathbb{C}^{\times}$ by $\chi_{w_k,\nu_k},$ where $w_k\geq 0.$ Then $\Ad\pi_v\boxplus\boldsymbol{1}$ corresponds to 
\begin{align*}
\qquad\Ad(\oplus_{j=1}^r\sigma_j)\oplus \boldsymbol{1}=\oplus_k^{r}\sigma_k\otimes\widetilde{\sigma}_k\bigoplus\oplus_{l=1}^{r-1}\oplus_{j=l+1}^r \sigma_l\otimes\widetilde{\sigma}_j\oplus \sigma_j\otimes\widetilde{\sigma}_l,
\end{align*}
where $\boldsymbol{1}$ is the trivial representation of $W_{F_v}.$ Therefore, we have 
\begin{align*}
W(\pi_v,\Ad)&=\prod_{k=1}^rW(\sigma_k\otimes\widetilde{\sigma}_k)\prod_{l=1}^{r-1}\prod_{j=l+1}^rW(\sigma_l\otimes\widetilde{\sigma}_j)W(\sigma_j\otimes\widetilde{\sigma}_l)\\
&=\prod_{k=r_0+1}^r(-1)^{w_k}\prod_{l=1}^{r_0}\prod_{j=r_0+1}^{r}(-1)^{w_j}=(-1)^{(r_0+1)\sum_{k=r_0+1}^rw_k}.
\end{align*}
Now applying the relation $r_0+2(r-r_0)=n$ one deduces easily that $(-1)^{(r_0+1)\sum_{k=r_0+1}^rw_k}=(-1)^{(n-1)\sum_{k=r_0+1}^rw_k}=w_{\pi_v}(-1)^{n-1}.$ So we have
\begin{equation}\label{300}
W(\pi_v,\Ad)=w_{\pi_v}(-1)^{n-1},\ \text{$\forall$ $v\mid\infty$ such that $F_v\simeq \mathbb{R}.$}
\end{equation}
\end{itemize}  
Then combining \eqref{299}, \eqref{300} with results from \cite{BH99} we conclude that 
\begin{equation}\label{0}
W(\pi,\Ad)=1,\ \forall\ F,\ \forall\ \pi\in\mathcal{A}_0(\GL(n,\mathbb{A}_F)).
\end{equation}

Let $v$ be a place of $F$. Let $\sigma_v$ be an $n$-dimensional representation of $W_{F_v}\times \GL(2,\mathbb{C})$ (resp. $W_{F_v}$) for $v$ nonarchimedean (resp. archimedean) associated to $\pi_v$ via local Langlands correspondence (see \cite{Hen00} and \cite{HT01}). Let $\Ad\sigma_v$ be the adjoint representation of $\sigma_v.$ Then $\dim \Ad\sigma_v=n^2-1.$

Now assume $\tau$ is nontrivial. If $\pi\otimes\tau\simeq\pi,$ then from previous result, we have
\begin{align*}
W(\pi,\Ad\otimes\tau)=\frac{W(\pi\times\widetilde{\pi})\otimes\tau}{W(\tau)}=\frac{W(\pi\times\widetilde{\pi})}{W(\tau)}=1,
\end{align*}
as $W(\tau)=1.$ Since $\tau$ is quadratic, then there exists some quadratic extension $K/F$ such that $\tau$ is the character associated to this extension. Let $\pi^*$ be the base change of $\pi$ with respect to $K/F.$ By proceeding analysis we may assume that $\pi\otimes\tau\ncong \pi.$ Then $\pi^*$ is cuspidal. Let $\theta=\otimes_v\theta_v$ be a nontrivial additive character on $F\backslash\mathbb{A}_F.$ Write $\tau=\otimes_v\tau_v.$ Let $v\in\Sigma_F$ and $\mathfrak{p}=\mathfrak{p}_v$ be a place of $K$ above $v,$ then $K_{\mathfrak{p}}$ is a quadratic extension of $F_v$ and $\tau_v$ is the character associated to this extension. Let $\sigma_v^*=\Res_{K_\mathfrak{p}/F_v}\sigma_v.$ Then one has (see \cite{Tat79}) that 
\begin{align*}
\epsilon(\Ind_{W_{K_\mathfrak{p}}}^{W_{F_v}}(\Ad\sigma_v^*\ominus(n^2-1)\boldsymbol{1}_{K_\mathfrak{p}}),\theta_{F_v})=\epsilon(\Ad\sigma_v^*\ominus(n^2-1)\boldsymbol{1}_{K_\mathfrak{p}},\theta_{F_v}\circ\Tr_{K_\mathfrak{p}/F_v}).
\end{align*}
Hence $\epsilon(\Ad\sigma_v,\theta_{F_v})\epsilon(\Ad\sigma_v\otimes\tau_v,\theta_{F_v})=\epsilon(\tau_v,\theta_{F_v})^{n^2-1}\epsilon(\Ad\sigma_v^*,\theta_{F_v}\circ\Tr_{K_\mathfrak{p}/F_v}),$ implying that $\epsilon(\Ad\sigma_v\otimes\tau_v,\theta_{F_v})=\epsilon(\tau_v,\theta_{F_v})^{n^2-1}\epsilon(\Ad\sigma_v^*,\theta_{F_v}\circ\Tr_{K_v/F_v})\epsilon(\Ad\sigma_v,\theta_{F_v})^{-1}.$ Therefore, via local Langlands correspondence we have 
\begin{align*}
\epsilon(\Ad\pi_v\otimes\tau_v,\theta_{F_v})=\epsilon(\tau_v,\theta_{F_v})^{n^2-1}\epsilon(\Ad\pi_v^*,\theta_{F_v}\circ\Tr_{K_\mathfrak{p}/F_v})\epsilon(\Ad\pi_v,\theta_{F_v})^{-1}.
\end{align*}
Then we have GLobally that  $W(\pi,\Ad\otimes\tau)=W(\tau)^{n^2-1}W(\pi^*,\Ad)W(\pi,\Ad).$ Since $\tau$ is quadratic, it is of orthogonal type. Thus by \cite{Del76}, $W(\tau)=1.$ Therefore we have $W(\pi,\Ad\otimes\tau)=W(\pi^*,\Ad)W(\pi,\Ad).$ Then Proposition \ref{76lem} will follow from \eqref{0}.
\end{proof}

\begin{proof}[Proof of Theorem \ref{main1}]
Recall that we have shown, for any test function $\varphi\in\mathcal{F}(w),$ 
\begin{align*}
I_0(s,\tau)=\int_{G(F)Z_G(\mathbb{A}_F)\setminus G(\mathbb{A}_F)}\K_0(x,x)E(x,\Phi,\tau;s)dx=I_{\Geo,\Reg}(s,\tau)+I_{\infty}(s,\tau),
\end{align*}	
where $I_{\infty}(s,\tau)=I_{\infty,\Reg}(s,\tau)+I_{\infty}^{(1)}(s,\tau)+I_{\Sin}(s,\tau).$

Since $n\leq 4,$ then according to Uchida-Van der Waal Theorem (see \cite{Uch75} and \cite{Waa75}) and its generalization to twist form (see \cite{MR00}), each $\Lambda_E\left(s,\tau\circ N_{E/F}\right)\cdot \Lambda_F(s,\tau)^{-1}$ admits a holomorphic continuation to the whole complex plane. It then follows from Theorem D in \cite{Yang19} that the function $I_{\Geo,\Reg}(s,\tau)/\Lambda_F(s,\tau)$ admits an entire continuation. 

Also, by Theorem \ref{Y}, Theorem \ref{78} and Theorem E in loc. cit., the function $I_{\infty}(s,\tau)/ \Lambda_F(s,\tau)$ admits a meromorphic continuation to $\mathcal{R}(1/2;\tau)^-\cup\mathcal{S}_{(1/2,\infty)},$ with possible simple poles at $s\in\{1/2,2/3,3/4\}.$ Moreover, if $L_F(2/3,\tau)=0,$ then $I_{\infty}(s,\tau)\cdot\Lambda_F(s,\tau)^{-1}$ is regular at $s=2/3;$  if $L_F(3/4,\tau)=0,$ then $I_{\infty}(s,\tau)\cdot\Lambda_F(s,\tau)^{-1}$ is regular at $s=3/4.$
\medskip 

Let $\rho\in\mathcal{R}(1/2;\tau)^-\cup\mathcal{S}_{(1/2,1)}$ be a zero of $\Lambda(s,\tau)$ of order $r_{\rho}\geq1.$ Denote by 
\begin{align*}
J(\rho;j)=\int_{G(F)Z(\mathbb{A}_F)\backslash  G(\mathbb{A}_F)}\K_0(x,x)\frac{\partial^{j}}{\partial s^{j}}E(x,\Phi,\tau;s)\mid_{s=\rho}dx,\ 0\leq j\leq r_{\rho}-1.
\end{align*}
If $\rho\neq1/2,$ we then see that $J(\rho;j)=0$ for any $0\leq j\leq r_{\rho}-1$ and $\varphi\in\mathcal{F}(w)$. According to the Proposition in Section 3.3 of \cite{JZ87}, one has, for all cuspidal representations $\pi\in \mathcal{A}_0\left(G(F)\setminus G(\mathbb{A}_F),w^{-1}\right),$ and all $K$-finite functions $\phi_1, \phi_2\in\V_{\pi},$ that 
\begin{align*}
\int_{G(F)Z(\mathbb{A}_F)\backslash  G(\mathbb{A}_F)}\phi_1(x)\overline{\phi_2(x)}\frac{\partial^{j}}{\partial s^{j}}E(x,\Phi,\tau;s)\mid_{s=\rho}dx=0.
\end{align*}
Then by Rankin-Selberg theory, we have, for all cuspidal representations $\pi\in \mathcal{A}_0\left(G(F)\setminus G(\mathbb{A}_F),w^{-1}\right),$ that $\frac{\partial^{j}}{\partial s^{j}}\Lambda(s,\pi\otimes\tau\times\widetilde{\pi})\mid_{s=\rho}=0,$ $1\leq j<r_{\rho},$ implying that the adjoint $L$-function  $\Lambda(s,\pi,\Ad\otimes\tau)$ is regular at $s=\rho.$
\medskip 

Now assume that $\rho=1/2,$ namely, $L_F(1/2,\tau)=0.$ If $\tau$ is not quadratic, then by Theorem \ref{78}, $I_{\infty}^{(1)}(s,\tau)\cdot \Lambda_F(s,\tau)^{-1}$ is regular at $s=1/2.$ Therefore, we have $J(1/2;j)=0,$ for $1\leq j\leq r_{1/2}-1.$ Hence, by similar analysis as above we see that $\frac{\partial^{j}}{\partial s^{j}}\Lambda(s,\pi\otimes\tau\times\widetilde{\pi})\mid_{s=1/2}=0,$ $1\leq j\leq r_{1/2}-1,$ implying that the adjoint $L$-function $\Lambda(s,\pi,\Ad)$ is regular at $s=1/2.$ Now we assume that $\tau^2=1.$ If $r_{1/2}\geq 2,$ then by Theorem \ref{Y}, Theorem \ref{78} and Theorem A in \cite{Yang19}, we see that $J(1/2;j)=0,$ for $1\leq j\leq r_{1/2}-2.$ Hence, by the Proposition in Section 3.3 of \cite{JZ87} we see that $\frac{\partial^{j}}{\partial s^{j}}\Lambda(s,\pi\otimes\tau\times\widetilde{\pi})\mid_{s=1/2}=0,$ $1\leq j<r_{1/2}-1,$ implying that the adjoint $L$-function $\Lambda(s,\pi,\Ad)$ has at most a simple pole at $s=1/2.$ Now we apply Proposition \ref{76lem} to exclude this possible simple pole at $1/2.$ Suppose that $\Lambda(s,\pi,\Ad\otimes\tau)$ has a pole at $s=1/2.$ Since the root number of $\Lambda(s,\pi,\Ad\otimes\tau)$ is trivial, then the order of the pole $s=1/2$ must be even. So $\Lambda(s,\pi,\Ad\otimes\tau)$ cannot have a simple pole at $s=1/2.$ A contradiction. If $r_{1/2}=1,$ then clearly, the adjoint $L$-function  $\Lambda(s,\pi,\Ad)$ has at most a simple pole at $s=1/2.$ The same argument on root number excludes the possibility of pole at $s=1/2.$ 

In all, we have shown that $\Lambda(s,\pi,\Ad\otimes\tau)$ is holomorphic in $\mathcal{R}(1/2;\tau)^-\cup\mathcal{S}_{(1/2,\infty)}.$ Now Theorem \ref{main1} follows from GLobal functional equation of $\Lambda(s,\pi,\Ad\otimes\tau).$
\end{proof}

\begin{proof}[Proof of Corollary \ref{2cor}]
It follows from local Langlands correspondence that the local factor $L_{\infty}(s,\pi_{\infty}\otimes\tau_{\infty}\times\widetilde{\pi}_{\infty})\cdot L_{\infty}(s,\tau_{\infty})^{-1}$ is equal to a finite product of exponential functions and Gamma functions. Therefore,  $L_{\infty}(s,\tau_{\infty})\cdot L_{\infty}(s,\pi_{\infty}\otimes\tau_{\infty}\times\widetilde{\pi}_{\infty})^{-1}$ admits a holomorphic continuation to the whole complex plane. Therefore, by Theorem \ref{main1},  we conclude that 
\begin{align*}
L(s,\pi,Ad\otimes\tau)=\Lambda(s,\pi,\Ad\otimes\tau)\cdot\frac{L_{\infty}(s,\tau_{\infty})}{L_{\infty}(s,\pi_{\infty}\otimes\tau_{\infty}\times\widetilde{\pi}_{\infty})}
\end{align*}	
admits a holomorphic continuation to the whole complex plane. 
\end{proof}

\bibliographystyle{alpha}

\bibliography{Ad}
\end{document}